\documentclass[12pt,twoside,a4paper]{article}

\usepackage{amsmath}
\usepackage{a4wide}
\usepackage{graphicx}
\usepackage{dsfont}
\usepackage{amsthm}
\usepackage{amssymb}
\usepackage{rotating}

\newcommand{\rem}[1]{}
\newcommand{\g}{{\text{R}}}
\newcommand{\rot}{{\text{R}}}
\newcommand{\angvel}{{\Omega}}

\newcommand{\Q}{{ \mathcal{X}  }}
\newcommand{\Qt}{\mathcal{J}}

\newcommand{\Qti}[1]{{ \mathcal{J}_{ {#1}}  }}

\newcommand{\Qtr}{\mathcal{Q}}
\newcommand{\qtr}{{ q  }}

\newcommand{\qtri}[1]{{ q_{ {#1}}  }}

\newcommand{\config}{{ \mathcal{X} }}
\newcommand{\SE}{\text{SE}}
\newcommand{\SO}{\text{SO}}

\newcommand{\ue}{\mathrm{e}}

\newcommand{\R}{\mathds{R}}

\newcommand{\M}{\text{M}}

\providecommand{\norm}[1]{\lVert#1\rVert}

\newcommand{\br}{\mathbf{r}}
\newcommand{\bs}{\mathbf{s}}

\newcommand{\be}{\boldsymbol{e}}

\newtheorem{theorem}{Theorem}[section]

\newtheorem{definition}{Definition}[section]
\newtheorem{remark}{Remark}[section]
\newtheorem{lemma}{Lemma}[section]
\newtheorem{corollary}{Corollary}[section]

\title{Hill regions of charged three-body systems}
\author{Igor Hoveijn, 
 Holger Waalkens and Mohammad Zaman\\[2ex]
Bernoulli Institute for Mathematics, Computer Science and Artificial Intelligence\\ University of Groningen\\ PO Box 407\\ 
9700 AK Groningen, The Netherlands}

\begin{document}
\maketitle

\abstract{
For charged three-body systems, we discuss  the configurations and orientations that are admissible  for given values of the conserved total energy and angular momentum. 
The admissible configurations and orientations are discussed on a configuration space that is reduced by the translational, rotational  and dilation symmetries of charged three-body systems.  We consider the examples of the  charged three-body systems given by the helium atom (two electrons and a nucleus) and the compound of two electrons and one positron.  For comparison, the well known example of the Newtonian gravitational three-body system is discussed for the scheme presented in this paper first.
}

%\date{\today}

\vspace*{1cm}

%\noindent
%PACS numbers:
%45.50 Jf, % Few and many-body systems
%45.20.Jj	%Lagrangian and Hamiltonian mechanics
%02.40.-k \\
\noindent
AMS classification numbers: ...
\noindent

%%%%%%%%%%%%%%%%%%%%%%%
\section{Introduction}

%$X$, ${\mathcal C}$, ${\mathcal J}$ , ${\mathcal X}$, ${\mathcal Q}$, ${\mathcal G}$ , $Q_{\text{\sf trd},1}$ $q$, $q^{\text{\sf trd}}$ $\Qtrdi{1}$ $\Qtrd$

\rem{
A system is called closed if  all forces are 
internal, i.e. they only depend on the mutual distances between the bodies and are directed parallel to the mutual difference vectors of positions. 
Such systems are conservative \cite{Arnold78}. If  $\mathbf{x}_{i}$, $i=1,\ldots, N$, denote the position vectors of the $N$ particles   of masses $m_i$, $i=1,\ldots,N$,  in $\R^3$ and $\mathbf{p}_i$ are the conjugate momenta, then the Hamiltonian is of the form
\begin{equation}
H(\mathbf{x},\mathbf{p}) = \sum_{i=1}^N \frac12 m_i \mathbf{p}_i \cdot \mathbf{p}_i + V(\mathbf{x}_1,\ldots,\mathbf{x}_N) = \frac12  \mathbf{p} \,\M^{-1} \,\mathbf{p} +V(\mathbf{x}) ,
\end{equation}
where $\mathbf{x}=(\mathbf{x}_1,\ldots, \mathbf{x}_N)$ and $\mathbf{p} = (\mathbf{p}_1,\ldots,\mathbf{p}_N)  $ are considered as vectors in $\R^{3N}$ and $\M$ is the $3N\times 3N$ diagonal matrix with diagonal $(m_1,m_1,m_1,m_2,m_2,m_2,\ldots, m_N,m_N,m_N)$. 
The phase space of an $N$-body system is the cotangent bundle $T^*\Q$ over the configuration space  $\Q=\R^{3N}$.
 } % end rem
 
 This is the third in a series of three papers where we study charged three-body systems. 
Following \cite{HoveijnWaalkensZaman2019} we define $N$-body systems as so-called natural Hamiltonian systems
 formed by  $N$ point masses with position vectors $\mathbf{x}_i\in \R^3$  and masses $m_i>0$, $i=1,\ldots,N$, which interact via a potential $V$ that is invariant under the action of the 
  special Euclidean group $\SE(3)=\SO(3) \rtimes \R^3$,
  %which is the semi-direct product of the group of rotations $\SO(3)$ and the group of translations $\R^3$, 
  i.e.  $V(\rot\, (\mathbf{x}_1+a),\ldots,\rot \,(\mathbf{x}_N+a))=V(\mathbf{x}_1,\ldots,\mathbf{x}_N)$ for all $\rot \in \SO(3)$ and $a\in \R^3$.

 More precisely, we define $N$-body systems and their subclass of charged $N$-body systems as follows.
 \begin{definition}\label{def:NBodySystem}
For $M = T^*\config$ where $\config = \R^{3N}-\Delta_c$ with $\Delta_c$ being a closed, possibly empty subset of $\R^{3N}$ (the collision set), we call the Hamiltonian system $(M,\omega,H)$, with $\omega$ denoting the standard symplectic form, an $N$-body system with masses $m_i>0$ of the $i$-th body, $i=1,\ldots,N$, if $H$ is of the form kinetic plus potential energy
$$
H = T + V,
$$
where the kinetic energy is $T(p) = \sum_i \frac{1}{2m_i} p_i^2$ and the potential $V$ is $\SE(3)$ invariant. We will refer to the system that has potential 
\begin{equation}\label{eq:potential}
V(\mathbf{x}) = \sum_{i<j} -\frac{\gamma_{ij}}{r_{ij}(\mathbf{x})}
\end{equation}
with $r_{ij} = \norm{\mathbf{x}_i-\mathbf{x}_j}$ and $\gamma_{ij}\in\R, \quad 1\le i < j \le N,$ as the \emph{charged $N$-body system}. In this case the collision set is given by 
\begin{equation}
\Delta_c  = \{ \mathbf{x}=(\mathbf{x}_1,\ldots, \mathbf{x}_N) \in \R^{3N} \,: \, \mathbf{x}_i = \mathbf{x}_j \text{ for some } 1\le i < j \le N \}\,.
\end{equation}
\end{definition}
We note that this definition of charged $N$-body systems covers besides Coulomb interactions also gravitational interactions and mixtures of these. Depending on the context we will often say  \emph{particle}  instead of \emph{body}.

 In the first paper  we studied the map of integrals, i.e. the map from the phase space $M = T^*\config$ to the conserved quantities resulting from the Euclidian symmetries, and the critical points of this map \cite{HoveijnWaalkensZaman2019}. We in particular showed that as opposed to the gravitational case,  the more general charged  $N$-body systems can have relative equilibria that are not related to central configurations. We will see an example of such a relative equilibrium 
 in the present paper. The topology of the integral manifolds can change at critical points \cite{mccord1998integral}.  As the integral manifolds can also be non-compact they can also change topology at critical points at infinity \cite{mccord1998integral}.
We studied such critical points at infinity for charged three body systems in the second paper \cite{CPI2021}. 
 
 The present paper concerns the Hill regions of charged $N$-body systems systems. 
 The notion Hill region alludes to the work of Hill on the accessible region in configuration space of the circular restricted three-body problem for a given value of the energy (or Jacobi constant). For unrestricted  $N$-body systems, one first needs to define what is meant by Hill region. To this end it seems natural to first reduce the Euclidian symmetries of $N$-body systems. 
%Before we can say what we mean by a Hill region for charged three-body systems we recall the reduction of the translational and rotational 
%symmetries  of  general  $N$-body system. 
We will see that the  phase space of  the reduced system no longer has the structure of a cotangent bundle over a configuration space, and hence the definition of the Hill region as the projection of an energy surface to configuration space like in the restricted case
is at first not possible. Reducing in addition to the symplectic Euclidean symmetries the non-symplectic dilation (scaling) symmetry of charged three-body systems we will define their Hill regions in terms of the admissible \emph{shapes} and \emph{orientations} for given values of the integrals.
Here shapes will refer to equivalence classes of configurations modulo similarity transformations and orientations will refer to the directions of the conserved angular momentum in a body fixed frame of the three-body system.

\rem{ 
For charged $N$-body systems, the potential $V$ is singular at configurations where the position vectors of two or more particles coincide. Taking out these collisions the configuration space becomes $\Q \backslash \Delta_c$, where 
\begin{equation}
\Delta_c  = \{  \mathbf{x}=(\mathbf{x}_1,\ldots, \mathbf{x}_N) \in \R^{3N} \,: \, \mathbf{x}_i = \mathbf{x}_j \text{ for some } 1\le i < j \le N \}
\end{equation}
is the \emph{collision set}. The phase space is then the cotangent bundle $T^*(\Q \backslash \Delta_c)$ which we identify with 
 $(\R^{3N}\backslash \Delta_c) \times \R^{3N}$.
} % end rem

This paper is organized as follows. In Sec.~\ref{charged_sec:Nbodyreduction} we represent the symplectic reduction of the Euclidian symmetries of $N$-body systems following to a large extent \cite{LittlejohnReinsch97}. 
In Sec.~\ref{charged_sec:triatomic} we make the reduction concrete for three-body systems. 
What we mean by  shapes and orientations will be defined in Sec.~\ref{charged_sec:shapespace}. 
Based on theses notions we will define  
Hill regions in Sec.~\ref{charged_sec:Hillregions}.
In Sec.~\ref{charged_sec:relative_equilibria} we discuss the critical points where the Hill regions change.  In Sec.~\ref{charged_sec:Examples} we discuss three examples of charged three-body systems which includes the well known Newtonian case. 
A short discussion and an outlook are given in Sec.~\ref{sec:conclusions}.

%%%%%%%%%%%%%%%%%%%%%%
\section{Reduction of  $N$-body systems}
\label{charged_sec:Nbodyreduction}

As a first step towards defining what we mean by Hill region we reduce the $\SE(3)$ symmetry of a charged three-body system. This reduction can be accomplished  symplectically as we will review for general $N$-body systems in this section. We will  follow to a large extent  \cite{LittlejohnReinsch97}.

The special Euclidian group $\SE(3)=\SO(3) \rtimes \R^3$ defines a symplectic group action on the phase space $M = T^*\config$ of an $N$-body system as defined in 
Def.~\ref{def:NBodySystem} according to 
$(\mathbf{x},\mathbf{p}) \mapsto ((\g\,(\mathbf{x}_1 + \mathbf{a} ) ,\ldots, \g\,(\mathbf{x}_N)+\mathbf{a} ),( \g\, \mathbf{p}_1,\ldots,\g\, \mathbf{p}_N)) $ for 
any $\g \in \SO(3)$ and $\mathbf{a}\in \R^3$ where we think of the $\g$ as $3\times3$ orthogonal matrices with unit determinant. The Hamiltonian function in Def.~\ref{def:NBodySystem}  is invariant under this action.
\rem{
Such systems are symmetric under the Euclidian group $\SE(3)$ which is the semi-direct product of the additive group of translations
$(\R^3$ which  for each $\mathbf{a}\in \R^3$ acts as $(\mathbf{x},\mathbf{p}) \mapsto ((\mathbf{x}_1 + \mathbf{a} ,\ldots, \mathbf{x}_N + \mathbf{a})    ,\mathbf{p})$
and the group of rotations $\SO(3)$ which for $\g\in \SO(3)$ acts as $(\mathbf{x},\mathbf{p}) \mapsto ((\g\,\mathbf{x}_1,\ldots, \g\,\mathbf{x}_N),( \g\, \mathbf{p}_1,\ldots,\g\, \mathbf{p}_N)) $. Here we think of the elements $\g$ of $\SO(3)$ as $3\times3$ orthogonal matrices with unit determinant.   
}
The  $\SE(3)$ symmetry can be reduced successively by reducing the translational symmetry $\R^3$ first and the rotational symmetry $\SO(3)$ afterwards.
The symmetry of  translations is easily reduced by choosing Jacobi
vectors $\mathbf{s}_{i}\in\R^3, i=1,...,N-1$ (see, e.g.,  \cite{cornille2003advanced} and Sec.~\ref{charged_sec:triatomic} for a concrete definition in the case of $N=3$). Together with the position vector of the center of mass of the $N$-body system  the Jacobi vectors uniquely determine
the positions of the $N$ bodies  in space.  Due to the absence of external forces, the center of mass of an $N$-body system is moving with a constant velocity and we can choose an inertial frame of reference which has the center of mass at its origin. We view the $\mathbf{s}_{i}, i=1,...,N-1$, to be the coordinate vectors of the Jacobi vectors with respect to this center of mass frame. If we now ignore the trivial position of the center of mass then we can view the space of the Jacobi vectors  $\Qt = \R^{3(N-1)}$ as the translation reduced configuration space. Taking out the collision set this space becomes $\Qt \backslash \Delta_c$ where $\Delta_c$ now is the collision set in terms of the Jacobi vectors (see Sec.~\ref{charged_sec:triatomic} for an example). If we moreover ignore the trivial constant momentum of the center of mass we have reduced the phase space to 
$T^*\Qt \cong (\R^{3(N-1)} \backslash \Delta_c ) \times \R^{3(N-1)}$.

%We note that 
%different choices of Jacobi vectors can be parametrized by the \emph{kinematic
%group} which is thoroughly studied in \cite{MitchellLittlejohn2000}. 

The metric associated with expressing the kinetic energy $T$ in terms of the velocities corresponding to the Jacobi vectors is diagonal.
We will choose the Jacobi  vectors to be mass-weighted so that
the metric becomes  Euclidean, i.e. the kinetic energy assumes the form

\begin{equation}
T= %\frac{1}{2} \sum^{N}_{i=1} m_i \, \mathbf{\dot{x}}^{2}_i=
\frac{1}{2} \sum^{N-1}_{i=1} \mathbf{\dot{s}}^{2}_i\,.
\end{equation}

\noindent
To address the rotational symmetry we consider a body fixed coordinate frame which is a frame that is related 
to the center of mass frame by a rotation.  Let us denote the coordinate vectors of the Jacobi vectors in the body fixed frame
by $\mathbf{r_i}, i=1,..,N-1$. We then have 
$\mathbf{s}_{i}=\g \, \mathbf{r}_{i}$ for some rotation matrix
$\g \in \SO(3)$. 
If we identify all configurations that are related by a rotation about the center of mass we get the so called  \emph{internal space} which we denote by $\Qtr$ and which is formally given by the quotient
$\mathbb{R}^{3(N-1)}/\SO(3)$.
 As the rotation group  $\SO(3)$ has dimension 3  (for example Euler angles or
Cayley-Klein parameters are coordinates on $\SO(3)$) the internal space $\Qtr$ has dimension $3(N-2)$.
Collinear configurations of the $N$ particles in $\R^3$ lead to singularities  in the reduction. Away from collinear configurations
 the internal space has a smooth structure (see Sec.~\ref{charged_sec:triatomic} for more details).  The coordinates on the internal space are called 
\emph{internal coordinates}. We will denote the internal space coordinates vectors by $\qtr$ and their components by $\qtri{\mu}$ with the convention that Greek indices run from 1 to $3(N-2)$. We note that the internal space and the internal coordinates are sometimes also referred to as shape space and shape space coordinates, respectively. We however reserve the term `shape'  for definitions that we make below.

We remark that specifying a body-fixed frame can be phrased in the language of \emph{gauge theory} \cite{LittlejohnReinsch97}. One specific  choice gives rise to one specific gauge.  
The reduced  Hamiltonian function on the phase space reduced by translations and rotations 
is again the sum of the kinetic energy and the potential energy.  
The reduced potential is then a function of the internal coordinates only. The reduced kinetic energy 
is a function of the internal space coordinates and their conjugate momenta and the rotational degrees of freedom. 
The reduction can be phrased in such a way  that the dependence on the gauge becomes apparent (see also \cite{CiftciWaalkens11,CiftciWaalkensBroer2014}). 
To see this we make the following definitions. 

Let $\g\in \SO(3)$ denote the rotation from the center of mass frame to the body frame and $\mathbf{L}$ denote
the angular momentum with respect to the center of mass which in terms of the mass weighted Jacobi vectors is given by

\begin{equation}
\mathbf{L}=\sum^{N-1}_{i=1}\mathbf{s}_{i}\times \mathbf{\dot{s}}_{i}.
\end{equation}

\noindent
Then the body velocities and body angular momentum are defined, respectively, by

\begin{equation}
\mathbf{\dot{r}}_{i}=\g^{T}\mathbf{\dot{s}}_{i},
\end{equation}
and
\begin{equation} \label{charged_eq:def_J}
\mathbf{J}=\g^{T}\mathbf{L} \,.
\end{equation}%

The moment of
inertia tensor $\mathbf{M} (q)$ of the $N$-body system is the tensor with components

\begin{equation} \label{charged_eq:def_moment_intertia_tensor}
 \mathbf{M}_{ij}(q) =  \sum^{N-1}_{k=1}(\mathbf{r}_{k}^{2}\delta _{ij}-r_{ki}r_{kj}) \,,
\end{equation}
where $\mathbf{r}_{k}=(r_{k1},r_{k2},r_{k3})$ in body coordinates. 
With the so called \emph{gauge potential}

\begin{equation}\label{charged_eq:def_gauge_potential}
\mathbf{A}_{\mu }(q)=\mathbf{M}^{-1}(q)\cdot \sum_{i=1}^{N-1} \left( \mathbf{r}_{i}\times \frac{\partial 
\mathbf{r}_{i}}{\partial q_{\mu }}\right)
\end{equation}

\noindent
and the metric

\begin{equation}\label{charged_eq:def_metric_g}
g_{\mu \nu } (q)=  \sum_{i=1}^{N-1}  \frac{\partial \mathbf{r}_{i }(q)} {\partial q_{\mu }} \cdot
\frac{\partial \mathbf{r}_{i }(q)}{\partial q_{\nu }}-\mathbf{A}_{\mu }(q) \cdot \mathbf{M}(q)\cdot \mathbf{A}%
_{\nu } (q)
\end{equation}

\noindent
the kinetic energy becomes

\begin{eqnarray*}
K=\frac{1}{2}(\angvel +\mathbf{A}_{\mu }\dot{q}_{\mu })\cdot \mathbf{M}\cdot
(\angvel +\mathbf{A}_{\nu }\dot{q}_{\nu })\mathbf{+}\frac{1}{2}g_{\mu \nu }%
\dot{q}_{\mu }\dot{q}_{\nu }\,,
\end{eqnarray*}

\noindent
where here and in the following we use the Einstein convention of summation over repeated indices.
Here $\angvel $ is the angular velocity which is obtained by viewing the rotation $R$ relating the space fixed center of mass  frame to the body frame to be time dependent and then using the so$(3)$ Lie algebra isomorphism
between the skew-symmetric matrix $R^{T}\dot{R}$ and the three-component vector $\angvel $ given by

\begin{equation}
\left(
\begin{array}{ccc}
0 & -\angvel_3 & \angvel_2 \\ 
\angvel_3 & 0 & -\angvel_1 \\ 
-\angvel_2 & \angvel_1 & 0
\end{array}
\right) \mapsto
\left(
\begin{array}{ccc}
\angvel_1 \\ 
\angvel_2  \\ 
\angvel_3 
\end{array}
\right)\,. \label{charged_eq:def_R}
\end{equation}

\noindent
By
using the equation

\begin{equation}
\mathbf{J}=\frac{\partial K}{\partial \angvel}=\mathbf{M}(\angvel +\mathbf{A}_{\mu }\dot{q}_{\mu }),
\end{equation}

\noindent
the conjugate momenta of the internal space coordinates are obtained to be

\begin{equation}
p_{\mu }=\frac{\partial K}{\partial \dot{q}_{\mu }}=g_{\mu \nu }\dot{q}_{\nu }+\mathbf{J}\cdot \mathbf{A}_{\mu }\,.
\end{equation}

\begin{definition}
The reduced \emph{ro-vibrational Hamiltonian} is defined  as
%5
%%
\begin{equation}  \label{charged_Hamiltonian}
H(q,p,\mathbf{J})=\frac{1}{2}\mathbf{J}\cdot \mathbf{M}^{-1}\cdot \mathbf{J+}\frac{1}{2}g^{\mu \nu
}(p_{\mu }-\mathbf{J}\cdot \mathbf{A}_{\mu })(p_{\nu }-\mathbf{J}\cdot 
\mathbf{A}_{\nu })
+V(q_{1},...,q_{3N-6})\, ,
\end{equation}
where in order to keep the notation reasonably short we omitted the argument $q$ for $\mathbf{M}$,  $\mathbf{A}_{\mu }$ and $g^{\mu \nu}$ even though these are functions of the internal coordinates  (see Equations~\eqref{charged_eq:def_moment_intertia_tensor}
\eqref{charged_eq:def_gauge_potential}
\eqref{charged_eq:def_metric_g}).
The first term on the right hand side of \eqref{charged_Hamiltonian} is called the \emph{rotational} or \emph{centrifugal kinetic energy} and the second term is called the  \emph{vibrational kinetic energy}. 
\end{definition}

The equations of motion are given by 

\begin{equation}\label{charged_Equations_motion_old}
\dot{q}_{\mu }=\partial H/\partial p_{\mu },\ \ \ \dot{p}_{\mu }=-\partial
H/\partial q_{\mu },  \ \ \  \dot{\mathbf{J}}  = \mathbf{J} \times \nabla_\mathbf{J} H , % \frac{\partial H}{\partial \mathbf{J}} 
\end{equation}

\noindent
where $\nabla_\mathbf{J} H = ( \partial_{J_1} H,   \partial_{J_2} H, \partial_{J_3} H  )$, $\mathbf{J}=(J_{1},J_{2},J_{3})$ and  $\mu =1,...,3N-6$ \cite{KozinRobertsTennyson00}. 
The equations of motion can be considered to describe a rigid body (the last term) whose time-dependent moment of inertia is determined by the internal ccordinates which together with their conjugate momenta change according to the first two sets of equations.
\rem{
The last equation is equivalent to \cite{KozinRobertsTennyson99}

\begin{equation}
\dot{J}_{k}=\left\{ J_{k},H\right\}, 
\end{equation}
where $k=1,2,3$. 
} % end rem
The magnitude  $r=\left\Vert \mathbf{J}\right\Vert$ of the body angular momentum $ \mathbf{J}$ is a constant of
motion.  
Away from collinear configurations, the phase space of the reduced system  with a magnitude of the total angular momentum equal to $r$ then has the structure of a product space given by the product of the angular momentum sphere 
\begin{equation}
S_{r }^{2} = \{             \mathbf{J} \in \R^3 \,|\, J^2_{1}+J^2_{2}+J^2_{3} =  r^2 \}
\end{equation}
 and the cotangent bundle over the internal space, $T^* \Qtr$, with coordinates $(q_{\mu},p_{\mu})$, $\mu=1,\ldots,3N-6$.
As the angular momentum sphere is two-dimensional, the reduced system can be viewed to have 
$1$ rotational degree of freedom and $3N-6$ vibrational degrees of freedom. The rotational and vibrational degrees of freedom  are coupled via the gauge potentials (see Eq.~\eqref{charged_Hamiltonian}) which  give rise to Coriolis terms in the equations of motion.

%%%%%%%%%%%%%%%%%%%%%%
\section{Reduction of three-body systems}
\label{charged_sec:triatomic}

In the following we make the reduction described in the previous section more concrete for the case of three-body systems.
Consider a system of three bodies with masses $m_1$, $m_2$ and $m_3$ and position vectors
$\mathbf{x}_{1},\mathbf{x}_{1},\mathbf{x}_{1} \in \mathbb{R}^3$.
For the potential \eqref{eq:potential}, we write
\begin{equation}\label{eq:potential3b}
V(\mathbf{x}) =  -\frac{\alpha_3}{r_{12}(\mathbf{x})} -  \frac{\alpha_2}{r_{13}(\mathbf{x})} - \frac{\alpha_1}{r_{23}(\mathbf{x})} 
\end{equation} 
with $\alpha_k\in \R$ and $r_{ij} = \norm{\mathbf{x}_i-\mathbf{x}_j}$, $i,j,k=1,2,3$.
We define  mass-weighted Jacobi vectors according to

\begin{eqnarray*}
\mathbf{s}_1&=&\sqrt{\mu_1}(\mathbf{x}_{1}-\mathbf{x}_{3}),\\
\mathbf{s}_2&=&\sqrt{\mu_2} 
(\mathbf{x}_{2}-\frac{m_{1}\mathbf{x}_{1}+m_{3}\mathbf{x}_{3}}{m_{1}+m_{3}}),
\end{eqnarray*}%

\noindent
where

\begin{equation} \label{charged_eq:def_reduced_masses}
 \mu_1=\frac{m_{1}m_{3}}{m_{1}+m_{3}}, \ \ \ \mu_2=\frac{m_{2}(m_{1}+m_{3})}{m_{1}+m_{2}+m_{3}}
\end{equation}
are the reduced masses of the two-body systems with masses $m_1$ and $m_3$ and $m_2$ and $m_1+m_3$, respectively 
(see Fig. \ref{charged_fig:conf}). 

As mentioned in Sec.~\ref{charged_sec:Nbodyreduction}
we can view the space $\Qt= \R^3 \times \R^3$ of Jacobi vectors $\mathbf{s}_1$ and $\mathbf{s}_2$ as the translation reduced configuration space.
Viewing the Jacobi vectors as column vectors of $3\times 2$ matrices we can identify the configuration space with the space of $3\times 2$ matrices $\R^{3\times 2}$.
The translation reduced configuration space $\Qt$ can then be viewed as the disjoint union \cite{IwaiYamaoka05}
\begin{equation}\label{charged_eq:Qunion}
\Qt = \Qti{0} \cup \Qti{1} \cup \Qti{2}  \,, 
\end{equation}
where for $k=0,1,2$,
\begin{equation}
\Qti{k} := \{ A \in \R^{3\times 2} \, : \, \text{rank} A = k \} \,.
\end{equation}
Here $\Qti{2}$ contains the non-collinear configurations, $\Qti{1}$ contains the collinear configurations, and $\Qti{0}$ the triple collision (which is the centre of mass located at the origin). 
We note that for systems with more than three particles, $\Qt$ in \eqref{charged_eq:Qunion} also contains the union with $\Qti{3}$ (defined in an analogous way).
For three particles,  $\Qti{3}$ is empty. 
Moreover, $\Qti{2}$ is a smooth manifold whose boundary is formed by $\Qti{0}\cup \Qti{1}$, i.e. $\partial \Qti{2}= \Qti{0}\cup \Qti{1}$.

The collision set is given by
\begin{equation}
\Delta_c = \{ (\mathbf{s}_1, \mathbf{s}_2) \in \R^3\times\R^3 \, : \,  \mathbf{s}_1 = 0, \,  \mathbf{s}_2 = \frac{\sqrt{\mu_2}}{\sqrt{\mu_1}}    \frac{m_1}{m_1+m_3} \mathbf{s}_1, \text{ or } \mathbf{s}_2 = \frac{\sqrt{\mu_2} }{\sqrt{\mu_1} }\frac{m_3}{m_1+m_3}\mathbf{s}_1\},
\end{equation}
where the conditions defining the set correspond to collisions of particles 1 and 3, 2 and 3, and 1 and 2, in this order. The collision set is contained in the boundary of $\Qti{2}$, i.e. $\Delta_c\subset \partial \Qti{2}$. Let
\begin{equation}
\Delta_{c,k} :=  \Delta\cap \Qti{k}
\end{equation}
for $k=0,1,2$.
Then $\Delta_{c,2}$  is empty, and $\Delta$ equals the disjoint union of 
$\Delta_{c,1}$ and $\Delta_{c,0}$, where $\Delta_{c,1}$ contains the double collisions which are no triple collisions and $\Delta_{c,0}$ contains the triple collision.

The rotation group $\SO(3)$ acts on $\Qt$ according to 
\begin{equation}\label{charged_eq:def_SO(3)_action}
(\mathbf{s}_1,\mathbf{s}_2) \mapsto (\g\, \mathbf{s}_1,\g\, \mathbf{s}_2), \quad \g \in \SO(3).
\end{equation}
The internal space  $\Qtr$ is then the quotient space $\Qt /\SO(3)$ which consists of the equivalence classes of configurations that can be mapped to one another via a rotation $\g\in \SO(3)$.
Let $\pi:\Qt \to \Qt /\SO(3)$ denote the quotient map.
For $(\mathbf{s}_1,\mathbf{s}_2) \in \Qt$,  let $\mathcal{O}_{(\mathbf{s}_1,\mathbf{s}_2) } := \{   (\g\,\mathbf{s}_1,\g\mathbf{s}_2)\, : \, \g \in \SO(3)\}$ be the $\SO(3)$ orbit through 
$(\mathbf{s}_1,\mathbf{s}_2)$ and
$\mathcal{G}_{(\mathbf{s}_1,\mathbf{s}_2) } :=  \{   \g \in \SO(3)   \, : \,   (\g\,\mathbf{s}_1,\g\,\mathbf{s}_2) =  (\mathbf{s}_1,\mathbf{s}_2) \} $ the isotropy group at $(\mathbf{s}_1,\mathbf{s}_2)$. Then
\begin{equation}
\mathcal{G}_{(\mathbf{s}_1,\mathbf{s}_2) } = \left\{ 
\begin{array}{cl}
\{ e \} & \text{ for } (\mathbf{s}_1,\mathbf{s}_2) \in \Qti{2},\\
\SO(2)  & \text{ for } (\mathbf{s}_1,\mathbf{s}_2) \in \Qti{1},\\
 \SO(3)  & \text{ for } (\mathbf{s}_1,\mathbf{s}_2) \in \Qti{0},
\end{array}
\right.
\end{equation}
and
\begin{equation}
\mathcal{O}_{(\mathbf{s}_1,\mathbf{s}_2) } = \left\{ 
\begin{array}{cl}
\SO(3) & \text{ for } (\mathbf{s}_1,\mathbf{s}_2) \in \Qti{2},\\
S^2  & \text{ for } (\mathbf{s}_1,\mathbf{s}_2) \in \Qti{1},\\
 \{ 0 \}  & \text{ for } (\mathbf{s}_1,\mathbf{s}_2) \in \Qti{0}.
\end{array}
\right.
\end{equation}
The configuration space can hence be viewed to be stratified into three strata defined via the orbit type, i.e. $\Qt=\Qti{2} \cup \Qti{1} \cup \Qti{0}$, and the projection is similarly stratified according to 
\begin{equation}
\Qti{2} \to \Qti{2}/\SO(3),\quad \Qti{1} \to \Qti{1}/\SO(3),\text{ and } \Qti{0} \to \Qti{0}/\SO(3).
\end{equation}

\begin{figure}
\begin{center}
\includegraphics[angle=0,width=7cm]{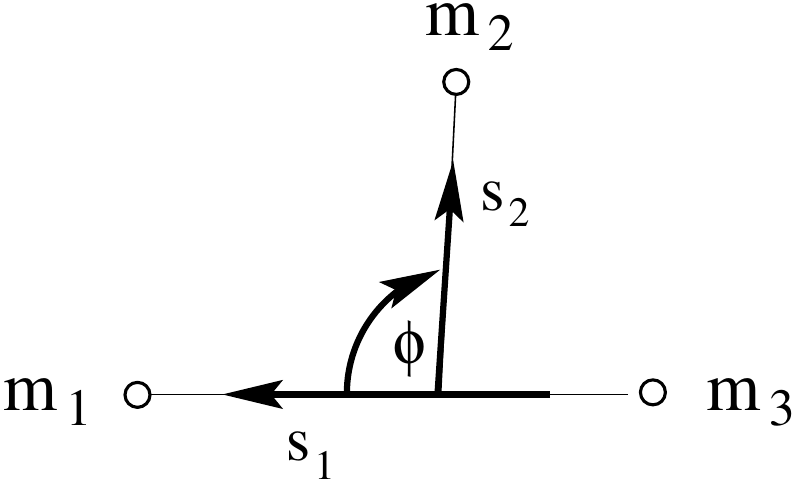}
\end{center}
\caption{\label{charged_fig:conf}
Directions of the mass weighted Jacobi vectors $\mathbf{s}_{1}$ and $\mathbf{s}_{2}$. The vector  $\mathbf{s}_{2}$ has its tail at the centre of mass of the particles 1 and 3.  The Jacobi coordinates $\rho_1$ and $\rho_2$ are the lengths of the vectors $\bs_1$ and $\bs_2$, respectively,  and $\phi$ is the angle between the two vectors. 
}
\end{figure}

As coordinates on the internal space we can take
the \emph{Jacobi coordinates} $(\rho_1,\rho_2,\phi)$  defined as
\begin{equation*}      \label{charged_eq:def_rhophi}        
\rho_1=\|\mathbf{s}_{1}\|\,,\quad \rho_2=\|\mathbf{s}_{2}\|\,, \quad \mathbf{s}_{1}\cdot\mathbf{s}_{2}=\rho_1\, \rho_2\, \cos\phi\,,
\end{equation*}
where $0\leq\phi\leq\pi$ (see Fig.~\ref{charged_fig:conf}).

Let $\{ \be_1, \be_2, \be_3 \}$ be the standard basis in $\R^3$. Then we can define a section $\sigma:\Qt/\SO(3) \to \Qt$ as 
\begin{equation}  \label{charged_eq:xxy_gauge_section}
\sigma(\rho_1,\rho_2,\phi) = (\br_1,\br_2) := ( \rho_1 \be_1,  \rho_2 \cos \phi\, \be_1 + \rho_2 \sin \phi\, \be_2).
\end{equation}
This section is called the $xxy$-gauge in  \cite{LittlejohnReinsch97} as it corresponds to the choice of a body frame where two bodies (bodies 1 and 3 in our case) are located on the $x$ axis
and the third body (body 2) is contained in the $xy$ plane.
The mass weighted Jacobi vectors $(\bs_1,\bs_2)$  of the three-body system are then given by
\begin{equation}
(\bs_1,\bs_2) = (\g\,\br_1,\g\,\br_2) =  ( \rho_1 \g\, \be_1, \rho_2 \cos \phi\, \g \,\be_1 + \rho_2 \sin \phi \,\g \,\be_2)
\end{equation}
for some $\g\in \SO(3)$.

\rem{
\begin{figure}
\begin{center}
\includegraphics[angle=0,width=7cm]{Body-fixed_xfig}
\end{center}
\caption{\label{charged_fig:xxy}
The $xxy$-gauge.
}
\end{figure}
} % end rem

Collinear configurations are given in terms of  the Jacobi coordinates by either of the equalities $\phi=0$, $\phi=\pi$ and $\rho_2=0$.  
For $\rho_1=0$ (in which case $\phi$ is not defined), particles 1 and 3 collide. For $\phi=0$ combined with 
$\rho_2 = \frac{m_3}{m_1 + m_3} \frac{\sqrt{\mu_2}}{\sqrt{\mu_1}} \rho_1$, particles 1 and 2 collide. For $\phi=\pi$ combined with 
$\rho_2 = \frac{m_1}{m_1 + m_3} \frac{\sqrt{\mu_2}}{\sqrt{\mu_1}} \rho_1$, particles 1 and 3 collide. At the triple collision $\rho_1=\rho_2=0$.

Besides the Jacobi coordinates another natural choice of coordinates is given by  the inter particle distances $r_{12}$, $r_{13}$ and $r_{23}$ which besides being nonnegative need to satisfy the triangle inequality $ r_{12}+r_{13} \ge r_{23} $ and its cyclic permutations. Collinearity is given by equality in either of the triangle inequalites. 
Double collisions between two particles are obviously given by the corresponding distance being zero. At the triple collision all distances are vanishing. 

As will become clear below for the discussion of the Hill regions, it is useful to introduce yet another coordinate system. To this end we first define
\begin{equation}\label{charged_eq:Jacobi2w}
(w_1,w_2,w_3) = (\rho_1^{2}-\rho_2^{2}, 2\, \rho_1 \, \rho_2\cos\phi, 2\, \rho_1 \, \rho_2\, \sin\phi )\,,
\end{equation}
where $w_1,w_2\in \R$ and $w_3\ge 0$. Equation~\eqref{charged_eq:Jacobi2w} shows that the Jacobi coordinates are confocal parabolic coordinates in the space of the coordinates $(w_1,w_2,w_3)$.  The coordinate $w_3$ is twice the area of the parallelogram spanned by the Jacobi vectors $\mathbf{s}_{1}$ and $\mathbf{s}_{2}$. This implies that collinear configurations are contained in the plane $w_3=0$.  This plane hence also contains the collisions. As double collisions of particles of particles 1 and 3 have $\rho_1=0$ we see from \eqref{charged_eq:Jacobi2w} that these are located on the negative $w_1$ axis. Double collisions of 
particles 1 and 2 occur on the line in the plane $w_3=0$ where

\begin{equation} \label{charged_eq:angle_collision_12}
\frac{w_2}{w_1}= 2\frac{\sqrt{m_1 m_2 m_3} \, \sqrt{m_1 + m_2 + m_3} }{m_1(m_1+m_2+m_3) -m_2 m_3 }.
\end{equation}
Similarly collisions of particles 2 and 3 occur in this plane at

\begin{equation}\label{charged_eq:angle_collision_23}
\frac{w_2}{w_1} = -2\frac{\sqrt{m_1 m_2 m_3}\, \sqrt{m_1 + m_2 + m_3} }{m_3(m_1 + m_2 + m_3) - m_1 m_2}.
\end{equation}
The triple collision is located at the origin $w_1=w_2=w_3=0$.

The final coordinate system we are considering is then given by 
spherical coordinates  in the  $(w_1,w_2,w_3)$ coordinate space which give the Dragt's coordinates $(\omega,\chi,\psi)$ defined as (see \cite{LittlejohnReinsch97} and the references therein)
\begin{equation}
(w_1,w_2,w_3) = (\omega\cos\chi\cos\psi, \omega\cos\chi\sin\psi ,\omega\sin\chi)\,,
\end{equation}
where $\omega\ge 0$, $0\le \chi\le \pi/2$ and  $0\le\psi\le 2\pi$.  
Note that $\chi$ is the latitude, not the colatitude, and
 $\omega = \rho_1^2 + \rho_2^2$.

For completeness, we also give the expressions for the inter particle distances in terms of Dragt's coordinates:
\begin{equation}
\begin{split}
r_{12} &= \frac{1}{\sqrt{2\mu_1}} \sqrt{\omega + \omega \cos \chi \cos \psi} \,, \\
r_{13} &= \sqrt{ \frac{\mu_1}{2m_1^2}  (\omega + \omega \cos \chi \cos \psi) + \frac{1}{2\mu_2} (\omega -\omega  \cos \chi \cos \psi ) -\frac{\sqrt{\mu_1}}{m_1 \sqrt{\mu_2}} 
\omega \cos \chi \sin \psi}\,, \\
r_{23} &= \sqrt{ \frac{\mu_1}{2m_3^2}  (\omega + \omega \cos \chi \cos \psi) + \frac{1}{2\mu_2} (\omega -\omega  \cos \chi \cos \psi ) +\frac{\sqrt{\mu_1}}{m_3 \sqrt{\mu_2}} 
\omega \cos \chi \sin \psi}\,.
\end{split}
\end{equation}

%%%%% to be included again later
%\rem{
\begin{figure}
\begin{center}
\includegraphics[angle=0,width=8cm]{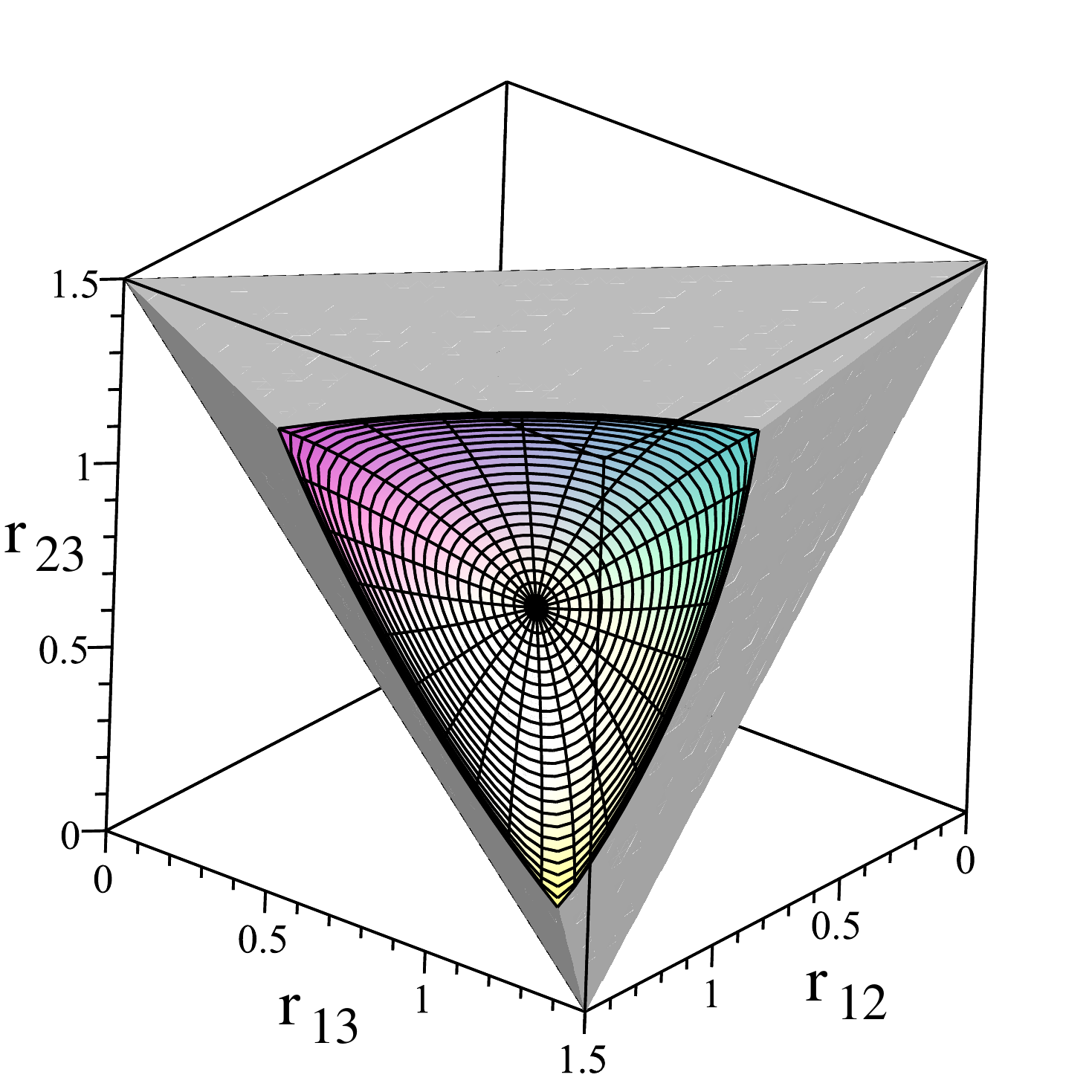}
\end{center}
\caption{\label{charged_fig:Dragtscoordinates}
The coordinate surface $\omega=1$ of the Dragt's coordinates in the space of the inter particle distances $r_{12}$, $r_{13}$ and $r_{23}$. This surface can be viewed as the shape  space $\tilde{\Qtr}$ defined in  \eqref{charged_def_dilredshapespace}.
}
\end{figure}
%} % end rem

By choosing the internal coordinates as the Jacobi coordinates ($\rho_1,\rho_2,\phi$) 
the inertia tensor becomes \cite{LittlejohnReinsch97}

\begin{equation} \label{eq:M_Jacobi}
\mathbf{M}=\left[
\begin{array}{ccc} 
\rho_{2}^{2} \sin^2 \phi & - \rho_{2}^{2} \sin \phi \cos \phi & 0 \\ 
-\rho_{2}^{2} \sin \phi \cos \phi & \rho_{1}^{2}+ \rho_{2}^{2} \cos^2  \phi & 0 \\ 
0 & 0 & \rho_{1}^{2}+ \rho_{2}^{2} 
\end{array}
\right]\,.
\end{equation}
Its eigenvalues give the  principal  moments of inertia
\begin{equation}\label{eq:M12_Jacobi}
M_{1/2} = \frac12 \big( \rho_1^2 + \rho_2^2 \mp \sqrt{\rho_1^4 + \rho_2^4 + 2\cos(2\phi) \rho_1^2 \rho_2^2 } \,\big)
\end{equation}
and $M_3 =  \rho_1^2 + \rho_2^2 $.
The metric and the gauge potential become
\begin{equation}
 \left[ g_{\mu\nu}\right] =\left[ 
\begin{array}{ccc}
1 & 0 & 0 \\ 
0 & 1 & 0 \\ 
0 & 0 & \frac{\rho_{1}^{2}\rho_{2}^{2}}{\rho_{1}^{2}+\rho_{2}^{2}}%
\end{array}%
\right]
\end{equation}

and

\begin{equation}
 \mathbf{A}_{\rho_1 }=\mathbf{A}_{\rho_2}=(0,0,0),\ \ \ \mathbf{A}_{\phi }=(0,0,\frac{\rho_{2}^{2}}{\rho_{1}^{2}+\rho_{2}^{2}} ).
\end{equation}

\noindent
In Dragt's coordinates
the inertia and metric tensors are diagonal   \cite{LittlejohnReinsch97}:

\begin{equation} \label{charged_eq:inertiaDragt}
\mathbf{M} =\left[ 
\begin{array}{ccc}
\omega \sin ^{2}\frac{\chi }{2} & 0 & 0 \\ 
0 & \omega \cos ^{2}\frac{\chi }{2} & 0 \\ 
0 & 0 & \omega %
\end{array}%
\right]
\end{equation}
and
\begin{equation}
\left[ g_{\mu\nu}\right] =\frac{1}{4}\, \left[ 
\begin{array}{ccc}
\frac{1}{\omega} & 0 & 0 \\ 
0 & \omega & 0 \\ 
0 & 0 & \omega \cos^2\chi%
\end{array}%
\right]\,,
\end{equation}
In particular the principal moments of inertia  $M_1=\omega \sin ^{2}\frac{\chi }{2} $, $M_2=\omega \cos ^{2}\frac{\chi }{2} $ and $M_3=\omega $ are ordered by magnitude on the diagonal of $\mathbf{M} $ (note that $0\le \chi\le \pi/2$).

\noindent
The gauge potential becomes in this case

\begin{equation}
\mathbf{A}_{\omega }=\mathbf{A}_{\chi }=(0,0,0),\ \ \ \mathbf{A}_{\psi }=(0,0,-\frac{1}{2}\sin \chi ).
\end{equation}%

\rem{
\noindent
Hence, if the coordinates (\ref{charged_Coordinates}) are used then
after relabeling $u=q_{0},v=p_{0},\omega =q_{1},\chi =q_{2},\psi =q_{3}$, the reduced  Hamiltonian becomes

\begin{eqnarray*} \label{charged_eq:DragtsCoordinates1}
H(q,p)&=&\frac{(r^{2}-p_{0}^{2})\cos ^{2}q_{0}}{2q_{1}\sin ^{2}\frac{q_{2}}{2}}+\frac{(r^{2}-p_{0}^{2})\sin ^{2}q_{0}}{2q_{1}\cos ^{2}\frac{q_{2}}{2}}
+\frac{p_{0}^{2}}{2q_{1}}+2q_{1}p_{1}^{2}
+\frac{2p_{2}^{2}}{q_{1}}\\
&&+\frac{(2p_{3}+ p_{0}\sin{q_{2}})^2}{2q_{1}\cos^2 q_{2}}
+V(q_{1},q_{2},q_{3})\,,
\end{eqnarray*}

\noindent
or if the coordinates (\ref{charged_Coordinates2}) are used one has

\begin{eqnarray*} \label{charged_eq:DragtsCoordinates2}
H(q,p)&=&\frac{p_{0}^{2}}{2q_{1}\sin^{2}\frac{q_{2}}{2}}+\frac{(r^{2}-p_{0}^{2})\sin^{2}q_{0}}{2q_{1}\cos^{2}\frac{q_{2}}{2}}+\frac{(r^{2}-p_{0}^{2})\cos^{2}q_{0}}{2q_{1}}+2q_{1}p_{1}^{2}+\frac{2p_{2}^{2}}{q_{1}}\\
&&+\frac{(2p_{3}+\sqrt{r^{2}-p_{0}^{2}}\cos q_{0}\sin{q_{2}})^2}{2q_{1}\cos^{2}q_{2}}+V(q_{1},q_{2},q_{3}).
\end{eqnarray*}
} % end rem

\rem{
\noindent
If we use coordinates (\ref{charged_Coordinates}) 
after relabeling $u=q_{0},v=p_{0},\rho_1 =q_{1},\rho_2=q_{2},\phi =q_{3}$, the Hamiltonian becomes

\begin{eqnarray*}
 H&=&\frac{1}{2} \{ \frac{q_{1}^{2}+q_{2}^{2} \cos^2  q_{3}}{q_{1}^{2}q_{2}^{2} \sin^2  q_{3}}(r^2-p_0^2)\cos^2 q_0
+\frac{2 \cos q_{3}}{q_{1}^{2} \sin q_{3}}(r^2-p_0^2)\cos q_0\sin q_0\\
&+&\frac{1}{q_{1}^{2}}(r^2-p_0^2)\sin^2 q_0+\frac{1}{q_{1}^{2} 
+q_{2}^{2}}p_{0}^{2}+p_{1}^{2}+p_{2}^{2}
+\frac{q_1^2+q_2^2}{q_1^2 q_2^2}(p_{3}-\frac{q_{2}^{2}}{q_{1}^{2}+q_{2}^{2}}p_{0})^{2} \} \\
&+&V(q_1,q_2,q_3),
\end{eqnarray*}

\noindent
If the coordinates (\ref{charged_Coordinates2}) are 
are chosen, then the Hamiltonian becomes

\begin{eqnarray*}
 H&=&\frac{1}{2} \{ \frac{q_{1}^{2}+q_{2}^{2} \cos^2  q_{3}}{q_{1}^{2}q_{2}^{2} \sin^2  q_{3}}p_{0}^{2}
+\frac{2 \cos q_{3}}{q_{1}^{2} \sin q_{3}}p_{0}\sqrt{r^2-p_0^2}\sin q_0\\
&+&\frac{1}{q_{1}^{2}}(r^2-p_0^2)\sin^2 q_0+\frac{1}{q_{1}^{2} 
+q_{2}^{2}}(r^2-p_0^2)\cos^2 q_0+p_{1}^{2}+p_{2}^{2}\\
&+&\frac{q_1^2+q_2^2}{q_1^2 q_2^2}(p_{3}-\frac{q_{2}^{2}}{q_{1}^{2}+q_{2}^{2}}\sqrt{r^2-p_0^2}\cos q_0)^{2} \}+V(q_1,q_2,q_3).
\end{eqnarray*}

} % end rem

\begin{remark} \label{charged_remark:moments_inertia}
The first and second principal moments of inertia add up to the third (see, e.g., \eqref{charged_eq:inertiaDragt}):
\begin{equation}
M_1 + M_2 = M_3\,.
\end{equation}
It follows from our choice of the $xxy$-gauge that
the principal moment of inertia $M_3$ corresponds to the axis of rotation that contains the centre of mass and is perpendicular to the plane in which  the three bodies are lying.
In the celestial mechanics literature $M_3$ is simply referred to as \emph{moment of inertia} (see, e.g., \cite{Moeckel}), and it is common to use the symbol $I$ for it. In Jacobi and Dragt's coordinates, respectively, it is then given by
\begin{equation} \label{charged_eq:def_moment_I}
I := M_3= \frac12 \text{tr} \,M =  \rho_1^2 + \rho_2^2  =  \omega.
\end{equation}
This will play an important role in Sec.~\ref{charged_sec:shapespace}.
\end{remark}

\begin{remark}\label{lemma:M_tilde_singular}
The moment of inertia tensor is a smooth function on the  internal space  $\Qtr$. However,
whereas the third principal moment of inertia $M_3$ is a smooth function on the  internal space  $\Qtr$, 
the first and the second principal moments of inertia $M_1$ and $M_2$ are continuous on $\Qtr$  and smooth only at points $q \in  \Qtr$  where  $M_1(q) \ne M_2(q) $. 
This can be seen from using Jacobi coordinates $\rho_1$, $\rho_2$ and $\phi$ which as opposed to Dragt's coordinates are smooth coordinates on the interior of the internal space (i.e. away from collinear configurations). 
The principal moments of inertia $M_1$ and $M_2$  are not differentiable when the square root in \eqref{eq:M12_Jacobi} vanishes which happens only for $\rho_1=\rho_2$ together with $\phi=\pi/2$, i.e. when $M_1=M_2$.  
In quantum physics such singular points of parameter families of Hermitian operators are referred to as \emph{diabolic points} and we follow this terminology. The diabolic points of the moment of inertia tensor have $(w_1,w_2,w_3)=(0,0,w)$ and the angle $\chi$  of Dragt's coordinates is not defined. In fact it is not possible to have smooth coordinates on the internal space in terms of which the moment of inertial tensor is diagonal.
\end{remark}

\begin{remark}\label{remark:equal_moments_inertia}
Collinear configurations (which include double collisions) are given in Jacobi coordinates by $\phi=0$, $\phi=\pi$, $\rho_2=0$ or $\rho_1=0$ where the latter always corresponds to a double collision, and in Dragt's coordinates by $\chi=0$, i.e. $w_3=0$, and hence $w_1^2+w_2^2=\omega^2$. Here $M_1=0$ and $M_2=M_3=\rho_2^2=\omega$.
\end{remark}

% end rem

%%%%%%%%%%%%%%%%%%%%%%
\section{Dilation symmetry and the shape-orientation space}
\label{charged_sec:shapespace}

As the next step towards our definition of a Hill region 
we reduce the non-symplectic scaling symmetry of charged three-body systems. After having reduced the translational and rotational symmetries in Secs.~\ref{charged_sec:Nbodyreduction} and \ref{charged_sec:triatomic} the
reduction of this scaling will lead to our notion of the shape and orientation of a charged three-body system which will then enter our definition of a  Hill region in Sec.~\ref{charged_sec:Hillregions}.

The potential $V$ in \eqref{eq:potential3b} and the inertia tensor $\mathbf{M}$ with the components given in \eqref{charged_eq:def_moment_intertia_tensor}
are homogenous functions of the distances between the particles of degree $-1$ and $2$, respectively. We make this more formal by defining an $\R$ action as follows.
\begin{definition}
The \emph{dilation transformation} is the $\R$ action on the translation-reduced configuration space  $\Qt $ defined as
\begin{equation} \label{charged_eq:def_dilation}
d : \R \times \Qt  \to \Qt , \quad (l,(\mathbf{s}_1,\mathbf{s}_2)) \mapsto (\ue^l \mathbf{s}_1, \ue^l \mathbf{s}_2).
\end{equation}
We write 
\begin{equation} \label{charged_eq:def_d_lambda}
d_\lambda (\mathbf{s}_1,\mathbf{s}_2) = (\lambda \mathbf{s}_1,\lambda \mathbf{s}_2)\,,
\end{equation}
where $\lambda = \ue^l>0$. 
\end{definition}

The dilation transformation defines a singular line bundle $\Qt \to \Qt/ \R$. The only point with nontrivial isotropy is the triple collision point. The bundle
$(\Qti{2} \cup \Qti{1}) \to (\Qti{2} \cup \Qti{1}) / \R$ is smooth. 
The dilation $\R$ action commutes with the $\SO(3)$ action in \eqref{charged_eq:def_SO(3)_action}, i.e. $d_\lambda (\g\, \mathbf{s}_1,\g\, \mathbf{s}_2) =   (\lambda \g\, \mathbf{s}_1,\lambda \g\, \mathbf{s}_2) = (\g \lambda \mathbf{s}_1,\g \lambda \mathbf{s}_2) $ for all $\lambda>0$ and $\g\in \SO(3)$. 
As we can identify the internal space $\Qtr=\Qt /\SO(3)$ with the (image of the) section given by the $xxy$-gauge the dilation transformation also 
`induces' a map on the internal space $\Qtr $. From \eqref{charged_eq:xxy_gauge_section} we get
\begin{equation}
d_\lambda  (\br_1,\br_2) = (\lambda  \rho_1 \be_1, \lambda \rho_2 \cos \phi\, \be_1 + \lambda \rho_2 \sin \phi\, \be_2),
\end{equation}
i.e. in Jacobi the induced dilation map $\Qtr \to \Qtr $, $q\mapsto d_\lambda (q)$ reads

\begin{equation}
d_\lambda (\rho_1,\rho_2,\phi) = (\lambda \rho_1, \lambda \rho_2, \phi).
\end{equation}
In terms of Dragt's coordinates the map becomes

\begin{equation}
d_\lambda (\omega,\chi,\psi) = (\lambda^2 \omega, \chi, \psi ).
\end{equation}
In order to avoid a cumbersome notation we here use the same symbol for the induced map as in \eqref{charged_eq:def_d_lambda}.

Homogeneity of $V$ and $\mathbf{M}$ now means that
\begin{equation} \label{charged_eq:homogeneities}
V(d_\lambda (q)) = \lambda^{-1} V(q) \quad \text{ and } \quad  \mathbf{M} (d_\lambda (q)) = \lambda^2 \mathbf{M}(q)  \,.
\end{equation}
For the moment of inertia $I$ defined in \eqref{charged_eq:def_moment_I}, we similarly have 
\begin{equation}\label{charged_eq:hom_I}
I(d_\lambda (q)) = \lambda^2 I(q).
\end{equation}

In fact the (induced) dilation map also defines a line bundle $\Qtr \to \Qtr / \R$. 
For this bundle, we can construct a section by noting that except for the triple collision point we can find for each point $q\in \Qtr$  
 a $\lambda>0$ such that for the moment of inertia defined in \eqref{charged_eq:def_moment_I} we have $I(d_\lambda(q))=1$. We can then identify the quotient space 
 $ \Qtr  / \R$ with this section.

\begin{definition}\label{charged_def:shap_space}
The \emph{shape space} is the dilation reduced interior of the internal space given by
\begin{equation}\label{charged_def_dilredshapespace}
\tilde{\Qtr } := \{ q \in \Qtr^{\mathrm{o}}  \,:\,  I(q)= 1  \}.
\end{equation} 
The points in  the shape space are referred to as \emph{shapes} and denoted by $\tilde{q}$.
\end{definition}

\rem{
\begin{equation}\label{charged_def_dilredshapespace}
\tilde{Q} := \{ (\mathbf{s}_1, \mathbf{s}_2) \in Q \,:\,  \|\mathbf{s}_{1}\|^2 +  \|\mathbf{s}_{2}\|^2 = 1  \}/\SO(3) 
\end{equation} 
the \emph{shape space}.  Note that in this definition we made use of the commutativity of the rotations and dilations.
} % end rem

\noindent
%Note that the condition  $ I(q)= 1$ excludes the triple collision. 
The boundary of the shape space (which is not part of the shape space) consists of collinear configurations. We excluded collinear configurations in the definition of the shape space to avoid difficulties resulting from singularities in the reduction of the rotational symmetry.

We can identify the shape space $\tilde{\Qtr}$ with the upper hemisphere of the unit sphere in the coordinate space $(w_1,w_2,w_3)$  on which we can use $(w_1,w_2)$ or Dragt's coordinates $(\chi,\psi)$ as coordinates. Note however that the latter are singular at the pole $(w_1,w_2,w_3)=(0,0,1)$. For comparison, $\tilde{\Qtr }$ is shown in the space of inter particle distances in Fig.~\ref{charged_fig:Dragtscoordinates}. The orbits of the dilation $\R$ action in the internal spaces with coordinates  $(w_1,w_2,w_3)$ and $(d_{12},d_{13},d_{23})$ are straight line rays emanating from (but not including) the origin which corresponds to the triple collision point. In view of the definition of the shape space according to \eqref{charged_def_dilredshapespace} we can view the shape space to form a submanifold of the internal space and this way also can define the action of the dilation transformation $d_\lambda$ on points $\tilde{q}\in \tilde{\Qtr }$. For   
$q=\tilde{q}$ with $q\in \Qtr$, we set $d_\lambda(\tilde{q}) := d_\lambda(q)$.
Let
\begin{equation} \label{charged_eq:def_VM_tilde}
\tilde{V} := \left. V \right|_{\tilde{Q}} 
 \quad \text{ and }\quad
\mathbf{\tilde{M}}  := \left. \mathbf{M} \right|_{\tilde{Q}}    .
\end{equation}
\rem{
\begin{equation}
\mathbf{\tilde{M}}  :=  \frac{1}{I} \mathbf{M}\quad \text{ and }\quad
\tilde{V} := \sqrt{I}  \, V.
\end{equation}
} % end rem
Then for $\tilde{q}\in \tilde{Q}$, 
\begin{equation} \label{charged_eq:hom_VM_tilde}
V(d_\lambda(\tilde{q})) = 
\lambda^{-1}  \tilde{V} (\tilde{q} ) 
 \quad \text{ and }\quad
\mathbf{M} ( d_\lambda(\tilde{q})   ) =  \lambda^{2} \mathbf{\tilde{M}} (\tilde{q}) .
\end{equation}

\rem{
Then $\mathbf{\tilde{M}} $ and $\tilde{V} $ are constant on the orbits of the dilation $\R$ action. In terms of the Dragt's coordinates this means that 
$\mathbf{\tilde{M}}$ and $\tilde{V} $ no longer depend on $\omega$ and are functions of $\chi$  and $\psi$ only.
} % end rem

Let $\tilde{M}_k$, $k=1,2,3$, be the dilation reduced principal moments of inertia, i.e. the eigenvalues of $\mathbf{\tilde{M}}$.
Then by construction $\tilde{M}_3$ is constant 1 on the  the shape space $\tilde{Q}$.
The graphs of $\tilde{M}_1$ and $\tilde{M}_2$  over the shape space considered as the unit disk in the $(w_1,w_2)$-plane are shown in  Fig.~\ref{charged_fig:principal_moments}.  Inline with Remark~\ref{lemma:M_tilde_singular}, we see that the graphs are not smooth  at  the diabolic point
$w_1=w_2=0$ where  $\tilde{M}_1 = \tilde{M}_2$. Note the rotational symmetry in Fig.~\ref{charged_fig:principal_moments} resulting from the independence of the principal moments of inertia in \eqref{charged_eq:inertiaDragt} from the angle $\psi$.  This means that there are loops in the internal space and shape space along which the principal moments are constant.

\begin{figure}
\begin{center}
\includegraphics[angle=0,width=7cm]{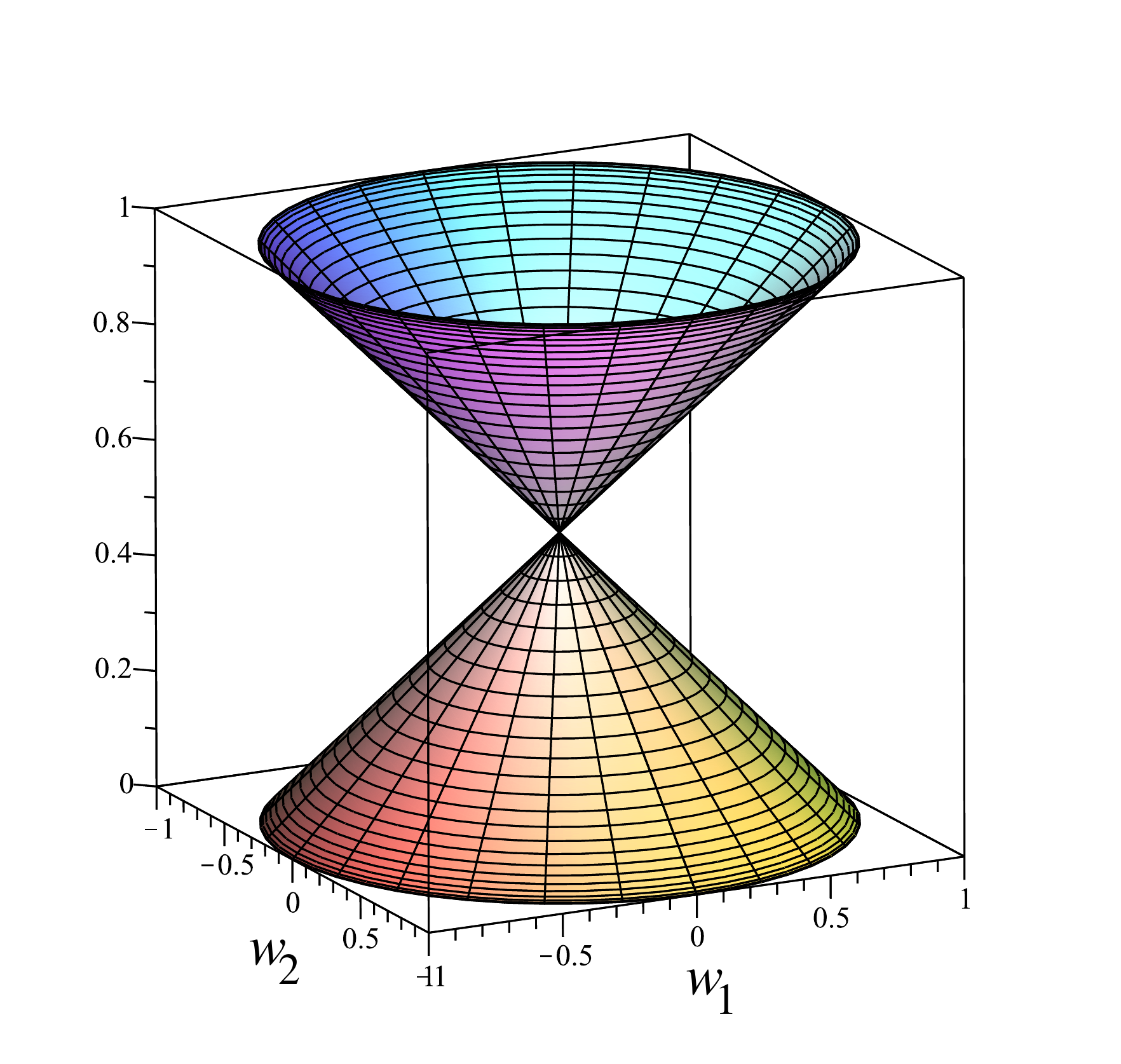}
\end{center}
\caption{\label{charged_fig:principal_moments}
Graphs of the principal moments of inertia $\tilde{M}_1$ (lower cone) and $\tilde{M}_2$ (upper cone) over the shape space given by the unit disk in the $(w_1,w_2)$-plane showing the diabolic point in Remark~\ref{lemma:M_tilde_singular}.
}
\end{figure}

In addition to the shape space we define the normalised body angular momenta

\begin{equation}
\mathbf{\tilde{J}}  = \frac{1}{\|  \mathbf{J}  \| } \mathbf{J} 
\end{equation}
and the normalised angular momentum sphere

\begin{eqnarray}
S^2_1 &=& \{   \mathbf{\tilde{J}} \in\R^3 \,:\,    \tilde{J}_1^2 +     \tilde{J}_2^2 +   \tilde{J}_3^2  = 1 \}.
\end{eqnarray}

Whereas the shape space determines the shape formed by the $3$ bodies the normalised angular momentum sphere $S^2_1$ contains information about the orientation of the $3$-body system in  space. In fact, a point on  $S^2_1$ specifies the direction of the conserved space fixed angular momentum $\mathbf{L}$ in the body frame. This fixes the orientation of the three-body system in space up to the angle of rotation of the three-body system about the axis through the center of mass with the direction of the angular momentum 
$\mathbf{L}$.

\rem{
{\bf In fact, from \eqref{charged_eq:def_J}  we have $\mathbf{\tilde{J}}= \frac{1}{r}R^T \mathbf{L}$ . . .   say more about how $S^2_1$ contains information about the orientation of the 3-body system (either here or in Sec.~\ref{charged_sec:Hillregions}) }
} % end rem
We make the following definition.

\begin{definition} The \emph{shape-orientation space} is the product
\begin{equation}
\tilde{Q} \times S_1^2.
\end{equation}
Points in this product are referred to as \emph{shape-orientation points} and denoted as $(\tilde{q}, \mathbf{\tilde{J}} )$.
\end{definition}
%%

%%%%%%%%%%%%%%%%%%%%%%%%%%%%%%%

\section{Hill regions}
\label{charged_sec:Hillregions}

For a given energy $E$, the Hill region of a Hamiltonian system whose phase space is a cotangent bundle over configuration space is in general defined as the projection of the energy surface to configuration space. As the angular momentum is conserved for $N$-body systems it is useful to not only consider the projection of the energy surface but the projection of the integral manifold where both the energy and the angular momentum are fixed. Furthermore it is useful to consider the projection of the integral manifold to a space reduced by the symmetries of translations and rotations and, for charged $N$-body systems, also by dilations. The reduced phase space is then however no longer a cotangent bundle (see Sec.~\ref{charged_sec:Nbodyreduction}) and it needs to be defined what a Hill region should be in this case. We do this as follows.

\begin{definition}\label{charged_def:Hillregion}
For given value $E$ of the energy and magnitude $r>0$ of the angular momentum (i.e. $r=\Vert \mathbf{L} \Vert >0$), we say that a shape-orientation point $(\tilde{q},\mathbf{\tilde{J}})$ 
is in the Hill region if for the Hamiltonian function $H(q,p,\mathbf{J})$ in \eqref{charged_Hamiltonian},  there  exists
a $\lambda>0$ and $p$ such that 
\begin{equation}\label{charged_eq:energy_equation_def_Hill}
H(d_\lambda(\tilde{q}),p,r\mathbf{\tilde{J}}) = E.
\end{equation} 
\rem{
\begin{equation} \label{charged_eq:def_energy_Hill}
r^2\frac{1}{2}\mathbf{\tilde{J}}\cdot \mathbf{M}^{-1}\cdot \mathbf{\tilde{J}} + \frac{1}{2}g^{\mu \nu
}(p_{\mu }-\mathbf{\tilde{J}}\cdot \mathbf{A}_{\mu })(p_{\nu }-\mathbf{J}\cdot 
\mathbf{A}_{\nu })
+V( d_\lambda (\tilde{q} ) )  =E\,.
\end{equation}
} % end rem
\end{definition}
For computations, it is useful to have the following characterisation. 
\begin{lemma}\label{charged_lemma:Hill}
For given values of the energy $E=H$ and magnitude of the angular momentum $r=\Vert \mathbf{L} \Vert >0$, a shape-orientation point $(\tilde{q},\mathbf{\tilde{J}})$ 
is in the Hill region if and only if there exists a $\lambda>0$ such that

\begin{equation}\label{charged_eq:F(lambda)}
F(\lambda) :=  \lambda^2\, E  -  r^2 \frac{1}{2}\mathbf{\tilde{J}}\cdot [\mathbf{\tilde{M}} (\tilde{q})]^{-1}\cdot \mathbf{\tilde{J}}  - \lambda \, \tilde{V}(\tilde{q}) \ge 0 ,
\end{equation}
where $ \tilde{V}$ and $\mathbf{\tilde{M}}$ are the restrictions of $V$ and $\mathbf{M}$ to the shape space (see \eqref{charged_eq:def_VM_tilde}).
\end{lemma}

\begin{proof}
As the metric $g^{\mu \nu}$ is positive definite the vibrational kinetic energy (the second term in the Hamiltonian in  \eqref{charged_Hamiltonian}) is nonnegative, we can satisfy the 
energy equation \eqref{charged_eq:energy_equation_def_Hill} for some
$\lambda>0$ and $p$ if and only if
\begin{equation} \label{charged_eq:Hill_general}
r^2 \frac{1}{2}\mathbf{\tilde{J}}\cdot [ \mathbf{M}(d_\lambda(\tilde{q}))]^{-1} \cdot \mathbf{\tilde{J}}  + V(d_\lambda(\tilde{q})) \le E.
\end{equation}
By  \eqref{charged_eq:hom_VM_tilde} this is equivalent to 
\begin{equation}\label{charged_eq:proof_F(lambda)_inequal}
r^2 \lambda^{-2} \frac{1}{2} \mathbf{\tilde{J}}\cdot [ \mathbf{\tilde{M}}(\tilde{q})]^{-1} \cdot \mathbf{\tilde{J}}  + \lambda^{-1} \tilde{V}(\tilde{q}) \le E.
\end{equation}
Now the inequality $F(\lambda)\ge 0$ in  \eqref{charged_eq:F(lambda)} is obtained from multiplying \eqref{charged_eq:proof_F(lambda)_inequal} by $\lambda^2$ and reordering terms.
\end{proof}

From Lemma~\ref{charged_lemma:Hill} we see that in order to decide whether a shape orientation point $(\tilde{q},\mathbf{J})$ is in the Hill region we need to study the polynomial $F(\lambda)$ with its coefficients being fixed by the given $(\tilde{q},\mathbf{J})$. 
Let 
\begin{equation} \label{charged_eq:def_E_R}
E_R= r^2 \frac{1}{2}\mathbf{\tilde{J}}\cdot \mathbf{\tilde{M}}^{-1}\cdot \mathbf{\tilde{J}}  
\end{equation}
denote the rotational kinetic energy. For $r>0$, $E_R$ is strictly positive.
The discriminant of $F(\lambda)$ is then  

\begin{equation} \label{charged_eq:discrimF}
\Delta= 4E E_R +\tilde{V}^2.
\end{equation}
The equation of vanishing discriminant, $\Delta=0$, defines a double cone in the space $(E,\tilde{V},E_R)$, see Fig.~\ref{charged_fig:lambda_cases}a, of which only the cone which has $E_R\ge 0$ is relevant. 
For a fixed value of $E_R$, the zero discriminant defines a parabola in the $(E,\tilde{V})$ plane (see Fig.~\ref{charged_fig:lambda_cases}b). For $E_R\to 0$, the parabolas collapse to the negative $E_R$-axis. 
For $E_R\to \infty$, the parabolas approach the $\tilde{V}$-axis. For fixed $E_R>0$, the parabola together with the coordinate axes  divides the $(E,\tilde{V})$-plane into the six regions  I, IIa, IIb, IIIa, IIIb and  IV marked in Fig.~\ref{charged_fig:lambda_cases}b. 
In the first quadrant of the $(E,\tilde{V})$-plane (region I), $F$ has real roots of which one is positive and the other is negative and $F$ has a minimum that is attained at a positive value of $\lambda$. 
The second quadrant consists of regions IIa and IIb. In region IIa, $F$ has two negative real roots and a maximum attained at a negative value. In region IIb the roots are complex but $F$ still has a maximum attained at a negative value. 
The third quadrant consists of regions IIIa and IIIb where in region IIIb  the roots are again complex and $F$ still has a maximum which is however now attained at a positive $\lambda$. In region IIIa there are two positive real  roots and $F$ has a maximum that is attained for a positive $\lambda$. 
In region IV in the fourth quadrant $F$ has two real roots of which one is negative and one is positive and $F$ has a minimum that is attained at a negative value of $\lambda$. 

We deduce 
\begin{theorem} \label{charged_thm:Hillregionfirst}
Let   
\begin{equation}
\lambda_\pm :=  %\frac{\tilde{V} }{2E} \pm \sqrt{\frac{\tilde{V}^2}{4E^2} + \frac{r^2 \mathbf{\tilde{J}}\cdot \mathbf{\tilde{M}}^{-1}\cdot \mathbf{\tilde{J}}}{2E} } = 
 \frac{\tilde{V} }{2E} \pm \sqrt{\frac{\Delta}{4E^2}  }
\end{equation}
be the roots of $F(\lambda)$ defined in \eqref{charged_eq:F(lambda)} and $r>0$.  Then we have:
\begin{enumerate}
\item For energies $E>0$, every shape-orientation point belongs to the Hill region.  
\item For energies $E<0$, a shape-orientation point  with  $\tilde{V}>0$ is for no orientation in the Hill region.  
\item For energies $E<0$, a shape-orientation point  with $\tilde{V}<0$ is in the Hill region if and only if the roots $\lambda_\pm$ are real.  
\end{enumerate}
\end{theorem}

It follows from Theorem~\ref{charged_thm:Hillregionfirst} that case 3 where $E,\tilde{V}<0$ requires more attention. For $E,\tilde{V}<0$, the condition to have real roots, $\Delta\ge 0$, is 
equivalent to
\begin{equation} \label{charged_eq:Hill_cond_dil_reduced}
\frac{\tilde{V}}{ 2 \sqrt{ \frac12 \mathbf{\tilde{J}}\cdot \mathbf{\tilde{M}}^{-1}\cdot \mathbf{\tilde{J}} }} \le - \sqrt{ -  E  r^2     }. % =: - \sqrt{\nu }\,.
\end{equation}
\rem{
or equivalently 
\begin{equation}
\frac{\tilde{V}}{ 2 \sqrt{ \tilde{E}_R}} \le - \sqrt{ \nu    } \,,
\end{equation}
where $ \tilde{E}_R =   \frac12 \mathbf{\tilde{J}}\cdot \mathbf{\tilde{M}}^{-1}\cdot \mathbf{\tilde{J}} $ is the `normalised' rotational energy  and $\nu:= -Er^2$.
} % end rem
\begin{figure}
\begin{center}
\raisebox{5cm}{a)}
\includegraphics[angle=0,width=7cm]{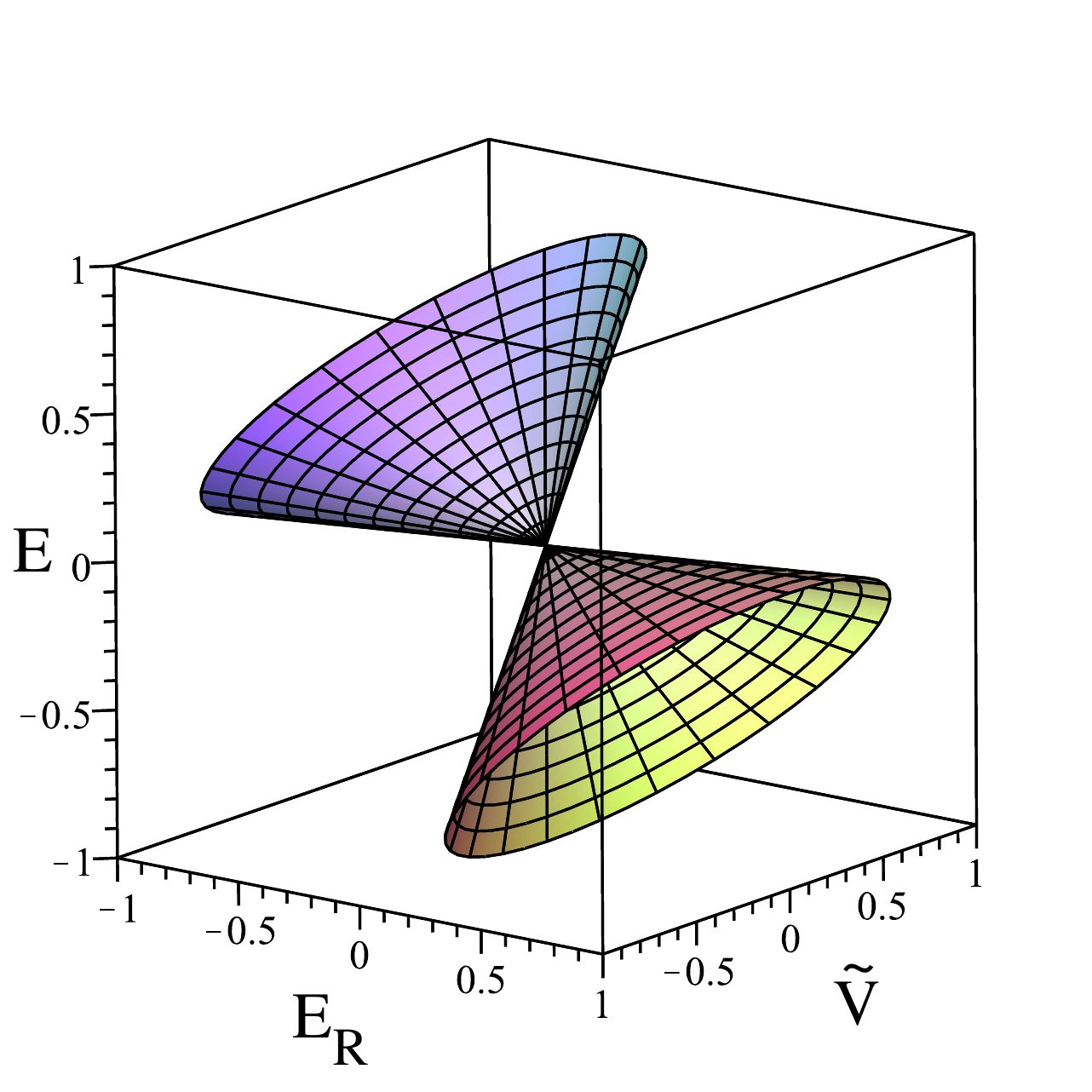}
\raisebox{5cm}{b)}
\raisebox{1cm}{
\includegraphics[angle=0,width=5cm]{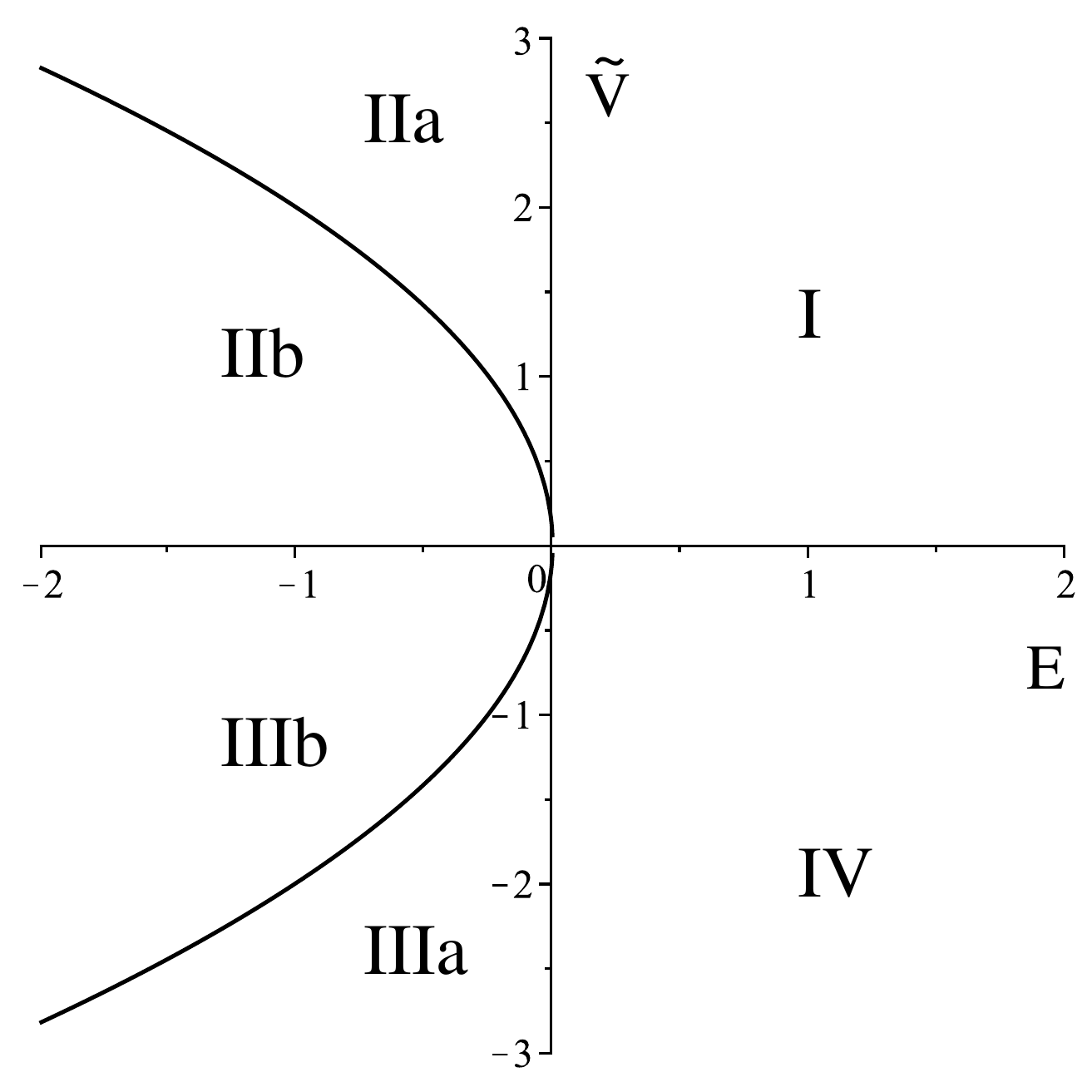}}\\[1.5ex]
\raisebox{3.8cm}{I}
\includegraphics[angle=0,width=4cm]{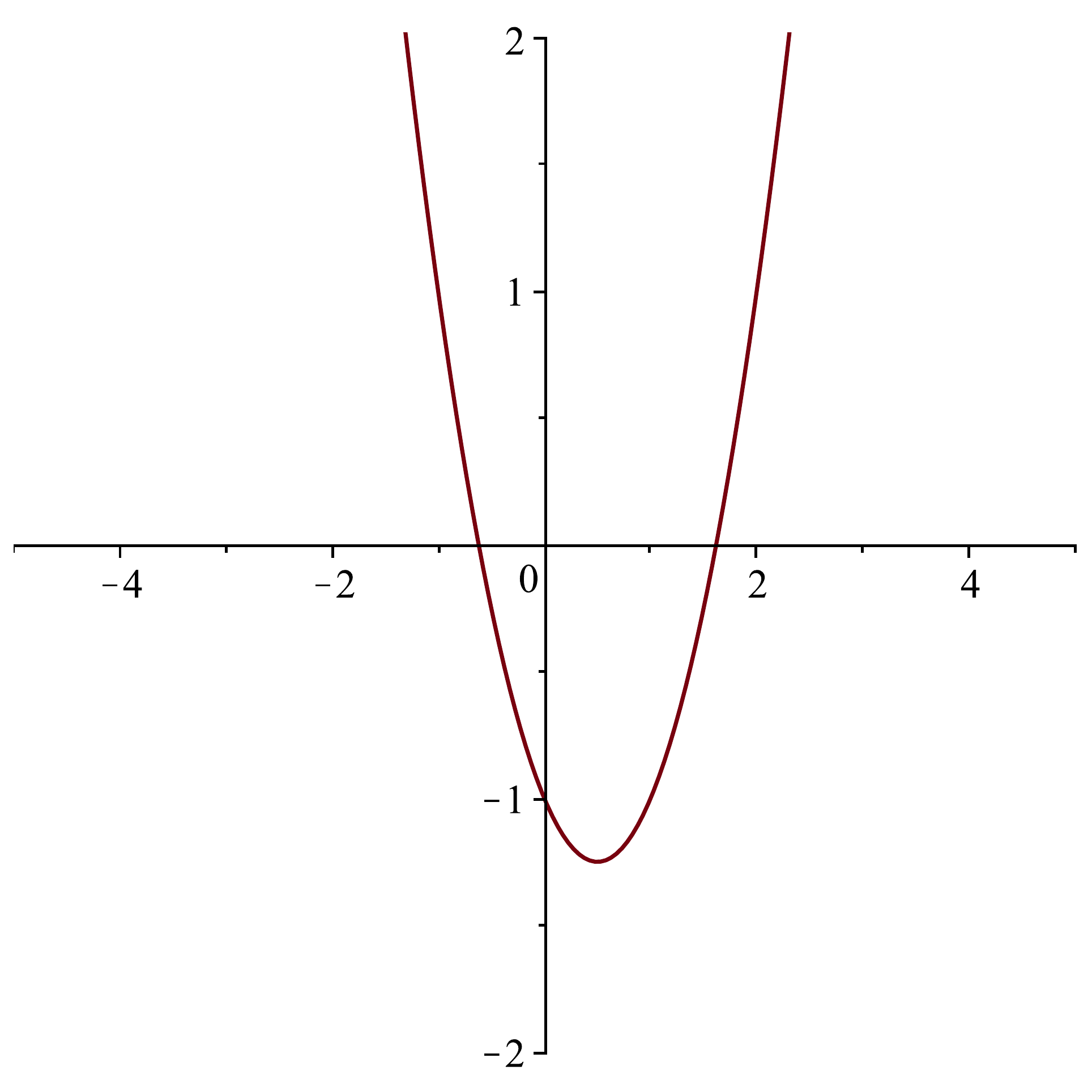}
\raisebox{3.8cm}{IIa}
\includegraphics[angle=0,width=4cm]{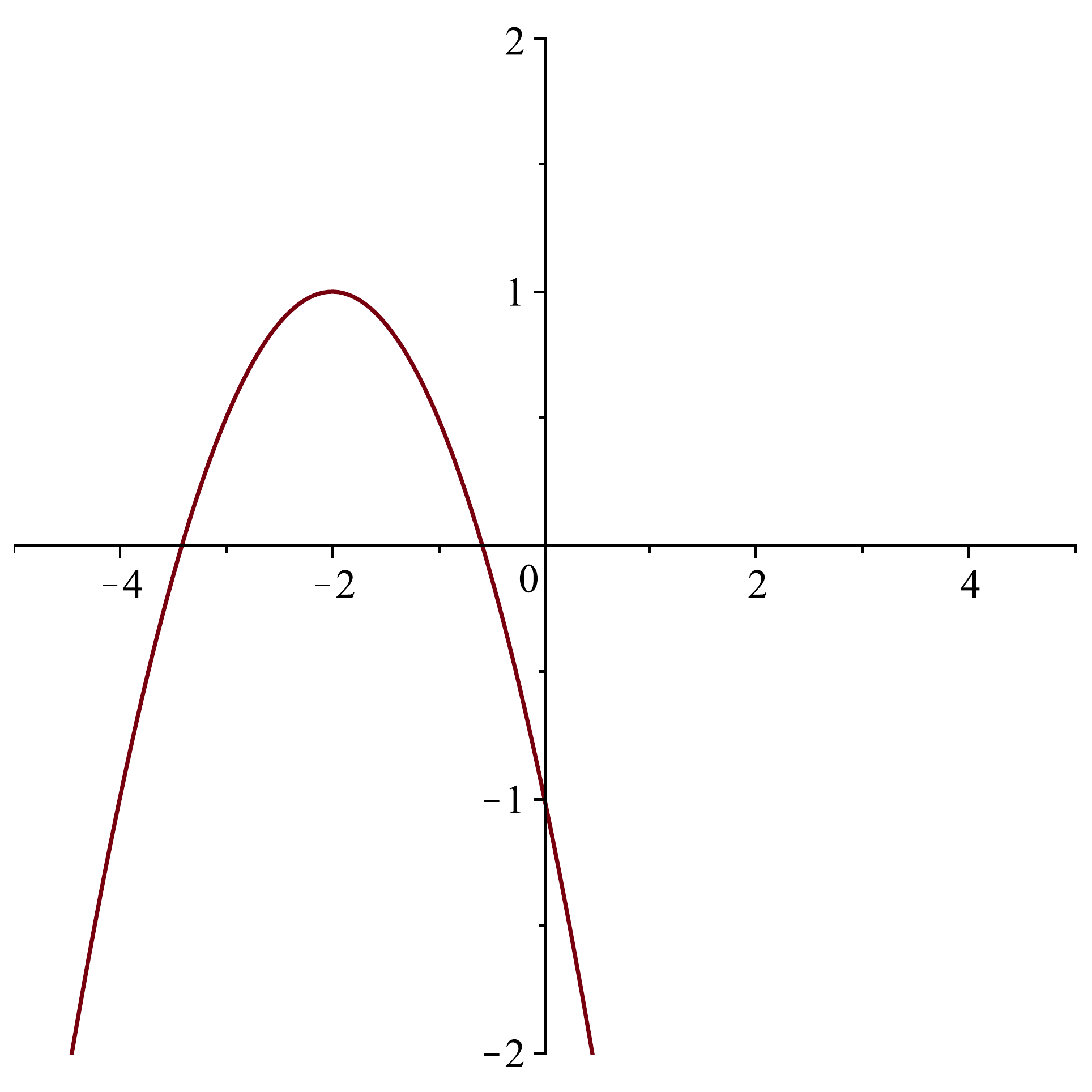}
\raisebox{3.8cm}{IIb}
\includegraphics[angle=0,width=4cm]{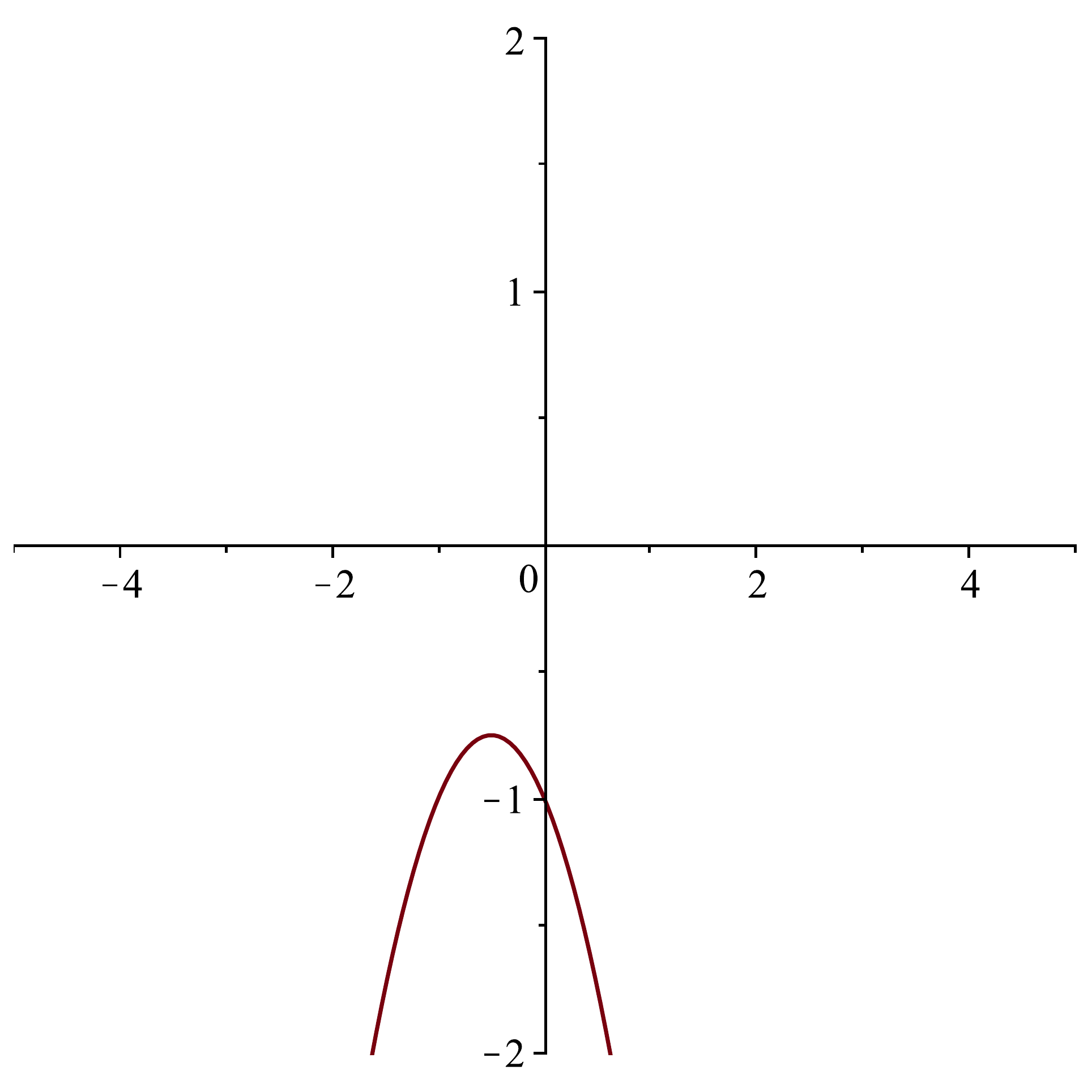}\\
\raisebox{3.8cm}{IIIb}
\includegraphics[angle=0,width=4cm]{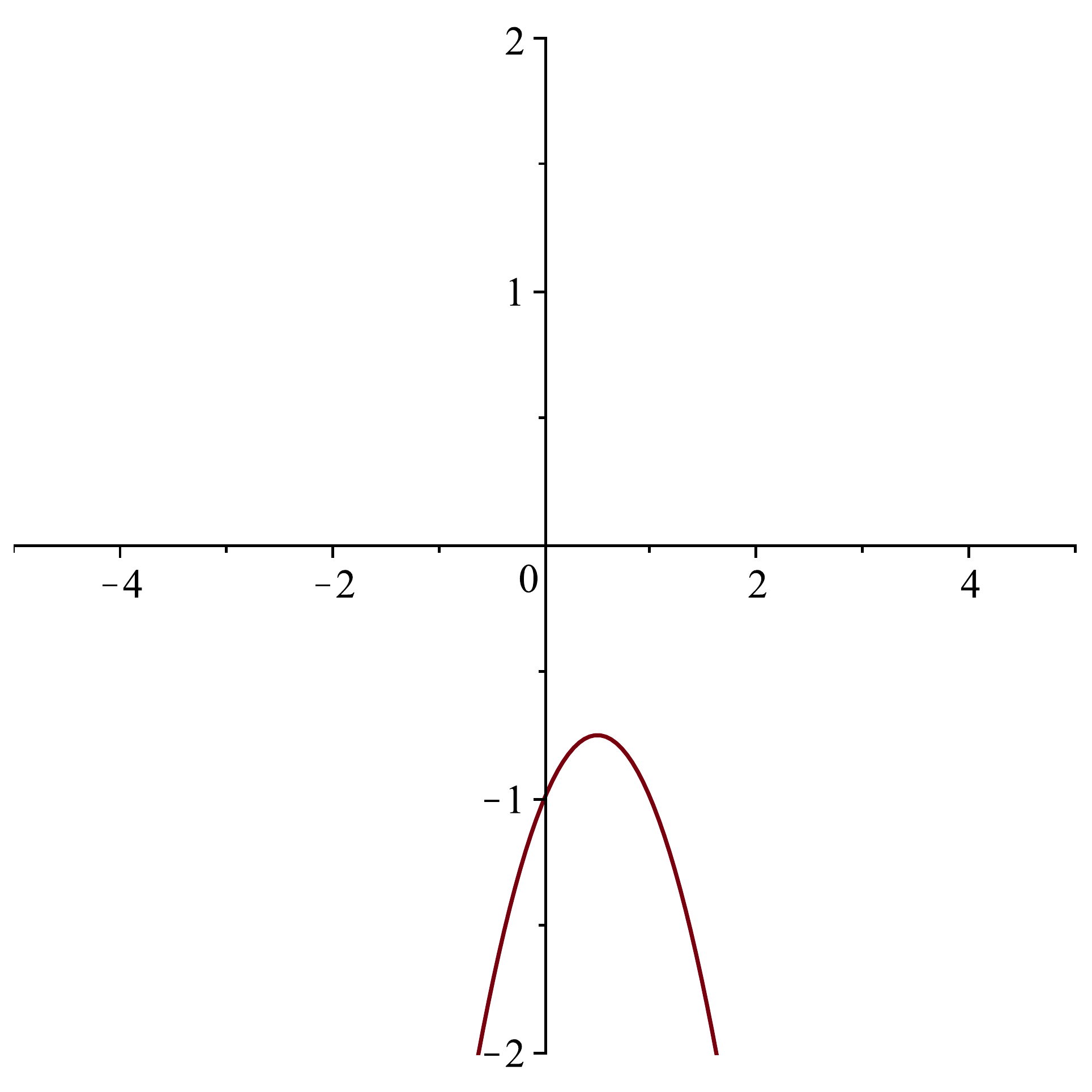}
\raisebox{3.8cm}{IIIa}
\includegraphics[angle=0,width=4cm]{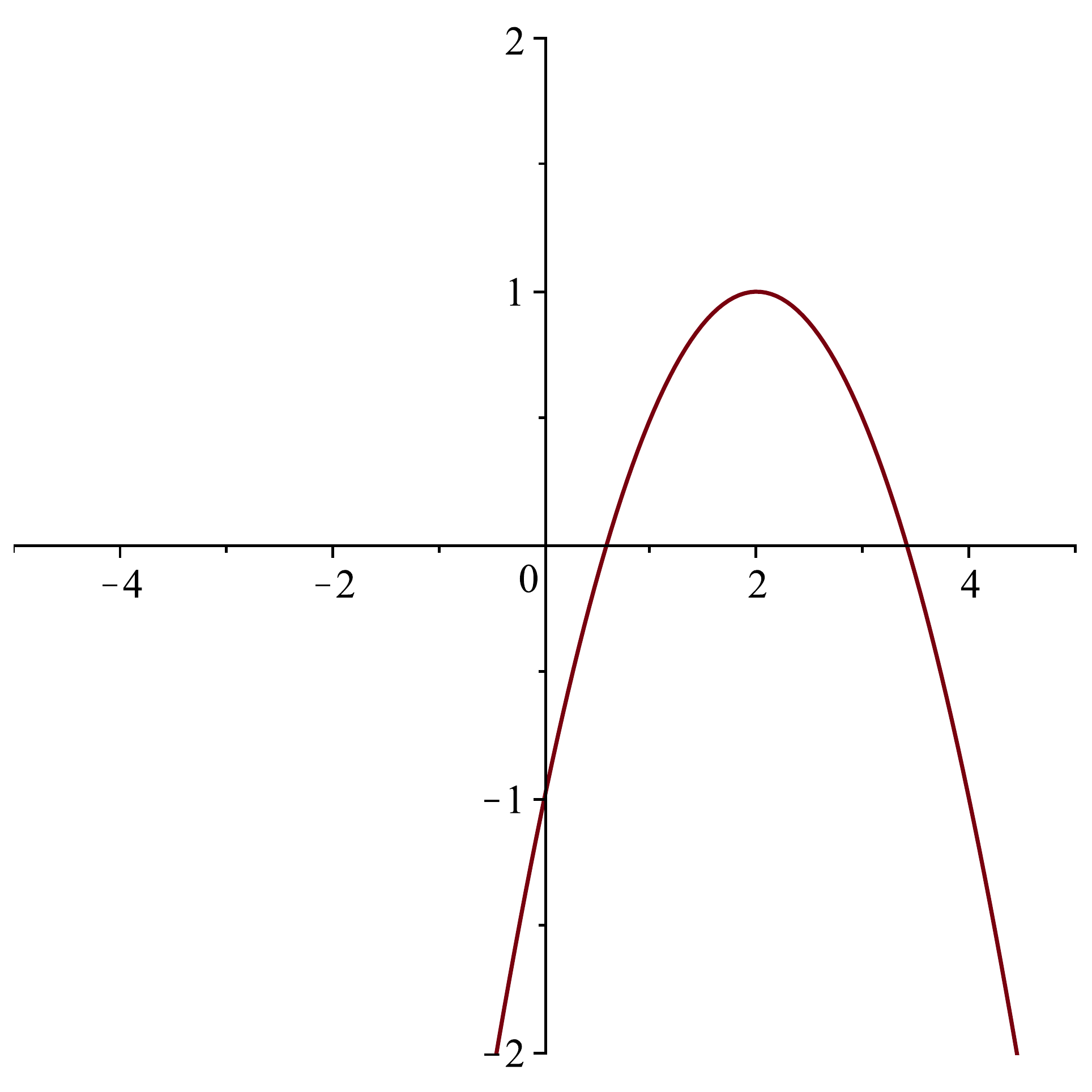}
\raisebox{3.8cm}{IV}
\includegraphics[angle=0,width=4cm]{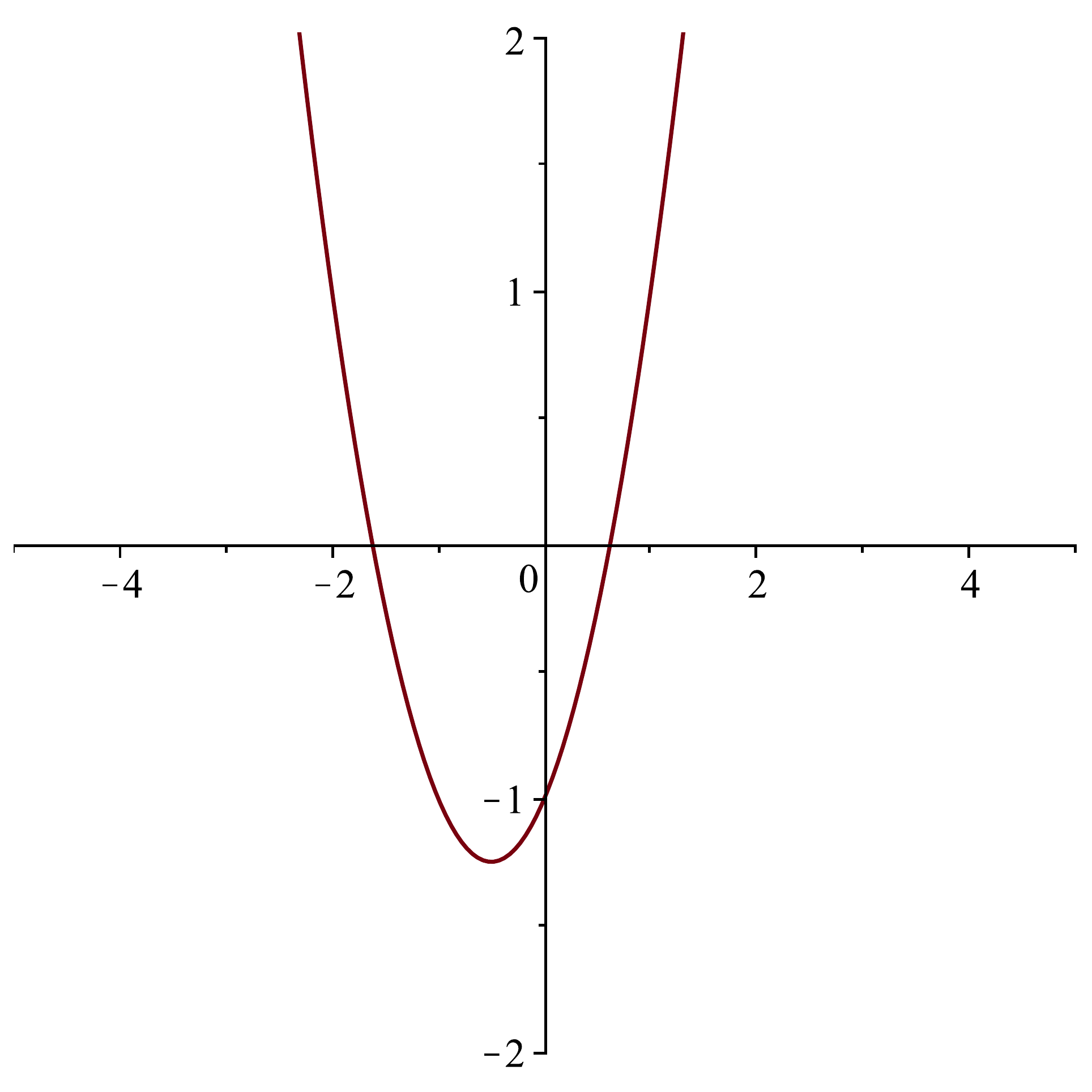}
\end{center}
\caption{\label{charged_fig:lambda_cases}
Surface of vanishing discriminant of $F(\lambda)$ defined in \eqref{charged_eq:F(lambda)}  in the space $(E,\tilde{V},E_R)$ where $E_R= r^2 \frac{1}{2}\mathbf{\tilde{J}}\cdot \mathbf{\tilde{M}}^{-1}\cdot \mathbf{\tilde{J}} $  is the rotational energy (a). In panel (b) a cut for fixed $E_R=1$ in (a) is shown. Together with the coordinate axes it divides the $(E,\tilde{V})$ plane into six regions I, IIa, IIb, IIIa, IIIb and IV. The remaining panels show the graphs of  
$F:\lambda\mapsto F(\lambda)$ for values of $(E,\tilde{V})$ in these different regions.
}
\end{figure}
With the left hand side of \eqref{charged_eq:Hill_cond_dil_reduced} we can define the function 
\begin{equation}\label{charged_eq:def_bif_function_Hill}
\tilde{\Qtr } \times S_1^2 \to \R , \quad (\tilde{q},\mathbf{\tilde{J}}) \mapsto \frac{\tilde{V}(\tilde{q})}{ 2 \sqrt{ \frac12 \mathbf{\tilde{J}}\cdot (\mathbf{\tilde{M}}(\tilde{q}))^{-1}\cdot \mathbf{\tilde{J}} }} .
\end{equation}
The Hill region for given $E<0$ and $r>0$ is then then the sublevel set  of this function in  $\tilde{Q}\times S_1^2$ for the value $- \sqrt{ -  E  r^2     }$.
To study the bifurcations of the Hill region we need to study the critical points of
the function \eqref{charged_eq:def_bif_function_Hill}.  The critical values can be expressed in terms of 

\begin{equation} \label{eq:def_nu}
\nu:=- E r^2
\end{equation}
which we can hence view as the bifurcation parameter.

\rem{
Following Remark~\ref{lemma:M_tilde_singular}  the moment of inertia tensor  $ \mathbf{\tilde{M}} (\tilde{q}) $ has a diabolic point at shapes $\tilde{q}$ where the first and the second principal moments of inertia are equal. 
 Critical points of the function \eqref{charged_eq:def_bif_function_Hill}  first diabolic points. 
 
 First consider shapes for which the principal moments of inertia are mutually different?
 Then consider the case where the first two are equal.
 What are the corresponding critical $\mathbf{\tilde{J}}$?
}

Finding the critical points of the function defined in \eqref{charged_eq:def_bif_function_Hill} simplifies to some extent. 
As $\tilde{V}$ does not depend on $\mathbf{\tilde{J}}$ the $\mathbf{\tilde{J}}$ components of the critical points of the function in 
\eqref{charged_eq:def_bif_function_Hill} 
are given by 
the critical points of the function
\begin{equation}\label{charged_eq:defGEuler}
S_1^2 \to \R,\quad 
\mathbf{\tilde{J}} \mapsto % G(\mathbf{\tilde{J}}) := 
\frac{1}{2}\mathbf{\tilde{J}}\cdot [\mathbf{\tilde{M}} (\tilde{q})]^{-1}\cdot \mathbf{\tilde{J}} 
\end{equation}
for fixed shape $\tilde{q}\in \tilde{\Qtr }$. 
The critical points   of \eqref{charged_eq:defGEuler}  on $S_1^2$ are $\mathbf{\tilde{J}}$ where the gradients of  $\frac{1}{2}\mathbf{\tilde{J}}\cdot [\mathbf{\tilde{M}} (\tilde{q})]^{-1}\cdot \mathbf{\tilde{J}} $ and the function defining the constraint  $ \mathbf{\tilde{J}}  \cdot \mathbf{\tilde{J}} =1 $ with respect to  $\mathbf{\tilde{J}}$ are parallel. This means that  $\mathbf{\tilde{J}}$  is an eigenvector of $\mathbf{\tilde{M}} (\tilde{q})$. 

\rem{
Now one would like proceed by solving this condition for $\tilde{q}$ and plug this into $ \frac{\tilde{V}(\tilde{q})}{ 2 \sqrt{ \frac12 \mathbf{\tilde{J}}\cdot (\mathbf{\tilde{M}}(\tilde{q}))^{-1}\cdot \mathbf{\tilde{J}} }}$ and find the critical values of 

\begin{equation}\label{charged_eq:def_VtildesqrtMk}
\tilde{q} \mapsto  \sqrt{\tilde{M}_k (\tilde{q}) } \, \tilde{V} (\tilde{q}) , \quad k =1,2,3, 
\end{equation}

By the Implicit Function Theorem this fails at the points where the derivative with respect  $\mathbf{\tilde{J}}$
} % end rem

A given shape $\tilde{q}$ can be considered to define a rigid body with moment of inertia tensor $\mathbf{\tilde{M}}$ (where the trace of $\mathbf{\tilde{M}}$ is equal to $1/2$).
%with principal moments of inertia given by the diagonal elements in \eqref{charged_eq:inertiaDragt}.  
Apart from the shape corresponding to the  diabolic point 
the shapes $\tilde{q}$  give mutually different  principal moments and the corresponding rigid body is asymmetric. 
%As the  principal moments of inertia are given by $ \sin\frac{\chi }{2} $, $ \cos\frac{\chi }{2} $, and $1$, we see that 
%the principal moments are different for $\chi\ne \pi/2$ 
Let us at first focus on this case. We will come back to the diabolic point below.
For shape coordinates for which the moment of inertia tensor is diagonal like Dragt's coordinates,
the critical values of $\mathbf{\tilde{J}}$ on the normalised angular momentum sphere $S^2_1$ are then
given by $ \mathbf{\tilde{J}} $ equal to $(\pm 1, 0,0)$,  $(0, \pm 1, 0)$ or  $(0,0,\pm 1)$. 
The critical (non-diabolic) shapes $\tilde{q}$ are then critical points of the functions
\begin{equation}\label{charged_eq:def_VtildesqrtMk}
\tilde{q} \mapsto  \sqrt{\tilde{M}_k (\tilde{q}) } \, \tilde{V} (\tilde{q}) , \quad k =1,2,3, 
\end{equation}
where the $\tilde{M}_k(\tilde{q})$ are the dilation reduced principal moments of inertia.
If we use Dragt's coordinates $(\chi,\psi)$ as shape coordinates then the critical shape-orientations points $(\tilde{q} ,\mathbf{\tilde{J}})$ 
are given by  the critical points of the functions
\begin{equation} \label{charged_eq:ciriticalfunctions}
(\chi,\psi) \mapsto \left\{ 
\begin{array}{ccc}
\tilde{V}   \sin\frac{\chi }{2} & \text{for} & \mathbf{\tilde{J}} = (\pm 1,0,0) \\
\tilde{V}   \cos\frac{\chi }{2} & \text{for} & \mathbf{\tilde{J}} = (0,\pm 1,0) \\
\tilde{V}    & \text{for} & \mathbf{\tilde{J}} = (0,0,\pm 1) \\
\end{array}
\right.
\end{equation}
(see \eqref{charged_eq:inertiaDragt}).

\rem{
If  $\tilde{q}$ equals the diabolic point then the first and the second principal moments of inertia are equal and the corresponding rigid body is symmetric. Let us use Jacobi coordinates $(\rho_1,\rho_2,\phi)$ to analyse this case (recall that the Jacobi coordinates are, as opposed to Dragt's coordinates, smooth at the diabolic points, see Remark~\ref{lemma:M_tilde_singular}).
At the diabolic point $\rho_1=\rho_2=1/\sqrt{2}$ and $\phi=\pi/2$. The moment of inertia tensor $\mathbf{\tilde{M}} (\tilde{q})$ is diagonal with diagonal elements $[1/\sqrt{2},1/\sqrt{2},1/2]$ (see \eqref{eq:M_Jacobi}). The function \eqref{charged_eq:defGEuler} becomes
$$
S_1^2 \to \R,\quad 
\mathbf{\tilde{J}} = ( {\tilde{J}}_1, {\tilde{J}}_2, {\tilde{J}}_3)
\mapsto % G(\mathbf{\tilde{J}}) := \frac{1}{2}\mathbf{\tilde{J}}\cdot [\mathbf{\tilde{M}} (\tilde{q})]^{-1}\cdot \mathbf{\tilde{J}}
(   {\tilde{J}}_1^2 + {\tilde{J}}_2^2  )/\sqrt{2} + {\tilde{J}}_3^2/2.
$$
It has critical points $(0,0,\pm 1)$ and the critical circle $ {\tilde{J}}_3=0$, ${\tilde{J}}_1^2 + {\tilde{J}}_2^2=1$.
} % end rem

In the examples in Sec.~\ref{charged_sec:Examples} we will visualise the Hill regions defined according to Definition~\ref{charged_def:Hillregion} in terms of their projection to the shape space  
$\tilde{\Qtr }$.
These projections can be understood in terms of the contours of the functions \eqref{charged_eq:def_VtildesqrtMk} (or \eqref{charged_eq:ciriticalfunctions} when we use Dragt's coordinates) on $\tilde{\Qtr }$ as follows.  Consider a fixed negative value of the energy $E<0$, a fixed value of the magnitude of the angular momentum $r>0$ and a given shape  $\tilde{q} \in \tilde{\Qtr }$.  Let us rewrite the condition for a shape-orientation $(\tilde{q},\mathbf{\tilde{J}})$ to be in the Hill region given by the inequality \eqref{charged_eq:Hill_cond_dil_reduced}  as 

 \begin{equation}\label{charged_eq:inequality_E_R}
- \frac{\tilde{V}(\tilde{q})}{ 2      \sqrt{ -  E  r^2  }  } \ge     \tilde{E}_R (\mathbf{\tilde{J}}), 
\end{equation}
where 

\begin{equation}
 \tilde{E}_R (\mathbf{\tilde{J}}) := 
\frac{1}{2} \mathbf{\tilde{J}}\cdot [ \mathbf{\tilde{M}}(\tilde{q})]^{-1} \cdot \mathbf{\tilde{J}}  
\end{equation}
is the  `normalised'  rotational energy which gives the rotational energy of a rigid body with moment of inertia tensor $\mathbf{\tilde{M}}(\tilde{q})$ rotating with a total angular momentum of unit magnitude. In Fig.~\ref{charged_fig:Euler} we show the level sets of $ \tilde{E}_R$ of this rigid body on the unit angular momentum sphere $S_1^2$ for the case of distinct principal moments of inertia. 
On $S_1^2$ the function $  \tilde{E}_R$ has the three critical values 
\begin{equation}
 \tilde{E}_{R\,3} := \frac{1} {2 \tilde{M}_3 (\tilde{q}) } <  \tilde{E}_{R\,2} :=  \frac{1} {2 \tilde{M}_2 (\tilde{q}) } <  \tilde{E}_{R\,1} :=  \frac{1} {2 \tilde{M}_1 (\tilde{q}) }
\end{equation}
corresponding to a minimum, a saddle and a maximum, respectively. If $- \tilde{V}(\tilde{q}) / ( 2      \sqrt{ -  E  r^2  }  ) < \tilde{E}_{R\,3}  $ then the inequality \eqref{charged_eq:inequality_E_R} is not satisfied for any point $\mathbf{\tilde{J}} \in S_1^2$.  If $ \tilde{E}_{R\,3} <- \tilde{V} (\tilde{q}) / ( 2      \sqrt{ -  E  r^2  }  ) < \tilde{E}_{R\,2} $ then   the inequality \eqref{charged_eq:inequality_E_R} is satisfied for points $\mathbf{\tilde{J}} \in S_1^2$ in the `caps' that are given by (closed) disk neighbourhoods  of the points $\mathbf{\tilde{J}} = (0,0,\pm1)$ (see Fig.~\ref{charged_fig:Euler}). 
 If $ \tilde{E}_{R\,2} <- \tilde{V} (\tilde{q}) / ( 2      \sqrt{ -  E  r^2  }  ) < \tilde{E}_{R\,1} $ then   the inequality \eqref{charged_eq:inequality_E_R} is satisfied for orientations  $\mathbf{\tilde{J}} \in S_1^2$ in the `ring' that is obtained from excluding two (open) disk neighbourhoods near the poles $\mathbf{\tilde{J}} = (\pm1,0,0)$ on the unit sphere $\mathbf{\tilde{J}} \in S_1^2$ . 
If $ \tilde{E}_{R\,1} <- \tilde{V} (\tilde{q}) / ( 2      \sqrt{ -  E  r^2  }  ) $ then the inequality \eqref{charged_eq:inequality_E_R} is satisfied for any orientation $\mathbf{\tilde{J}} \in S_1^2$.  The latter two cases are also illustrated in Fig.~\ref{charged_fig:Euler}. To illustrate the projection of the Hill region for energy $E$ and magnitude $r$ of the angular momentum  to the shape space we will colour 
each shape $\tilde{q}\in \tilde{\Qtr } $ if the sublevel set of $ \tilde{E}_R (\mathbf{\tilde{J}})$ on $S_1^2$    for the value $- \frac{\tilde{V}(\tilde{q})}{ 2      \sqrt{ -  E  r^2  }  }$ is empty in dark grey, two caps in blue, a ring in red and the full sphere in green.

%%%
% Euler top; temporarily excluded because of the large file sizes
%%%
%\rem{
 \begin{figure}
\begin{center}
\includegraphics[angle=0,width=12cm]{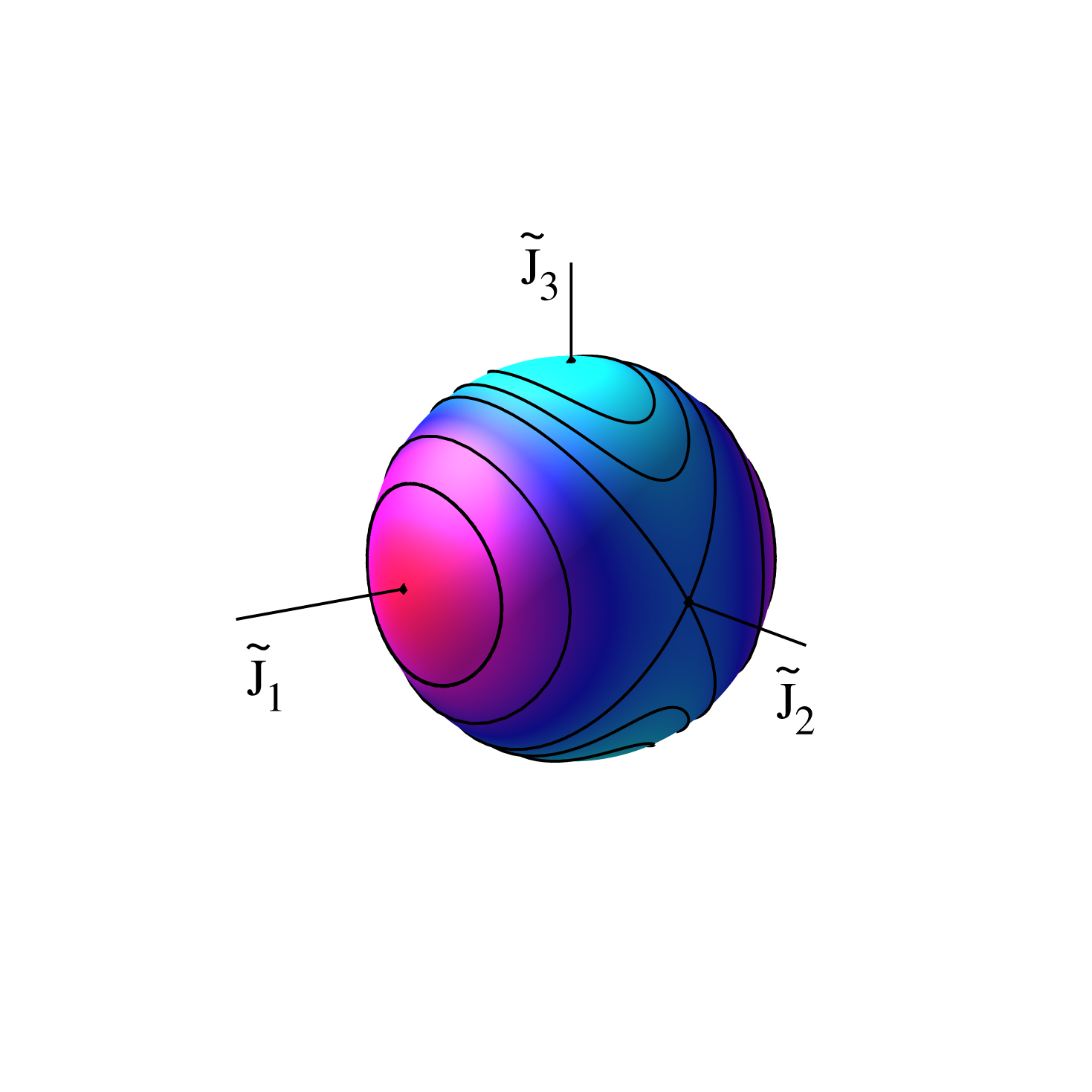}\\
\hspace*{1.8cm}
\includegraphics[angle=0,width=3cm]{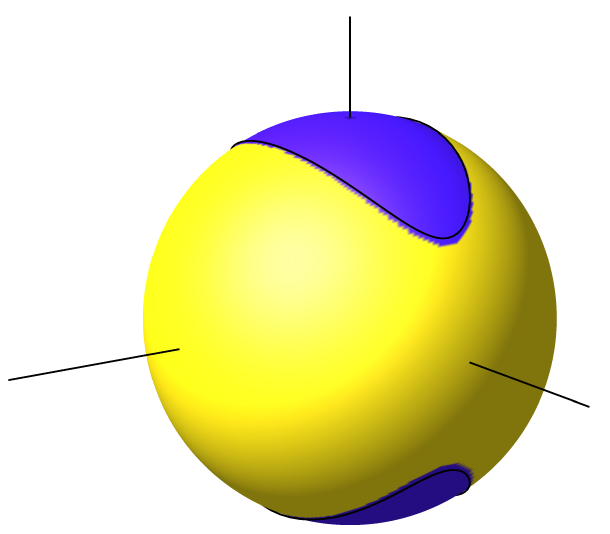}
\includegraphics[angle=0,width=3cm]{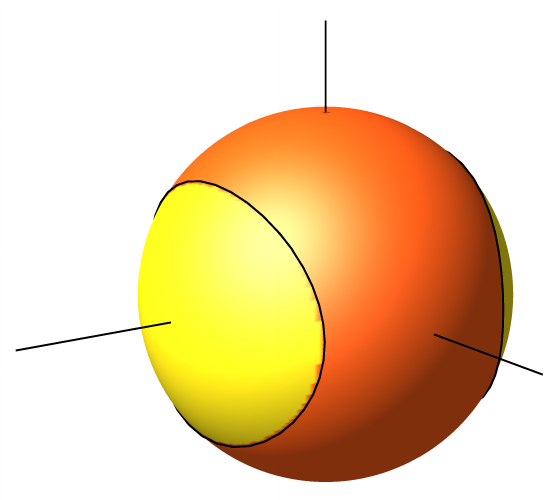}
\includegraphics[angle=0,width=3cm]{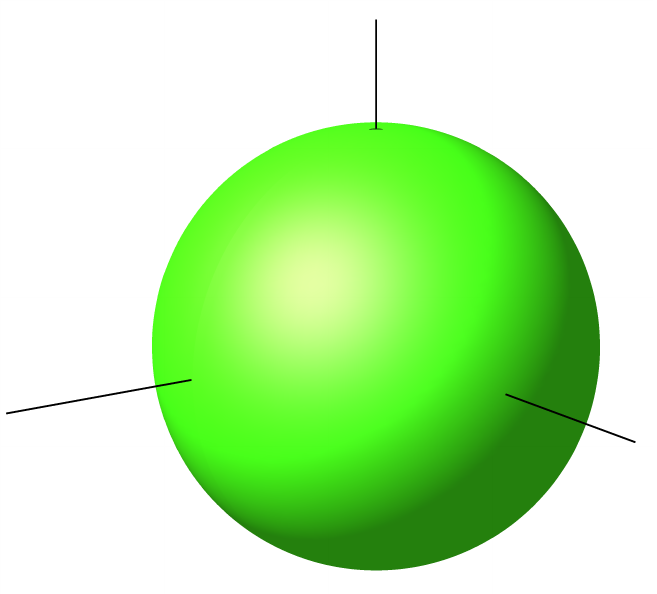}\\
\includegraphics[angle=0,width=14cm]{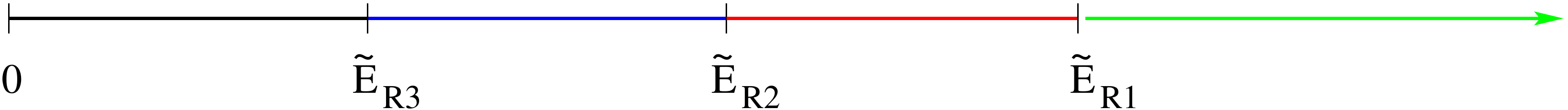}\\
\end{center}
\caption{\label{charged_fig:Euler}
Top panel: level sets of the function defined in \eqref{charged_eq:defGEuler}  on the normalised angular momentum sphere $S_1^2$ for a fixed shape corresponding to an asymmetric Euler top (with principal moments of inertia $1/3$, $2/3$ and $1$). The function value increases as the colour goes from light blue to red.
Lower panel: dependence of the accessible region on the normalised angular momentum sphere  
 on the value of $- \tilde{V} / ( 2      \sqrt{ -  E  r^2  }  )$ (see the discussion at the end of Sec.~\ref{charged_sec:Hillregions}).
}
\end{figure}
%} % end rem 

Let us now come back to the diabolic point, and let us reconsider our approach to compute the critical points of \eqref{charged_eq:def_bif_function_Hill} outlined above.
The critical points are elements of $\tilde{\Qtr } \times S_1^2$. We first determined the components  $ \mathbf{\tilde{J}}\in  S_1^2$ of the critical points and plugged these in  \eqref{charged_eq:def_bif_function_Hill} to define the functions \eqref{charged_eq:def_VtildesqrtMk} whose critical points gave us the components $\tilde{q}\in \tilde{\Qtr }$ of \eqref{charged_eq:def_bif_function_Hill}. Here we made use of the fact that we can write the critical $ \mathbf{\tilde{J}}\in  S_1^2$ as functions of $\tilde{q}\in \tilde{\Qtr }$. 
This tacit use of the Implicit Function Theorem breaks down at the diabolic point where the moments of inertia tensor is degenerate. This is also reflected by the functions \eqref{charged_eq:def_VtildesqrtMk} not being differentiable at the diabolic point for $k=1,2$. We expect to see signatures of this singularity also in the visualization of the Hill regions in terms of their projection to the shape space when the contours of the function \eqref{charged_eq:def_bif_function_Hill} reach $(\tilde{q},\mathbf{\tilde{J}})\in\tilde{\Qtr } \times S_1^2 $ where $(\tilde{q}$ is the diabolic point and $\mathbf{\tilde{J}}\in S_1^2$ is an eigenvector of $\mathbf{\tilde{M}}(\tilde{q})$
 in the two-dimensional eigenspace for the eigenvalue $\tilde{M}_1=\tilde{M}_2$.  As this point is not a proper critical point of  \eqref{charged_eq:def_bif_function_Hill} 
 we will refer to it as \emph{pseudo critical point}.

%%%%%%%%%%%%%%%%%%%%%%
\section{Critical points and bifurcations of the Hill region}
\label{charged_sec:relative_equilibria}

In the two preceding papers \cite{HoveijnWaalkensZaman2019} and \cite{CPI2021} we have studied critical points of the map of integrals for charged three-body systems. The integral manifolds, i.e. the joint level sets of the integrals, are expected to change their topology at the corresponding critical values. 
Due to the non-compact nature of the integral manifolds  we  had to distinguish between (finite) critical points given by relative equilibria and critical points at infinity. 
%The finite critical points are given by relative equilibria and in the first paper  \cite{HoveijnWaalkensZaman2019}  we studied how these are related to central configurations.
A natural question that arises is how the critical points of the map of integrals are related to the critical points discussed in the previous section.
It is to be expected that bifurcations of the integral manifolds also manifest themselves in bifurcations of the Hill regions. Special attention require however critical points of the map of integrals that correspond to points on the boundary of the shape space as we explicitly excluded the boundary from the definition of the shape space and hence of the definition from the Hill regions (see Definitions~\ref{charged_def:shap_space}  and \ref{charged_def:Hillregion}, respectively). 

In particular the projections of relative equilibria to configuration space can be collinear or non-collinear. The reduction of the $\SO(3)$ symmetry underlying the definition of the Hill region in this paper
is singular at collinear configurations even though the corresponding relative equilibria do not have isotropy. The reason for the latter is  that at these relative equilibria the momenta are not collinear with the position vectors of the bodies.
The singularity of the reduction is due to the fact that the reduction is built on a cotangental lift of the reduction of the configuration space $\SO(3)$ symmetry. 
There should be a reduction not making use of the cotangent bundle property of the unreduced problem that does not have this problem. Whereas the reduction used in this paper leads a splitting of the rotational and internal degrees of freedom (which however inevitably remain coupled \cite{Guichardet1984}) and preserves the cotangent bundle structure of the phase space for the internal degrees of freedom this is probably no longer the case for a reduction for which the collinear relative equilibria are not singular. These properties of the reduction used in this paper are however in our opinion crucial for a meaningful definition of the Hill region. 

So we cannot expect that relative equilibria that project to collinear configurations in configuration space are related to critical points discussed in the previous section. The same however holds for critical points at infinity. For three bodies,  a critical point at infinity involves one co-rotating pair of two bodies with the third body at rest at infinity (see Sec.~\ref{sec:critptsinfty} below). After scaling  such a configuration by the dilation transformation  such that the polar moment of inertia $I$ becomes $1$ the co-rotating pair appears to collide. The corresponding shapes are hence also contained in the boundary of the shape space.

The only question that remains is then how the critical points in Sec.~\ref{charged_sec:Hillregions} are related to non-collinear relative equilibria. 
The latter are the stationary solutions of the reduced equations of motion \eqref{charged_Equations_motion_old}.
They satisfy the equations \cite{KozinRobertsTennyson00}

\begin{eqnarray}
\mathbf{J}\times \left( \mathbf{M}^{-1}\cdot \mathbf{J}\right)  &=&0,   \label{charged_eq:rel_equ_1} \\ 
p_{\mu } &=&\mathbf{J}\cdot \mathbf{A}_{\mu },  \label{charged_eq:rel_equ_2} \\ 
\frac{\partial }{\partial q_{\mu }}(\frac12 \mathbf{J}\cdot \mathbf{M}^{-1}\cdot \mathbf{J}+
V(q)) &=&0. \label{charged_eq:rel_equ_3}
\end{eqnarray}
Equation~\eqref{charged_eq:rel_equ_1}  implies that at relative equilibria the body angular momentum $\mathbf{J}$ is an eigenvector of the moment of inertia tensor, i.e. 
$\mathbf{J}$ is parallel to a principal axis. This means that
the three-body system is rotating about one of its principal axes.  
Using this fact the internal coordinates of relative equilibria can be found from the critical points of the \emph{effective potential}

\begin{equation}\label{charged_eq:def_V_eff}
V_{\text{eff}}(q)  :=  \frac{1}{2} \mathbf{J}\cdot [\mathbf{M} (q)]^{-1}\cdot \mathbf{J} + V(q)\,,
\end{equation}
where  $\mathbf{J}$ is a fixed vector of a given modulus $r$ parallel to a chosen principal axis.

For rotations about the first, second or third principal axis where $\mathbf{J}=(\pm r,0,0)$, $\mathbf{J}=(0,\pm r,0)$ or $\mathbf{J}=(0,0,\pm r)$, respectively, the effective potential becomes
\begin{equation}\label{charged_eq:def_V_eff_principal}
V_{\text{eff}}(q)  =  \frac{1}{2} \frac{r^2}{M_k(q)} + V(q)\,, \quad k=1,2,3.
\end{equation}
The critical internal coordinates satisfy
\begin{equation} \label{charged_eq:stationary_effective_potential}
\frac{\partial V} { \partial q_\mu} =  \frac12 \frac{r^2}{M_k^2}  \frac{\partial M_k} { \partial q_\mu} .
\end{equation}
Filling in the critical internal coordinates into \eqref{charged_eq:rel_equ_2} then gives the momenta of the relative equilibria.

\begin{theorem} \label{charged_theorem_rel_equilibria}
The critical points of the function  \eqref{charged_eq:def_bif_function_Hill} are 
%diabolic points (see Remark~\ref{lemma:M_tilde_singular})
%where $ \tilde{M}_1=\tilde{M}_2$ and hence points where $\tilde{M}_1$ and  $\tilde{M}_2$ are singular (see Remark~\ref{lemma:M_tilde_singular}) or 
shape-orientation points of non-collinear relative equilibria.
\end{theorem}

\begin{proof}
At the end of Sec.~\ref{charged_sec:Hillregions} we have seen that the $\mathbf{\tilde{J}}$ components of the critical shape-orientations points of \eqref{charged_eq:def_bif_function_Hill}
are eigenvectors of the shape dependent moment inertia tensor $\mathbf{\tilde{M}}$.  The shape components satisfy
%$\sqrt{ \tilde{M}_k } \,  \tilde{V} $, $k=1,2,3$, defined in \eqref{charged_eq:def_VtildesqrtMk}, i.e. 
%%
%%
%\begin{equation}
%0 = \frac{\partial}{ \partial \tilde{q}_\mu}
%\sqrt{ \tilde{M}_i } \,  \tilde{V}% =   \frac{\partial}{ \partial \tilde{q}_\mu} \sqrt{ M_i}  \,V  
%= \frac12 \frac{\tilde{V}}{\sqrt{\tilde{M}_k}}   \frac{\partial \tilde{M}_k} { \partial \tilde{q}_\mu} +   \sqrt{\tilde{M}_k} \,  \frac{\partial \tilde{V}} { \partial \tilde{q}_\mu} ,\quad \mu=1,2,
%\end{equation}

\begin{equation}
0 = \frac{\partial}{ \partial \tilde{q}_\mu}
\frac{ \tilde{V}}{\sqrt{\mathbf{\tilde{J}} \cdot [\mathbf{\tilde{M}} (\tilde{q})]^{-1} \cdot \mathbf{\tilde{J}}} }
= 
\frac{  \partial \tilde{V}/ \partial \tilde{q}_\mu}{\sqrt{\mathbf{\tilde{J}} \cdot [\mathbf{\tilde{M}} (\tilde{q})]^{-1} \cdot \mathbf{\tilde{J}}} }  -   
\frac{1}{2} \tilde{V} 
 \frac{\mathbf{\tilde{J}} \cdot  \frac{\partial}{ \partial \tilde{q}_\mu}  [\mathbf{\tilde{M}} (\tilde{q})]^{-1} \cdot \mathbf{\tilde{J}}}{\big(\mathbf{\tilde{J}} \cdot [\mathbf{\tilde{M}} (\tilde{q})]^{-1} \cdot \mathbf{\tilde{J}}\big)^{3/2}}    ,\quad \mu=1,2,
\end{equation}
or 

\begin{equation} \label{charged_eq: stationary_Hill}
\frac{\partial \tilde{V}} { \partial \tilde{q}_\mu}  -\frac{1}{2} \tilde{V} 
 \frac{\mathbf{\tilde{J}} \cdot  \frac{\partial}{ \partial \tilde{q}_\mu}  [\mathbf{\tilde{M}} (\tilde{q})]^{-1} \cdot \mathbf{\tilde{J}}}{\mathbf{\tilde{J}} \cdot [\mathbf{\tilde{M}} (\tilde{q})]^{-1} \cdot \mathbf{\tilde{J}}} =0  \,\quad \mu=1,2.
\end{equation}
We complete the shape space coordinates $\tilde{q}_\mu$, $\mu=1,2$, to coordinates on the internal space $\Qtr $ by adding the coordinate $\tilde{q}_3 = \lambda$ and noting that for each $q$ in $\Qtr$, there exist a $\lambda>0$ and $\tilde{q}\in \tilde{\Qtr}$ such that $d_\lambda(\tilde{q})=q$.

Dividing  \eqref{charged_eq: stationary_Hill} by $\lambda$ and
using the homogeneity in  \eqref{charged_eq:hom_VM_tilde}  and  $ \mathbf{J}=r  \mathbf{\tilde{J}}$  we see that 
the equivalence of \eqref{charged_eq:rel_equ_3} and \eqref{charged_eq: stationary_Hill} requires for $\mu=1,2$ that

\begin{equation}\label{charged_eq:virial_cond}
 -V  =  \mathbf{J} \cdot [\mathbf{M} (q)]^{-1} \cdot \mathbf{J}   .
\end{equation}
For $\mu=3$, i.e. $q_3=\lambda$,  we have  by using again the homogeneity in  \eqref{charged_eq:hom_VM_tilde}

\begin{equation} \label{charged_eq:virial_cond_1}
\frac{\partial V} { \partial q_3} = \frac{\partial V} { \partial \lambda} =  \frac{\partial} { \partial \lambda}  \frac{1}{\lambda} \tilde{V}  = -\frac{1}{\lambda^2} \tilde{V}=   -\frac{1}{\lambda} V
\end{equation}
and
\begin{equation} \label{charged_eq:virial_cond_2new}
%
 %\mathbf{M} (d_\lambda (q)) = \lambda^2 \mathbf{M}(q) 
%
\frac{\partial \mathbf{M}^{-1}} { \partial q_3} = \frac{\partial \mathbf{M}^{-1}} { \partial \lambda} =  \frac{\partial} { \partial \lambda}  \frac{1}{\lambda^2} \mathbf{\tilde{M}}^{-1}  = 
-\frac{2}{\lambda^3} \mathbf{\tilde{M}}^{-1} =
- \frac{2}{\lambda} \mathbf{M}^{-1} \,.
\end{equation}
Using \eqref{charged_eq:virial_cond_1}  and \eqref{charged_eq:virial_cond_2new} in  \eqref{charged_eq:rel_equ_3} we see that \eqref{charged_eq:rel_equ_3}  for $\mu=3$ to be satisfied requires again \eqref{charged_eq:virial_cond}.

\rem{
\begin{equation} \label{charged_eq:virial_cond_2}
 \frac12 \frac{r^2}{M_k^2}  \frac{\partial M_k} { \partial \lambda} =  \frac12 \frac{r^2}{M_k^2}  \frac{\partial } { \partial \lambda} \lambda^2 \tilde{M}_k
 =  \frac12 \frac{r^2}{M_k^2}   2 \lambda  \tilde{M}_k  
 =  \frac{r^2}{M_k^2}    \frac{1}{\lambda}  M_k .
\end{equation}
Equality of \eqref{charged_eq:virial_cond_1} and\eqref{charged_eq:virial_cond_2} is again equivalent to \eqref{charged_eq:virial_cond}.

} % end rem

It hence remains to be shown that \eqref{charged_eq:virial_cond} holds.   
The right hand side of  \eqref{charged_eq:virial_cond} can be identified with twice the kinetic energy of a relative equilibrium (use \eqref{charged_eq:rel_equ_2} in \eqref{charged_Hamiltonian} to see that the vibrational kinetic energy is vanishing at a relative equilibrium). Like the potential the kinetic energy is stationary at a relative equilibrium. The equality then follows from the 
Virial Theorem which says that for homogeneous potential of degree $-1$, twice  the time average of the kinetic energy equals minus the time average of the potential energy \cite{Clausius1870}. For relative equilibria, the time average is trivial as it involves averaging constant quantities.

\end{proof}

\rem{
\begin{corollary}
Relative equilibria occur in one-parameter families. 
\end{corollary}
} % end rem

\rem{

\begin{theorem} \label{charged_theorem_rel_equilibria}
The critical value  $\nu=-Er^2$ of a relative equilibrium involving rotation about the $k$th principal axis, $k=1,2,3$,  is given by
\begin{equation} \label{eq:critical_nu_relequil}
\nu = \frac12  \tilde{M}_k (\tilde{q})  \, \tilde{V}^2 (\tilde{q})  = \frac12  M_k ({q})  \, {V}^2 ({q}) ,
\end{equation}
where $\tilde{M}_k(\tilde{q})$ and  $ \tilde{V} (\tilde{q})$  are the dilation reduced principal moment of inertia and potential energy, respectively, of the  shape $\tilde{q}$ corresponding to the dilation scaled internal space configuration $q$ of the relative equillibrium. 
\end{theorem}

\begin{proof}
From \eqref{charged_eq:rel_equ_2} we see that the vibrational kinetic energy in \eqref{charged_Hamiltonian} is vanishing for a relative equilibrium. 
For a relative equilibrium involving rotation about the $k$th principal axis with magnitude $r$ of the angular momentum, the ro-vibrational energy \eqref{charged_Hamiltonian} leads to the energy equation
\[
E = \frac12 \frac{r^2}{M_k(q) } + V(q).
\]
By the Virial Theorem we have that the kinetic energy equals $-V(q)/2$ and hence $E=V/2$. Plugging this into the energy equation above gives
\[
r^2= - M_k(q)V(q)
\]
and hence
\[
\nu = \frac12  M_k(q)V^2(q).
\]
Equation~\eqref{eq:critical_nu_relequil} then follows from the homogeneity of $M_k$ and $V$.  
\end{proof}
} % end rem

%---------------------------------------------------------------------------------------------------

Even though our approach does not allow for a discussion of the manifestation of all critical points of the map of integrals in the bifurcation of the Hill regions for the reasons mentioned above, we in the following determine for all of them the corresponding value of the bifurcation parameter $\nu$ defined in \eqref{eq:def_nu}.  The reason is that in the examples in 
Sec.~\ref{charged_sec:Examples}  we will observe bifurcations of the Hill regions also due to collinear relative equilibria and critical points at infinitiy, i.e. bifurcations of the Hill regions due to events associated with the boundary of the shape space. We will use that, like for non-collinear relative equilibria, the corresponding $\nu$ can be computed from
\begin{equation} \label{eq:critical_nu_relequil}
\nu = \frac12  \tilde{M}_k (\tilde{q})  \, \tilde{V}^2 (\tilde{q})  = \frac12  M_k ({q})  \, {V}^2 ({q}) \,
\end{equation}
where $\tilde{q}$ is the dilation scaled internal  configuration $q$ of the relative equilibrium or critical point at infinity. In fact the formula also holds for the value $\nu$ of the pseudo critical point due to the diabolic point of the moment of inertia tensor. The reason is that in any of these cases bifurcations occur when the contours of the function \eqref{charged_eq:def_bif_function_Hill}  sweep over the corresponding shape-orientation point be it in (the interior of) $\tilde{\Qtr } \times S_1^2$ or on the boundary  $\partial\tilde{\Qtr } \times S_1^2$ (where in the latter case the continuous extension of \eqref{charged_eq:def_bif_function_Hill} needs to be considered). The second equality in 
\eqref{eq:critical_nu_relequil} follows from the homogeneity of $V$ and $\mathbf{M}$, see \eqref{charged_eq:homogeneities}.

%---------------------------------------------------------------------------------------------------

\subsection{Non-collinear relative equilibria}

Finding relative equilibria is not easy. In the gravitational case they are all related to central configurations. For three bodies interacting via gravitational forces, the Lagrange equilateral triangle is the only non-collinear central configuration. As shown in \cite{HoveijnWaalkensZaman2019}, there can in general be relative equilibria that do not project to central configurations for charged three-body systems different from gravitational three-body systems. An example is the Langmuir orbit discussed below.

\subsubsection{Lagrange type relative equilibria}
\label{sec:Lagrange}

Let $\alpha_1=Gm_2m_3$,   $\alpha_2=Gm_1 m_3$ and  $\alpha_3=Gm_1m_2$, where $G$ is the gravitational constant.
The distances to the center of mass in a Lagrange equilateral triangle of side length $a$ %, see Fig.~\ref{charged_fig:LagrangeTriangle}, 
are \cite{Danby1992}
\begin{eqnarray}
d_1 = \frac{a}{m_1+m_2+m_3} \sqrt{m_2^2 + m_2m_3 + m_3^2}\,, \\
d_2 = \frac{a}{m_1+m_2+m_3} \sqrt{m_1^2 + m_1m_3 + m_3^2}\,, \\
d_3 = \frac{a}{m_1+m_2+m_3} \sqrt{m_1^2 + m_1m_2 + m_2^2}\,.
\end{eqnarray}
The moment of inertia for rotations about the axis containing the center of mass and perpendicular to the equilateral triangle is then given by
$$
m_1 d_1^2 + m_2 d_2^2 + m_3 d_3^2 = \frac{a^2}{m_1+m_2+m_3} (m_1m_2 + m_2 m_3+ m_1m_3) 
$$
which together with the gravitational potential energy $V=-G(m_1 m_2+ m_2 m_3 + m_1 m_3)/a$  leads to the critical value 
\begin{equation} \label{eq:Lagrange_critical_nu_general}
\nu_{\text{Lagrange}} = \frac{G^2}{2} \frac{(m_1m_2 + m_2 m_3 + m_1 m_3)^3}{m_1+m_2+m_3}.
\end{equation}

We expect that Lagrange-type relative equilibria also exist more generally if the $\alpha_i$ in \eqref{eq:potential3b}
are all positive but not necessarily  products of masses and that the bodies do in general not form an equilateral triangle in this case. However, when the interaction between the bodies is due to Coulomb forces then the force between at least two bodies is repulsive and there is no non-collinear relative equilibrium that projects to a central configuration as shown in \cite{HoveijnWaalkensZaman2019}. We will not consider the case of all  $\alpha_i$ being positive in more detail in this paper. 

\subsubsection{Langmuir orbit}
\label{sec:Langmuir}

 An example of a relative equilibrium whose projection to configuration space is not a central configuration is the \emph{Lagnmuir orbit}. This
 is a periodic orbit studied by Irving Langmuir for the quantization of the helium atom in 1920 \cite{Langmuir1920}. 
%The candidate for the non-collinear relative equilibria leading to the critical values   $\nu_3$ for the compound of two electrons and one positron and  $\nu_4$ for the helium atom in Secs.~\ref{charged_sec:electronpositronelectron} and \ref{charged_sec:Helium}, respectively, is the 
It consists of the two electrons and the helium nucleus  moving along three coplanar concentric circles with the orbit of the helium nucleus sandwiched between the orbits of the electrons in a symmetric fashion, see Fig.~\ref{fig:Langmuir_config}. Langmuir assumed the nucleus to have infinite mass. The orbit exists however also for a finite mass and charge ratios different from the case of helium as we will see in the following.

\begin{figure}
\begin{center}
\includegraphics[angle=0,width=7cm]{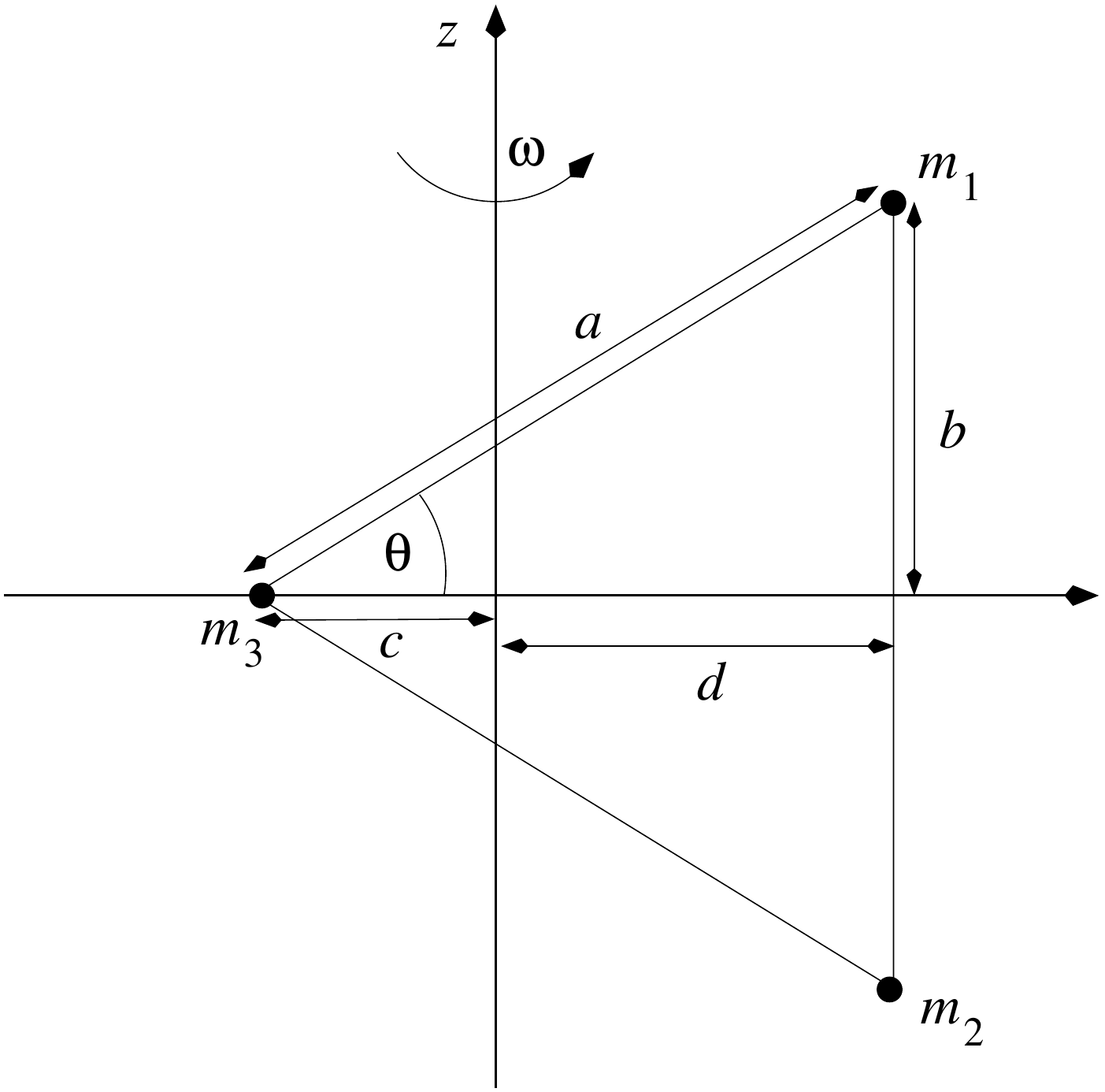}
\end{center}
\caption{\label{fig:Langmuir_config}
Configuration of the Langmuir orbit. Three bodies with  $m_1=m_2$, $\alpha_1=\alpha_2>0$ and $\alpha_3<0$ are placed at the vertices of an isosceles triangle.
}
\end{figure}

We assume that three bodies with two equal masses $m_1=m_2$ interacting such that $\alpha_1=\alpha_2>0$ and $\alpha_3<0$ rotate with angular frequency $\omega$ about the $z$-axis in the configuration shown in Fig.~\ref{fig:Langmuir_config}. This in particular implies that the center of mass is on the $z$-axis.  For the coordinates introduced in Fig.~\ref{fig:Langmuir_config},  this leads to the condition
\begin{equation}\label{eq:app_centermass}
m_3 c = 2 m_1 d. 
\end{equation}
The moment of inertia with respect to rotations about the $z$-axis is given by
\begin{equation}\label{eq:M_Langmuir}
M = 2d^2m_1 + c^2 m_3
\end{equation}
and the potential energy is
\begin{equation} \label{eq:Langmuir_pot}
V = -2\frac{\alpha_{2}}{a} -  \frac{\alpha_3}{2b}\,. 
\end{equation}

\rem{
The goal in the following  is to give a value for the bifurcation parameter $\nu=-Er^2$ in terms of the masses and the charges. With the coordinates in Fig.~\ref{} the energy and the angular momentum  are given by
\begin{equation} \label{eq:app_energy}
E=\frac{m_3}{2} \omega^2 c^2 + 2 \frac{m_1}{2} \omega^2 d^2 + 2\frac{Q_3 Q_1}{a} +  \frac{Q_1^2}{2b} 
\end{equation}
and
\begin{equation}\label{eq:app_angmom}
r = m_3 \omega c^2 + 2 m_1 \omega d^2, 
\end{equation}
respectively. 
} % end rem

In order to proceed we write down the force balances between the internal and the centrifugal forces for the individual bodies. 
For the horizontal components, we get for body 3
\begin{equation} \label{eq:app_m3_hor}
m_3 \omega^2 c = 2 \frac{\alpha_2}{a^2} \cos \theta
\end{equation}
and for body 1 (or body 2 which would give the same equation)
\begin{equation} \label{eq:app_m1_hor}
m_1 \omega^2 d =  \frac{\alpha_2}{a^2} \cos \theta.
\end{equation}
For the vertical component for body 1 (and equivalently body 2),
we get
\begin{equation} \label{eq:app_m1_ver}
 \frac{\alpha_3}{(2b)^2}   = - \frac{\alpha_2}{a^2} \sin \theta.
\end{equation}
Due to the reflection symmetry the equation for the vertical component for body 3 is trivial.

Dividing Eq.~\eqref{eq:app_m3_hor} by Eq.~\eqref{eq:app_m1_hor} gives again the center of mass condition \eqref{eq:app_centermass}. 
From \eqref{eq:app_m1_ver} we get

\[
\frac{a^2}{b^2} = -4 \frac{\alpha_2}{\alpha_3} \sin \theta.
\]
Using that $b=a\sin \theta$ we get that the angle $\theta$ is determined by
\begin{equation}\label{eq:app_theta_general}
\sin \theta = \left( - \frac{\alpha_3}{4\alpha_2} \right)^{1/3}\,.
\end{equation}
Interestingly the angle only depends on the charge ratio and not on the masses of the bodies.
For the case $\alpha_2/\alpha_3=-2$, \eqref{eq:app_theta_general} gives $\sin \theta = \frac12 $, i.e. $\theta = \pi/6=30^{\circ}$ and the three bodies hence form the vertices of an equilateral triangle.
This agrees with the result found by Langmuir for Helium for an infinite mass $m_3$ \cite{Langmuir1920}.

Using $b=a\sin \theta$ the potential energy in \eqref{eq:Langmuir_pot} becomes
\begin{equation} \label{eq:pot_Langmuir_a}
V %= 2\frac{Q_3 Q_1}{b/\sin\theta} +  \frac{Q_1^2}{2b}   %= \frac{1}{b}  (2Q_3 Q_1\sin \theta+\frac{Q_1^2}{2}). 
=  -\frac{1}{a}  (2\alpha_2   +  \frac{\alpha_3}{2\sin \theta}).  
\end{equation}

\rem{
From the Virial Theorem we know that the kinetic energy $T$ is given by $-V/2$ such that we get for the full energy $E$
\[
E =T+V= \frac12 V =\frac{1}{b} \left( Q_3Q_1\sin \theta + \frac{Q_1^2}{4} \right).
\]
In order to write the bifurcation parameter $\nu=- E r^2 $ as function of masses and charges only, we need to write $r^2$ as a product of $b$ and a factor depending on masses and charges only. Solving \eqref{eq:app_m3_hor} for $\omega^2$, we get
\[
r^2 =  (m_3 c^2 + 2m_1 d^2)^2 \omega^2= (m_3 c^2 + 2m_1 d^2)^2 \left(  -2\frac{Q_3Q_1}{m_3ca^2} \cos \theta \right) 
\]

... and hence
\[
\nu %= - \left( Q_3Q_1\sin \theta + \frac{Q_1^2}{4} \right) \left(m_3 \frac{c^2}{a^2} + 2m_1 \frac{d^2}{a^2} \right)^2 \left(  -2\frac{Q_3 Q_1}{m_3} \frac{a}{c}\frac{a}{b} \cos \theta \right) \,.
= - \left( \alpha_2\sin \theta + \frac{\alpha_3}{4} \right) \left(m_3 \frac{c^2}{a^2} + 2m_1 \frac{d^2}{a^2} \right)^2 \left(  -2\frac{\alpha_2}{m_3} \frac{a}{c}\frac{a}{b} \cos \theta \right) \,.
\]
} % end rem

In order to derive an expression for the bifurcation parameter $\nu=\frac12 MV^2$  and given $V$ in \eqref{eq:pot_Langmuir_a} we proceed by expressing the moment of inertia in \eqref{eq:M_Langmuir} in terms of $a$ and charges and masses. This can be achieved by using $(c+d)/a=\cos \theta$ and removing $d$ or $c$, respectively, via \eqref{eq:app_centermass} which gives
$$
\frac{c}{a} = \frac{2m_1}{2m_1+m_3}\cos \theta \text{ and } \frac{d}{a} = \frac{m_3}{2m_1+m_3}\cos \theta.
$$
Plugging this in and simplifying 
we finally get
\begin{equation}\label{eq:app_nu_general}
\nu_{\text{Langmuir}} = % \frac12 \mu \frac{\cos^4 \theta}{\sin \theta} Q_3Q_1^2 (4Q_3 \sin \theta +Q_1),
\frac12 \mu \frac{\cos^4 \theta}{\sin \theta} (4\alpha_2^2 \sin \theta +\alpha_2\alpha_3),
\end{equation}
where $\mu$ is the reduced mass 
$$
\mu = \frac{2m_1m_3}{2m_1+m_3}
$$ 
and $\theta$ is determined by \eqref{eq:app_theta_general}.

Without giving further details we note that it can be shown that the Langmuir orbit exists also for asymmetric mass and charge ratios.

%---------------------------------------------------------------------------------------------------

\subsection{Non-collinear pseudo critical point due to the diabolic point of the  inertia tensor}
\label{sec:diabolic}

Following the discussion at the end of Sec.~\ref{charged_sec:Hillregions} a pseudo critical point arises from the diabolic point of the inertia tensor. At this point 
two principal moments of inertia (the smaller ones) are equal. In this case the mass weighted Jacobi vectors are perpendicular and have equal length (see Remark~\ref{lemma:M_tilde_singular}). 
\rem{
For the distances $a$, $b$, $c$ and $d$ defined in Fig.~\ref{}, this means $a+b=\rho/\sqrt{\mu_1}$ and $m_1a=m_3b$, and $c+d=\rho/\sqrt{\mu_2}$ and $c(m_1+m_3)=d m_2$. These pairs of equations can be solved for $a,b$ and $c,d$, respectively. Pythagoras' Theorem then yields
} % end rem
It is easily shown that in this case the distances between the bodies are
\begin{equation}
\begin{split}
r_{12} = \sqrt{\frac{m_1+m_2}{m_1 m_2}} \,  \rho,  \quad
r_{23} = \sqrt{\frac{m_2+m_3}{m_2 m_3}} \,  \rho   \text{ and }
r_{13} =  \sqrt{\frac{m_1+m_3}{m_1 m_3}}\, \rho \,.
\end{split}
\end{equation}
\rem{
For the distances of the masses to the center of mass, we get
\begin{equation}
\begin{split}
d_1 &= \sqrt{\frac{m_2 +m_3}{m_1(m_1+m_2+m_3)}}  \, \rho\,,\\
d_2 &= \sqrt{\frac{m_1 +m_3}{m_2(m_1+m_2+m_3)}} \, \rho \,, \\
d_3 &= \sqrt{\frac{m_1 +m_2}{m_3(m_1+m_2+m_3)}}\, \rho  \,.
\end{split}
\end{equation}
} % end rem
The  moment of inertia corresponding to rotations about any axis containing the center of mass and in the  plane spanned by the three bodies is $\rho^2$. Using \eqref{eq:critical_nu_relequil}
we get the critical value 
\begin{equation} \label{eq:nu_diabolic}
\nu_{\text{diabolic}} = \frac12
\left( 
\alpha_1   \sqrt{ \frac{m_2m_3}{m_2 + m_3}}  +   \alpha_2  \sqrt{ \frac{m_1 m_3}{m_1 + m_3}} + \alpha_3 \sqrt{\frac{m_1+m_2}{m_1 m_2}} 
\right)^2\,.
\end{equation}

%There are no analogous singularities for the collinear case because there is only one moment of inertia in this case ...

%---------------------------------------------------------------------------------------------------

\subsection{Collinear relative equilibria}
\label{sec:collinear}

For a collinear relative equilibrium,
imagine three masses $m_1$, $m_2$ and $m_3$ at positions $0\le a \le b$ on the real axis (for other orders simply relabel the bodies). Then the center of mass is at
\[
cm = \frac{am_2+bm_3}{m_1+m_2+m_3}
\]
and the signed distances to the center of mass of the three masses are
\[
d_{1} = -cm \, , \quad  d_{2} =\frac{am_1-(b-a)m_3}{m_1+m_2+m_3} \quad \text{and} \quad d_{3} =\frac{bm_1+(b-a)m_3}{m_1+m_2+m_3}\,,
\]
respectively. Using that the distances between the bodies are $r_{12} = a$, $r_{13}=b$ and $r_{23} = b-a$ the moment of inertia $m_1 d_{1cm}^2 +  m_2 d_{2cm}^2 +  m_3 d_{3cm}^2 $  becomes
\[
M =  \frac{ m_1 m_2 a^2 + m_2 m_3 (b-a)^2 +  m_1 m_3 b^2 }{m_1+m_2+m_3} 
= \frac{ m_1 m_2  r_{12}^2 + m_2 m_3  r_{23}^2  + m_1 m_3  r_{13}^2    }{m_1+m_2+m_3} \,.
\]
Using \eqref{eq:critical_nu_relequil} the corresponding value of the bifurcation parameter becomes
\begin{equation}\label{eq:nu_collinear}
\nu_{\text{collinear}} = \frac12  \frac{ m_1 m_2  r_{12}^2 + m_2 m_3  r_{23}^2  + m_1 m_3  r_{13}^2    }{m_1+m_2+m_3} \left( \frac{\alpha_1}{r_{23}} +   \frac{\alpha_2}{r_{13}} + \frac{\alpha_3}{r_{12}}  \right)^2\,.
\end{equation}
%and $\nu_7$, $\nu_8$ and  $\nu_9$ in \eqref{eq:Euler_critical_nu} are immediately recognized to be of the form  \eqref{eq:critical_nu_relequil}.
For collinear relative equilibria that project to central configurations,  the distances $r_{12}$, $r_{13}$ and $r_{23}$  in \eqref{eq:nu_collinear}  are obtained from the roots of a  polynomial of degree five as discussed in detail in \cite{HoveijnWaalkensZaman2019}.

%---------------------------------------------------------------------------------------------------

\subsection{Critical points at infinity}
\label{sec:critptsinfty}

%It is interesting to note that also the other critical values are of the form  \eqref{eq:critical_nu_relequil}. For $\nu_2$, $\nu_3$ and $\nu_4$ which are due to critical points at infinity this is easily shown.   In fact, 

The critical points at infinity for a system of three bodies consist of one pair of co-rotating bodies interacting with an attractive force and a third body at rest located on the axis of rotation of the co-rotating pair 
with an infinite distance between the third body and the co-rotating pair. 
A straightforward computation shows that  the moment of inertia of two masses $m_1$ and $m_2$ co-rotating at a distance $r_{12}$ apart about their  center of mass is given by 
%\begin{equation}
$
\mu d^2_{12}
$
where $\mu= m_1m_2/(m_1+m_2)$ is the reduced mass. Using \eqref{eq:critical_nu_relequil} with $V=-\alpha_3/r_{12}$ and $\alpha_3>0$
gives 
\begin{equation}\label{charged_eq:nu_for_corotating_charges}
\nu_{\infty} = \frac12 \mu \alpha_3^2\,.
\end{equation}

%%%%%%%%%%%%%%%%%%%%%%
\section{Examples }
\label{charged_sec:Examples}

In the following we will give two examples of charged three-body systems interacting via Coulomb forces: a compound of two electrons and one positron, 
and  the helium atom consisting of a nucleus and two electrons.   We start  however with a gravitational three-body problem to illustrate the procedure.

%%%%%%%%%%%%%%%%%%%%%%

\subsection{Gravitational three-body problem }
\label{charged_sec:gravitational}

Let $\alpha_1=Gm_2m_3$,   $\alpha_2=Gm_1 m_3$ and  $\alpha_3=Gm_1m_2$ where $G$ is the gravitational constant.
As discussed in \cite{mccord1998integral}
there are nine critical values of the bifurcation parameter $\nu$ for the gravitational three-body problem.  
In our example we choose the masses $m_1 = 1.6$, $m_2=1.2$ and $m_3=1$. Also we set $G=1$ (which can always be achieved by a suitable scaling/choice of units). 
The numerical values of the critical bifurcation parameter can then be computed from the formulas at the end of Sec.~\ref{charged_sec:relative_equilibria} and are given by
\begin{equation}
\begin{split} 
		\nu_1&=0 \,,\quad
                      \nu_{\infty\,2}  \approx     0.3927272727 \,,\quad 
                       \nu_{\infty\,3}  \approx    0.7876923077 \,,\\
                       \nu_{\infty\,4}  &\approx  1.263908571\,,\quad
                      \nu_{\text{diabolic}\,5}  \approx     6.961348535 \,, \quad
                      \nu_{\text{Lagrange}\,6} \approx     13.83605894 \,,\\ 
                      \nu_{\text{collinear}\,7}  &\approx     18.56904438 \,,\quad
                       \nu_{\text{collinear}\,8}  \approx    19.12865697 \,,\quad 
                      \nu_{\text{collinear}\,9}  \approx     19.44296212 \,.
\end{split}
\end{equation}
We note that for the gravitional three-body problem, the diabolic point has been discussed earlier by Simo \cite{Simo1975} and Saari \cite{Saari1984,Saari1987,Saari1987b}.
We point out again that it is not a proper critical point where the Hill region nor the integral manifold changes topology (see also the discussion in \cite{mccord1998integral}).

\rem{
Ordered by magnitude they are given by $\nu_1=0$, 
\begin{equation}
\nu_{\infty\,2} = \frac{G^2}{2} \frac{(m_1 m_2)^3}{m_1+m_2},\quad 
\nu_{\infty\,3} = \frac{G^2}{2} \frac{(m_1 m_3)^3}{m_1+m_3},\quad 
\nu_{\infty\,4} = \frac{G^2}{2} \frac{(m_2 m_3)^3}{m_2+m_3},\quad 
\end{equation}
which are due to critical points at infinity (see Sec.~\ref{sec:critptsinfty}),
\begin{equation}
\nu_{\text{diabolic}\,5} = \frac{G^2}{2} \left(  m_1 m_2 \sqrt{\frac{m_1 m_2}{m_1 + m_2}} + m_1 m_3 \sqrt{\frac{m_1 m_3}{m_1 + m_3}}  + m_2 m_3 \sqrt{\frac{m_2 m_3}{m_2 + m_3}}    \right)^2,
\end{equation} 
%% note that in Mohammad's thesis there is a square missing for nu_5
corresponding to the diabolic point (see Sec.~\ref{sec:diabolic}; this critical value has been discussed earlier by Simo \cite{Simo1975} and Saari \cite{Saari1984,Saari1987,Saari1987b}),
\begin{equation} \label{eq:Lagrange_critical_nu}
\nu_{\text{Lagrange}\,6} = \frac{G^2}{2} \frac{(m_1m_2 + m_2 m_3 + m_1 m_3)^3}{m_1+m_2+m_3}
\end{equation} 
which comes from the central configuration given by the Lagrange equilateral triangle (see Sec.~\ref{sec:Lagrange})
and for $k=7,8,9$,
\begin{equation}\label{eq:Euler_critical_nu}
\nu_{\text{collinear}\,k} = \frac{G^2}{2} \frac{m_1 m_2 r_{12}^2  + m_2 m_3 r_{23}^2 +  m_1 m_3 r_{13}^2    }{m_1 + m_2 + m_3}
 \left( 
 \frac{m_1 m_2}{r_{12} } +  \frac{m_2 m_3}{r_{23}} +   \frac{m_1 m_3}{r_{13} } 
 \right)^2,
\end{equation}
where the inter body distances $r_{ij}$ are determined by the three Euler collinear central configurations (see Sec.~\ref{sec:collinear}). 
} % end rem

\rem{
 \begin{figure}
\begin{center}
\includegraphics[angle=0,width=8cm]{Figures/LagrangeTriangle}\\
\end{center}
\caption{\label{charged_fig:LagrangeTriangle}
Lagrange triangle.
}
\end{figure}
} % end rem

Following Sec.~\ref{charged_sec:shapespace} we identify the shape space $\tilde{\Qtr}$ with the upper hemisphere of the unit sphere in the $(w_1,w_2,w_3)$-space. We visualise the shape space by projecting this hemisphere to the unit disk in the $(w_1,w_2)$-plane. 
In Fig.~\ref{charged_fig:Newton_levels} we show the contours of the functions defined in \eqref{charged_eq:def_VtildesqrtMk} in this projection. 
Recall from the discussion in Sec.~\ref{charged_sec:triatomic} that the  boundary of the unit disk in the  $(w_1,w_2)$-plane corresponds to the collinear configurations.
Following \eqref{charged_eq:angle_collision_12} and \eqref{charged_eq:angle_collision_23} 
the collision of bodies 1 and 2 is located on the boundary at a polar angle of about $48^\circ$, the  collision of 2 and 3 is at a polar angle of about $-71^\circ$ and the collision of 1 and 3 is at polar angle $180^\circ$. At these points the potential function $\tilde{V}$ is $-\infty$. 
The Euler collinear configurations are located  at polar angles of approximately 
$117^\circ$ with body 1 between bodies 2 and 3,
$-121^\circ$ with body 3 between bodies 1 and 2, and
$-19^\circ$ with body 2 between bodies 1 and 3.

The critical values of $\nu$ can be identified with the following events for the contours in Fig.~\ref{charged_sec:shapespace} upon varying $\nu$:

\begin{itemize}
\item[(i)] At $\nu_{\infty\,2} $, $\nu_{\infty\,3} $ and $\nu_{\infty\,4} $ the contours of $\sqrt{\tilde{M}_1} \tilde{V}$ (see Fig.~\ref{charged_sec:shapespace}a) successively touch/detach from the boundary of the  shape space.
\item[(ii)] At $\nu_{\text{diabolic}\,5}$ the contours of  $\sqrt{\tilde{M}_1} \tilde{V}$ and $\sqrt{\tilde{M}_2} \tilde{V}$ simultaneously shrink to or emerge from, respectively, a point at the centre of the unit disk (see Figs.~\ref{charged_sec:shapespace}a and b). 
\item[(iii)] At $\nu_{\text{Lagrange}\,6}$ the contours of $\sqrt{\tilde{M}_3} \tilde{V}$ shrink to a point or emerge from a point close to but not at the centre of the unit disk  (see Fig.~\ref{charged_sec:shapespace}c).
\item[(iv)] At $\nu_{\text{collinear}\,7} $, $\nu_{\text{collinear}\,8} $ and $\nu_{\text{collinear}\,9} $ the contours of $\sqrt{\tilde{M}_1} \tilde{V}$ and $\sqrt{\tilde{M}_2} \tilde{V}$ simultaneously  touch/detach from the boundary of the  shape space.

\end{itemize}

In addition when  $\nu$ approaches $ \nu_1=0$ from above then the contour of $\sqrt{\tilde{M}_1} \tilde{V}$ converges to the boundary of the shape space.

\begin{figure}
\begin{center}
\raisebox{4cm}{a)}
\includegraphics[angle=0,width=4cm]{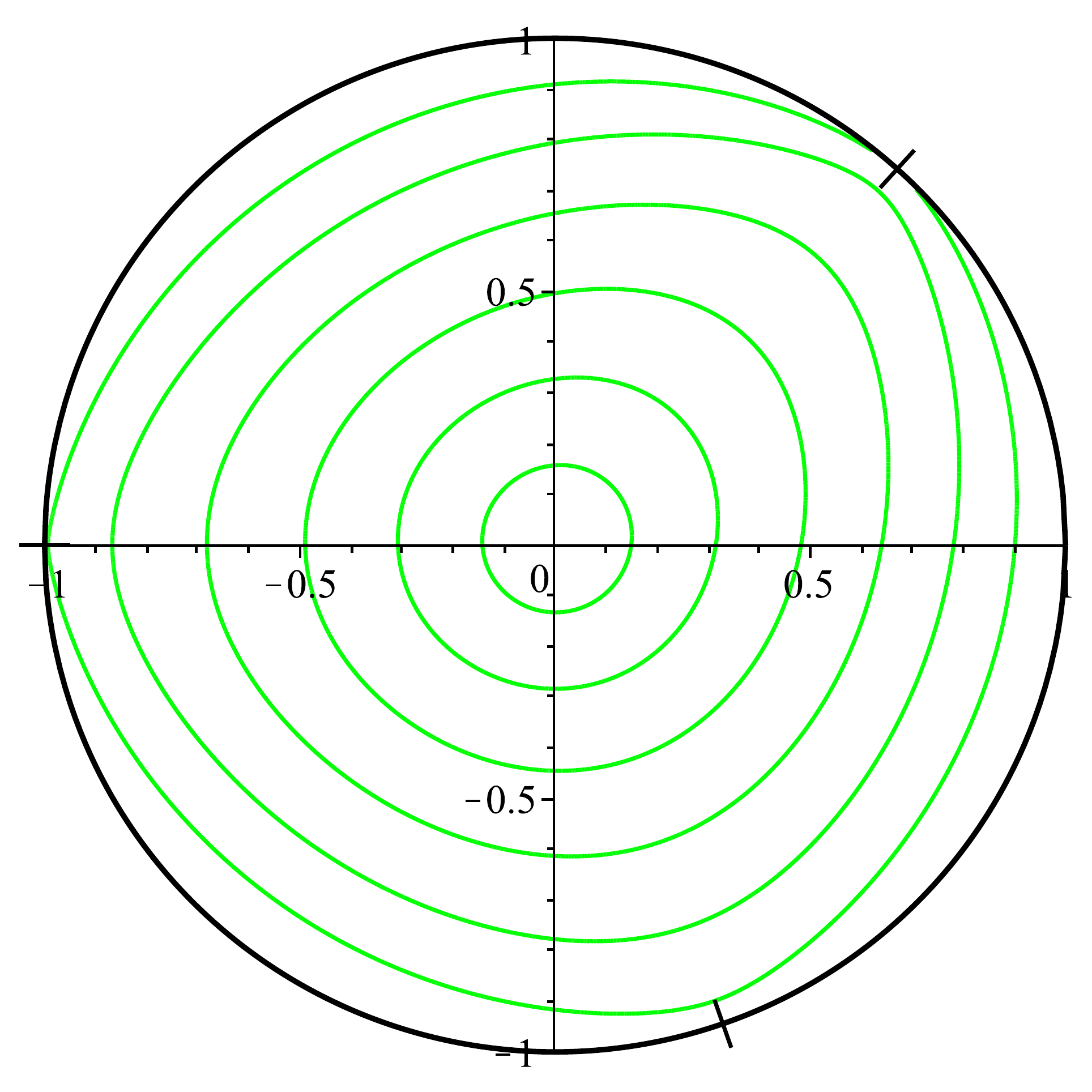}
\raisebox{4cm}{b)}
\includegraphics[angle=0,width=4cm]{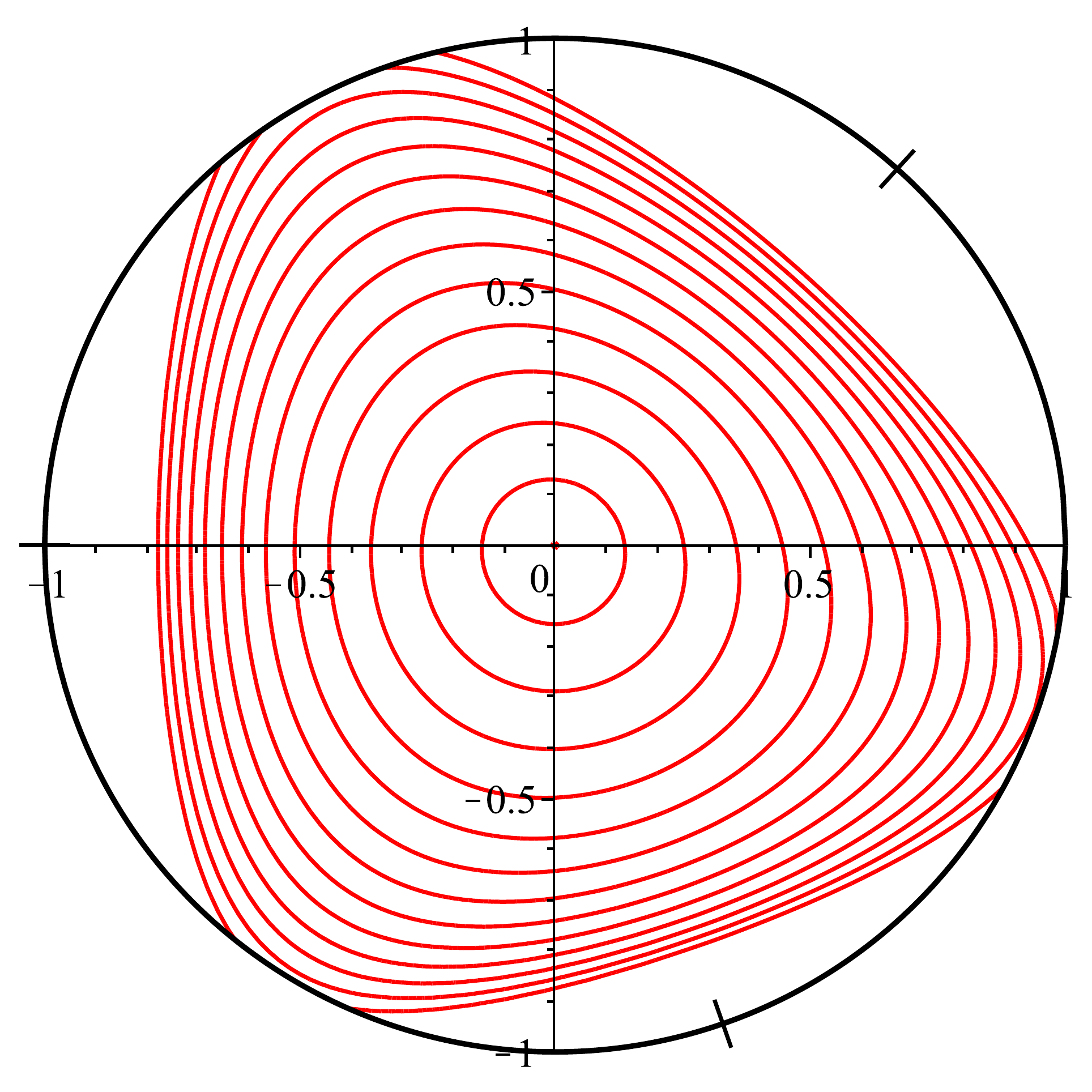}
\raisebox{4cm}{c)}
\includegraphics[angle=0,width=4cm]{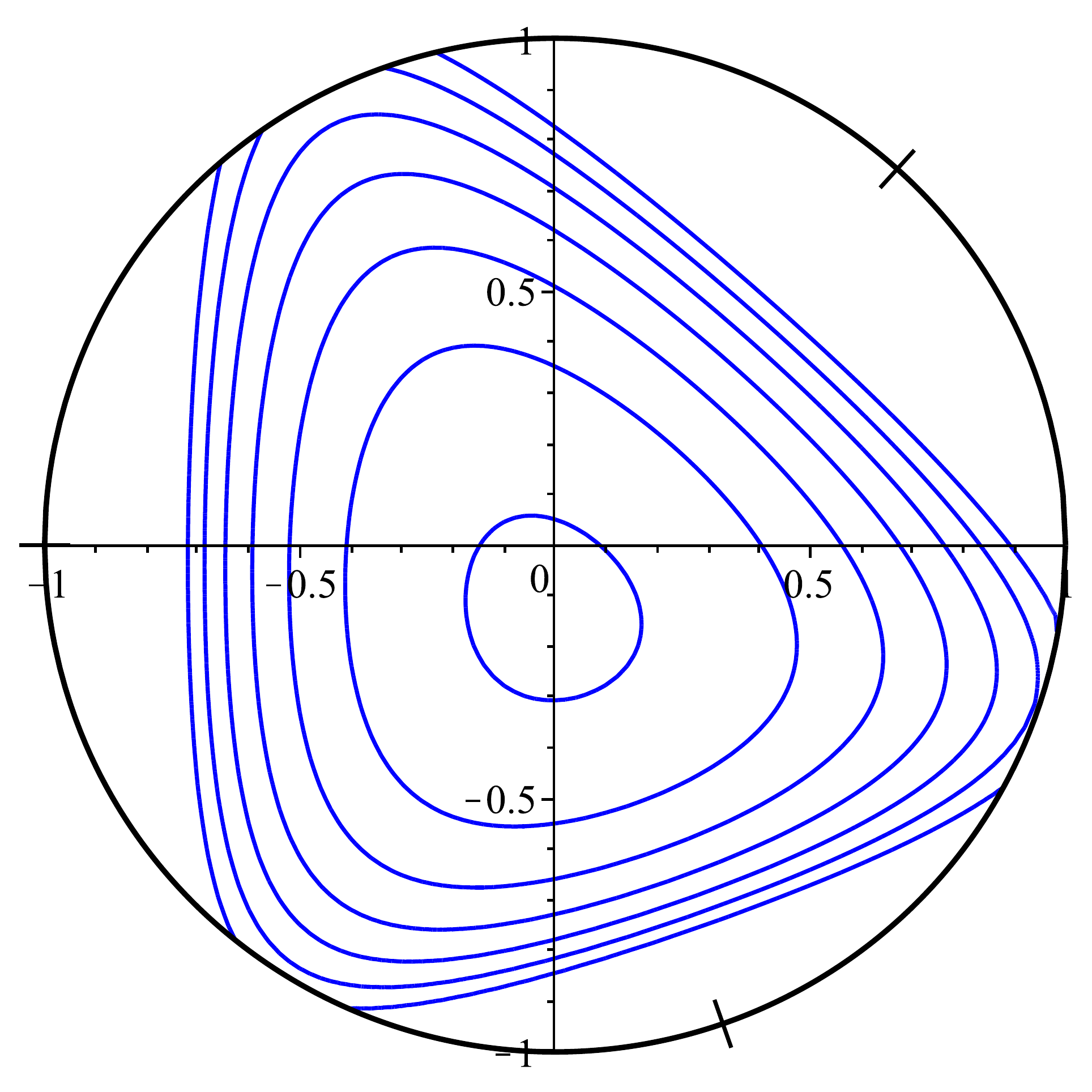}
\end{center}
\caption{\label{charged_fig:Newton_levels}
Gravitational three-body problem:
contours $-\sqrt{2n}$, $n=1,\ldots,20$, of the functions 
$  \sqrt{\tilde{M}_k  } \, \tilde{V} $ 
defined in \eqref{charged_eq:def_VtildesqrtMk} with $k=1$ (a),  $k=2$ (b) and $k=3$ (c). The contours are shown on the shape space $\tilde{\Qtr}$ represented as the projection of the upper hemisphere of the unit sphere in the $(w_1,w_2,w_3)$-space  projected  to the  unit disk in the $(w_1,w_2)$-plane. The ticks on the boundary mark the double collisions (see the main text). 
}
\end{figure}

In the different panels in Figs.~\ref{charged_fig:Newton_levelsa} %, \ref{charged_fig:Newton_levels_psi_chi} 
and \ref{charged_fig:Newton_levelsb}  we show 
the projection of the Hill region to the shape space following the colouring procedure described at the end of Sec.~\ref{charged_sec:Hillregions}
%superpositions of the contours   of the functions $\sqrt{\tilde{M} _k} \tilde{V}$, $k=1,2,3$, 
for a representative values of $\nu$. 
The pictures are obtained from the superposition of the shape space contours   of the functions $  \sqrt{\tilde{M} _k} \tilde{V}$, $k=1,2,3$ in \eqref{charged_eq:def_VtildesqrtMk} for the value of $\nu$ under consideration and colouring the regions enclosed by the contours using the colour code in Fig.~\ref{charged_fig:Euler}.   
 We start with large values of $\nu$. For $\nu>\nu_{\text{collinear}\,9} $, there is a simply connected dark grey shaded region corresponding to points for which the accessible region on the orientation sphere is empty, i.e. these shapes do not belong to the Hill region. This dark shaded region `spills  out' of the shape space at three places. Attached to the dark grey shaded region are three blue strips where the accessible regions on the orientation sphere are two caps for every point inside. Attached to these regions are the three red regions  where for each point inside this region the accessible region on the orientation sphere is a ring. Each of these red regions contains one double collision on its boundary.  For large values of $\nu$ (i.e. low energy and/or large values of the magnitude of the angular momentum) the accessible region on the shape space consists of three disconnected components where either component corresponds to one pair of bodies that are close to each other  and the third body cannot close to this pair. 

 \begin{figure}
\begin{center}
\raisebox{4cm}{a)}
\includegraphics[angle=0,width=4cm]{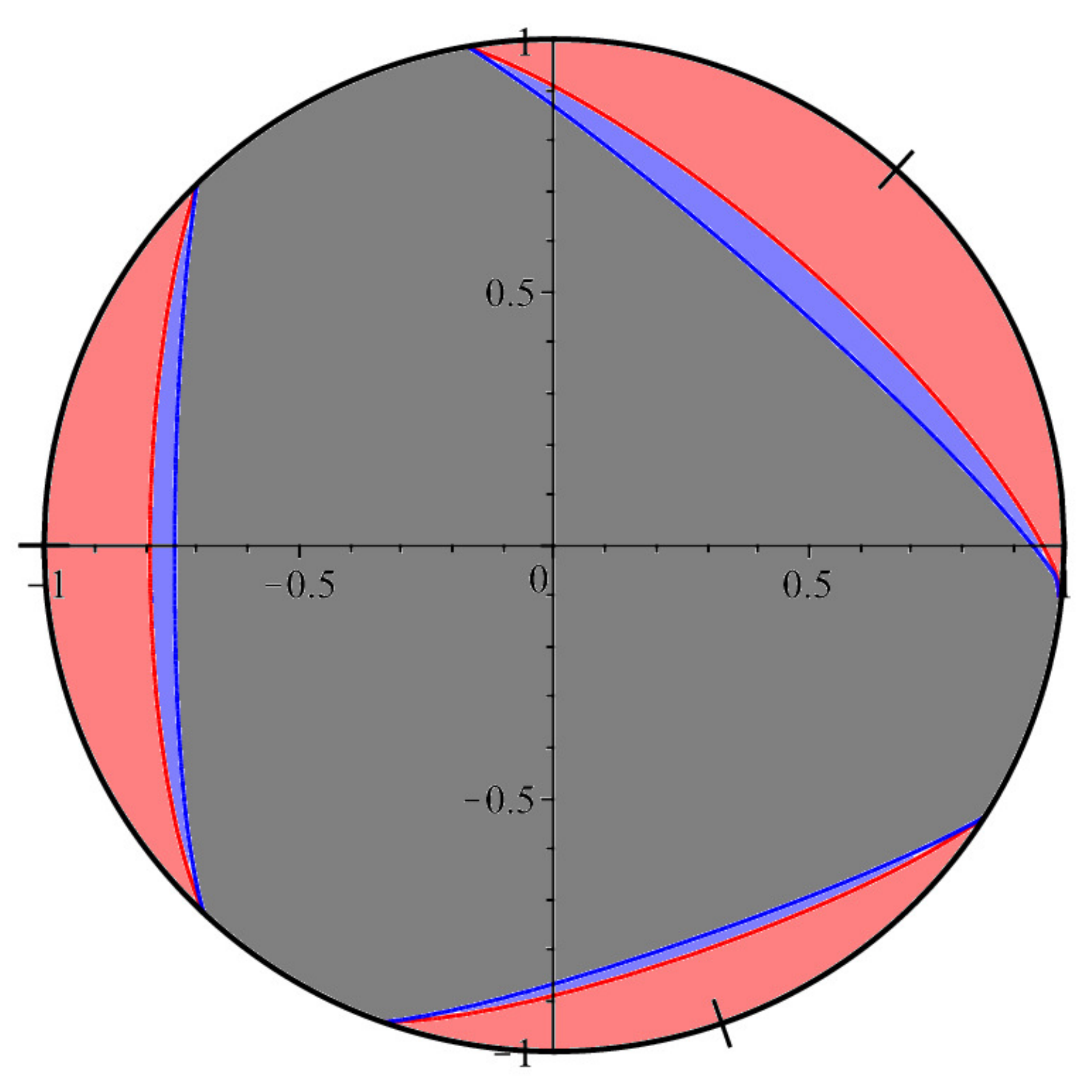} % 
\raisebox{4cm}{b)}
\includegraphics[angle=0,width=4cm]{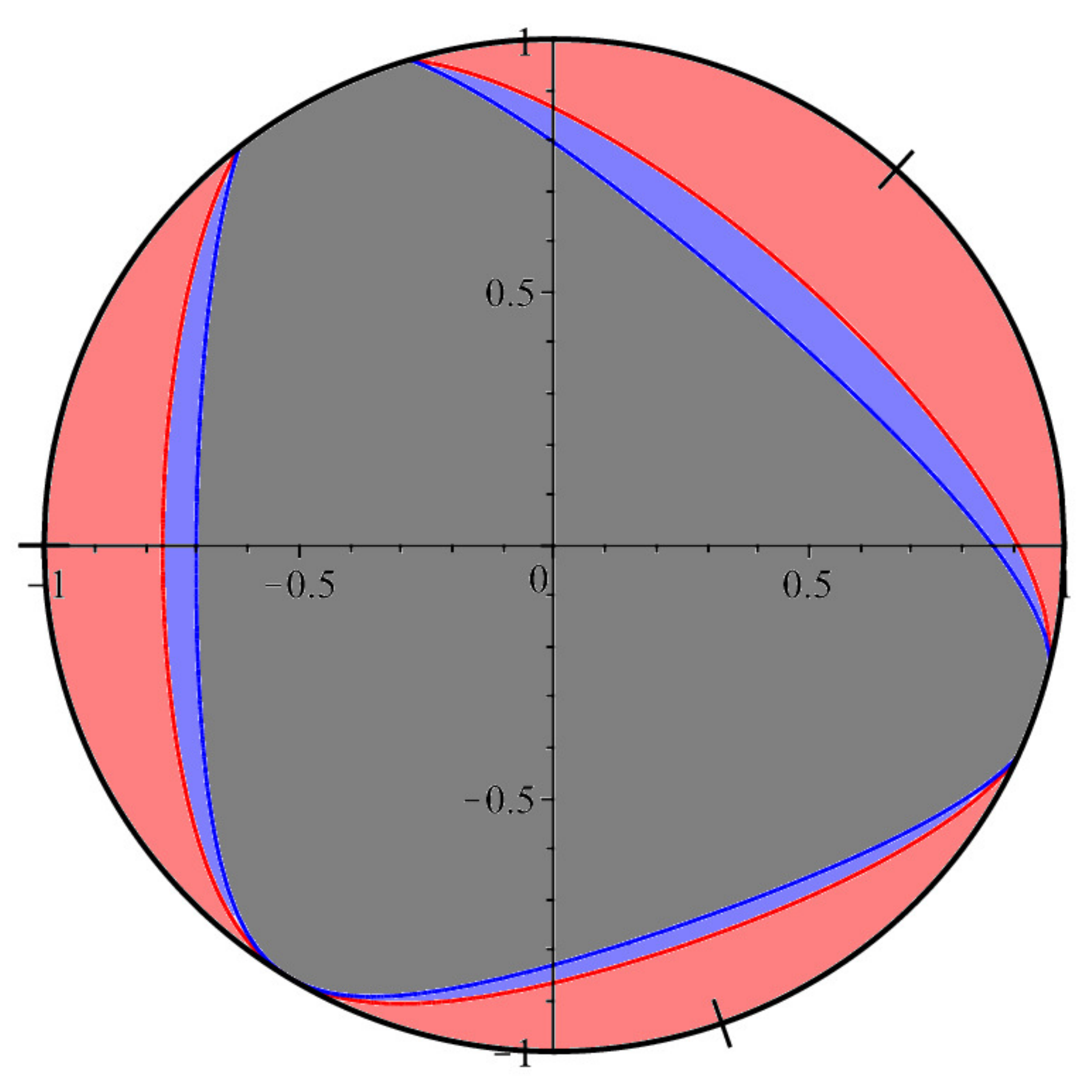}% 
\raisebox{4cm}{c)}
\includegraphics[angle=0,width=4cm]{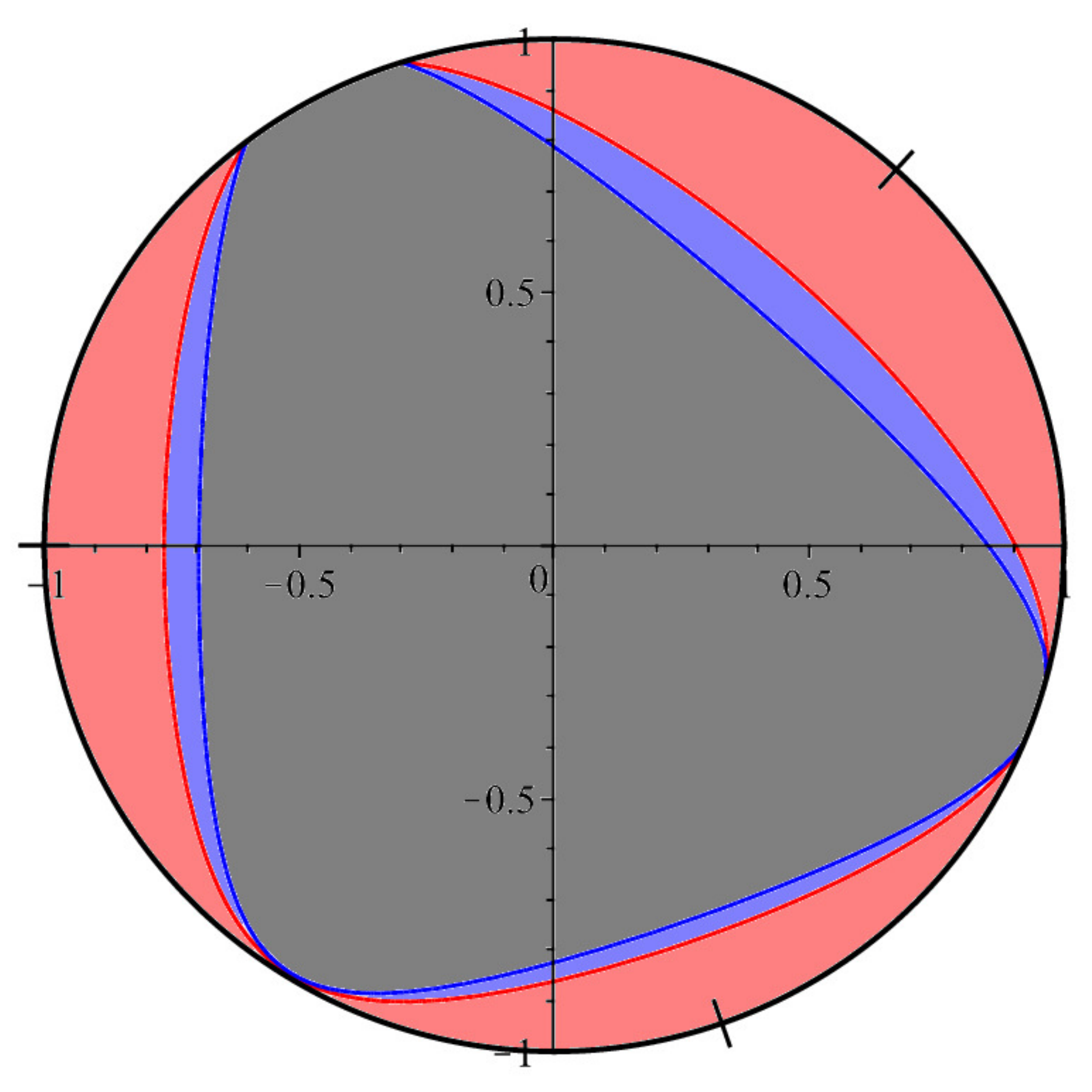}\\ %
\raisebox{4cm}{d)}
\includegraphics[angle=0,width=4cm]{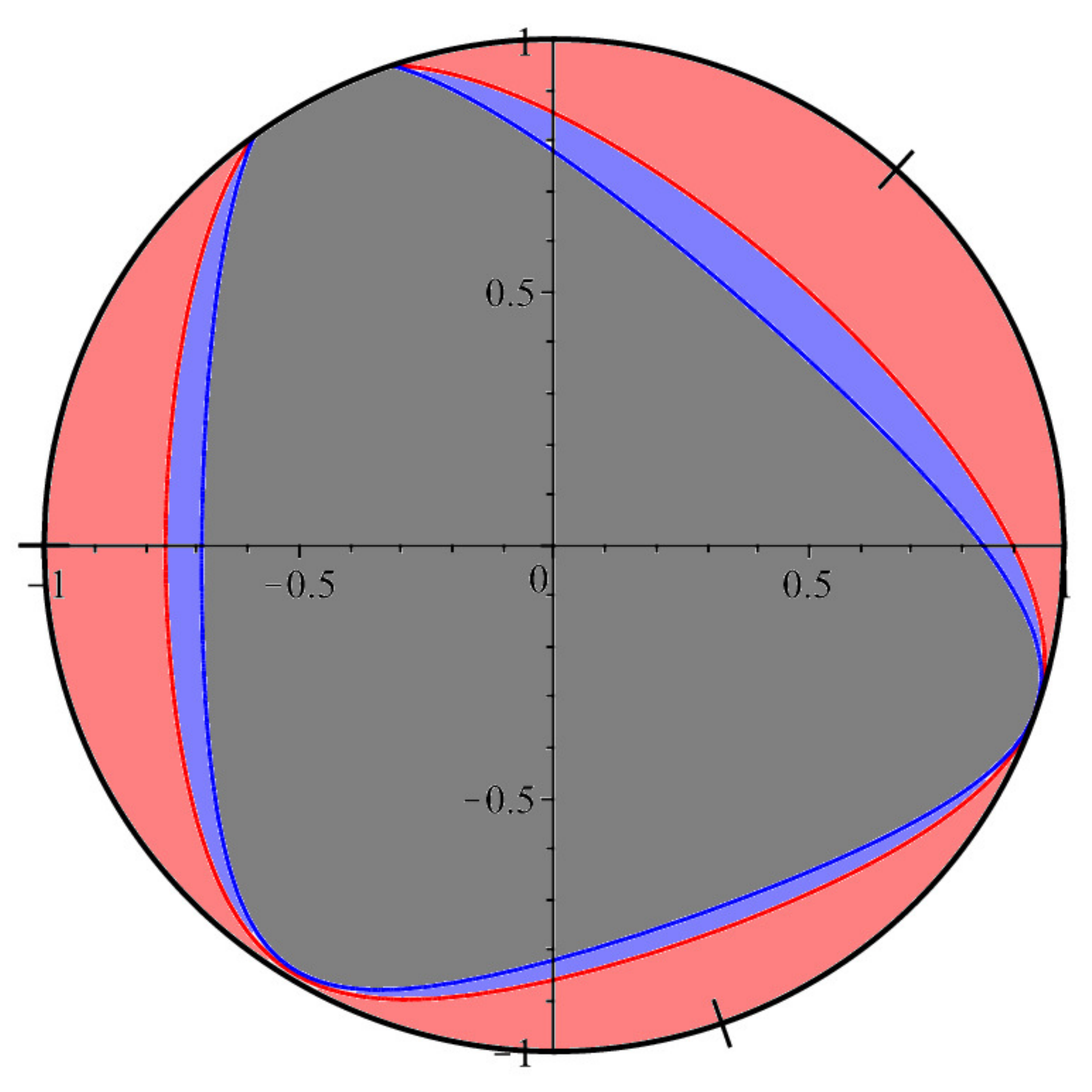} % 
\raisebox{4cm}{e)}
\includegraphics[angle=0,width=4cm]{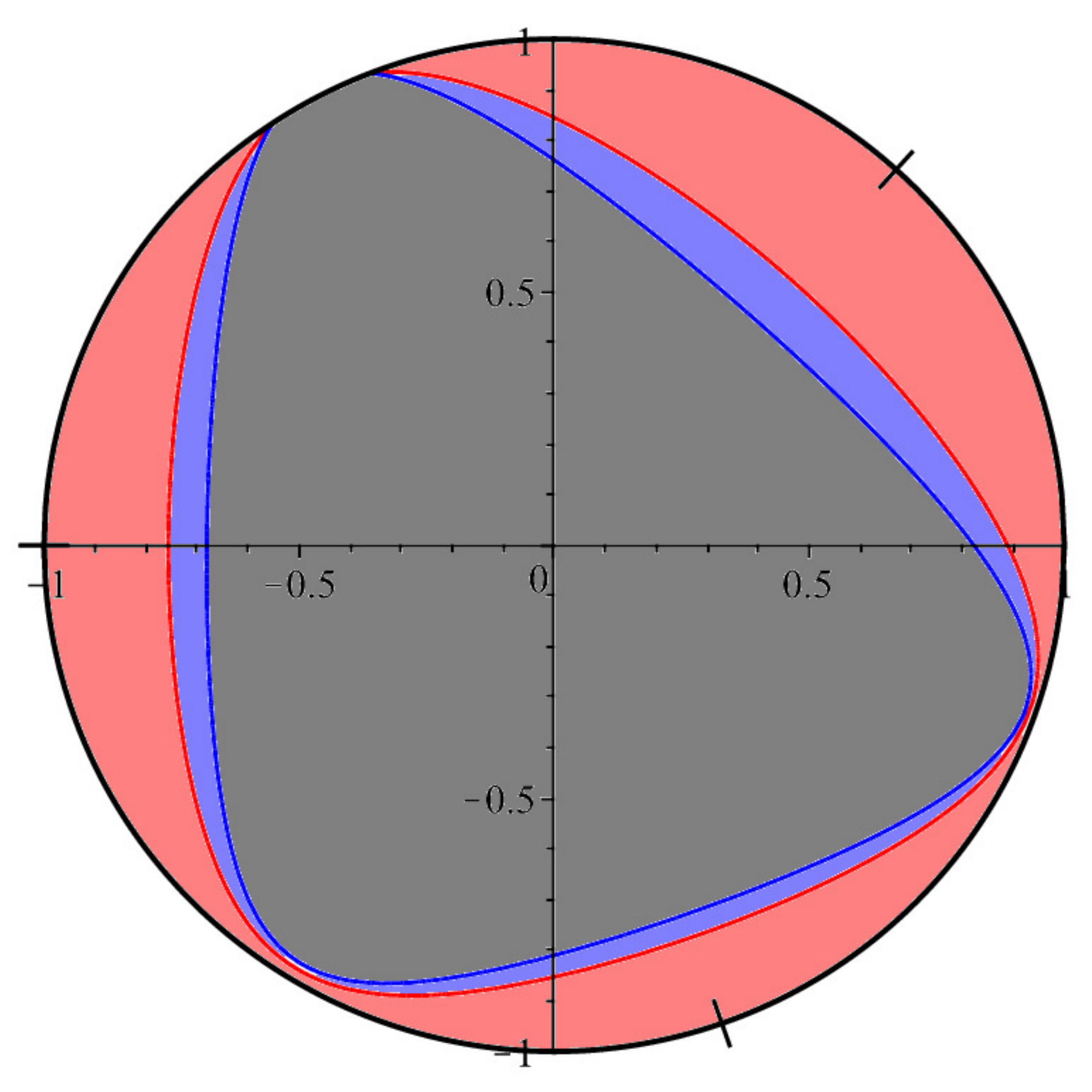}% 
\raisebox{4cm}{f)}
\includegraphics[angle=0,width=4cm]{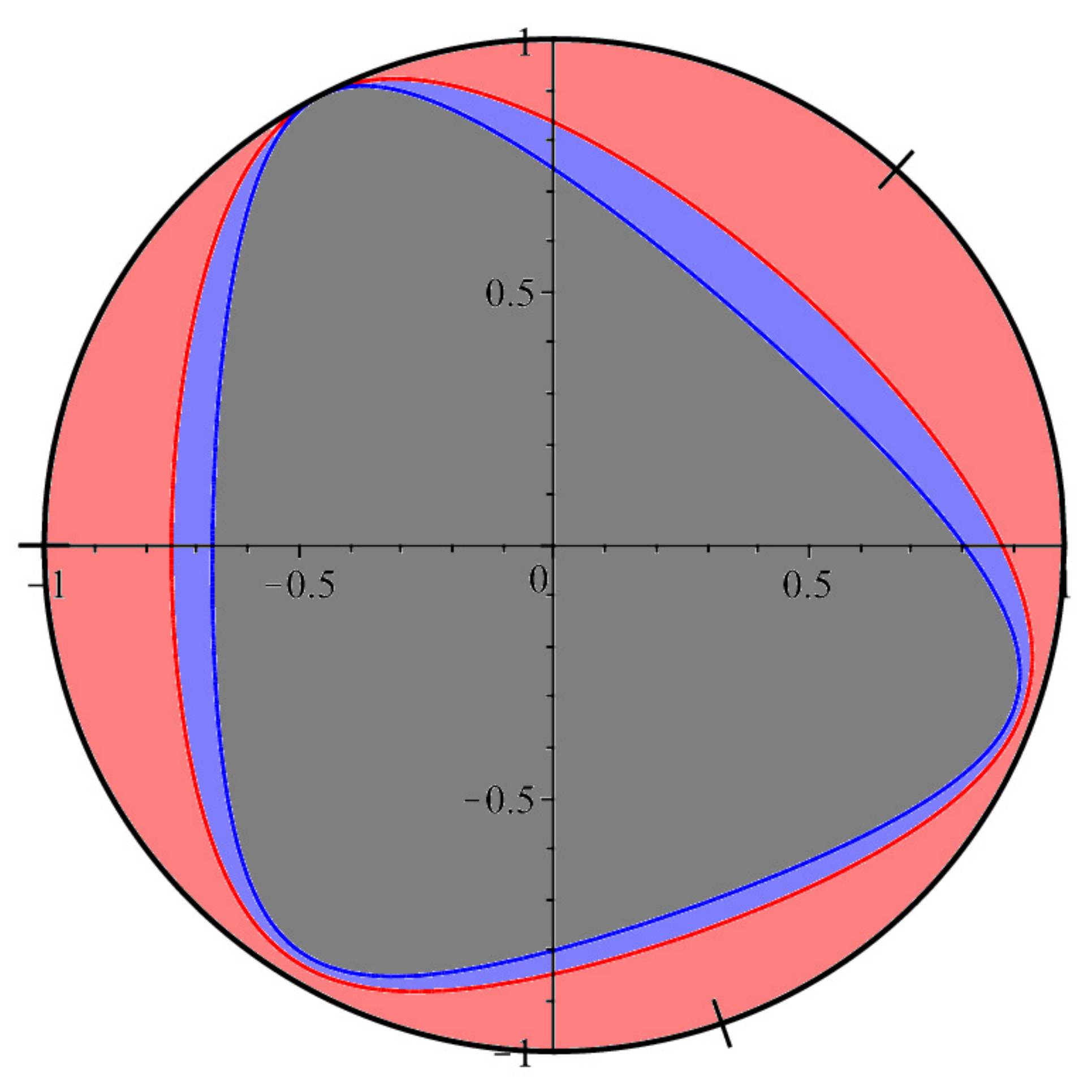}\\ %
\raisebox{4cm}{g)}
\includegraphics[angle=0,width=4cm]{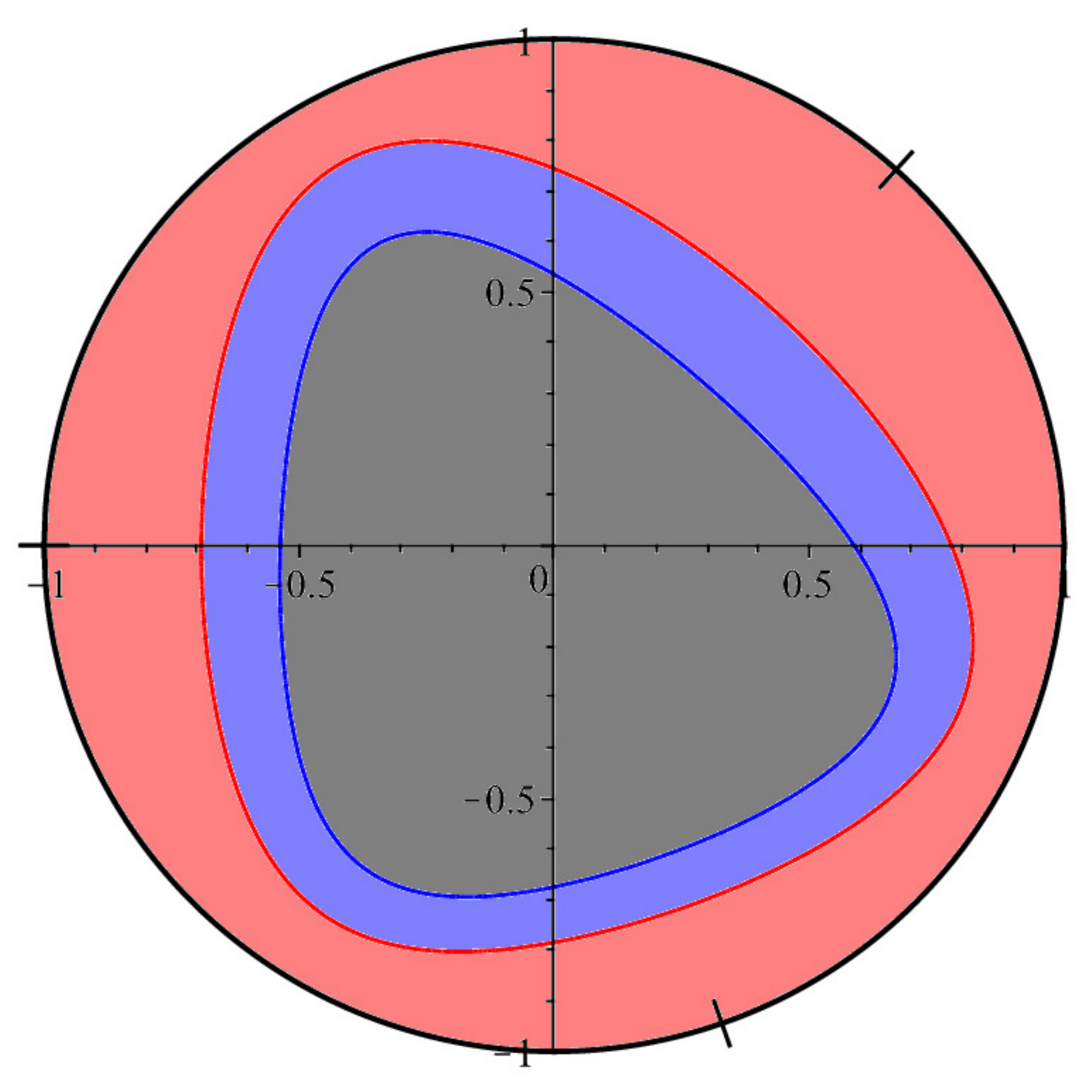} %
\raisebox{4cm}{h)}
\includegraphics[angle=0,width=4cm]{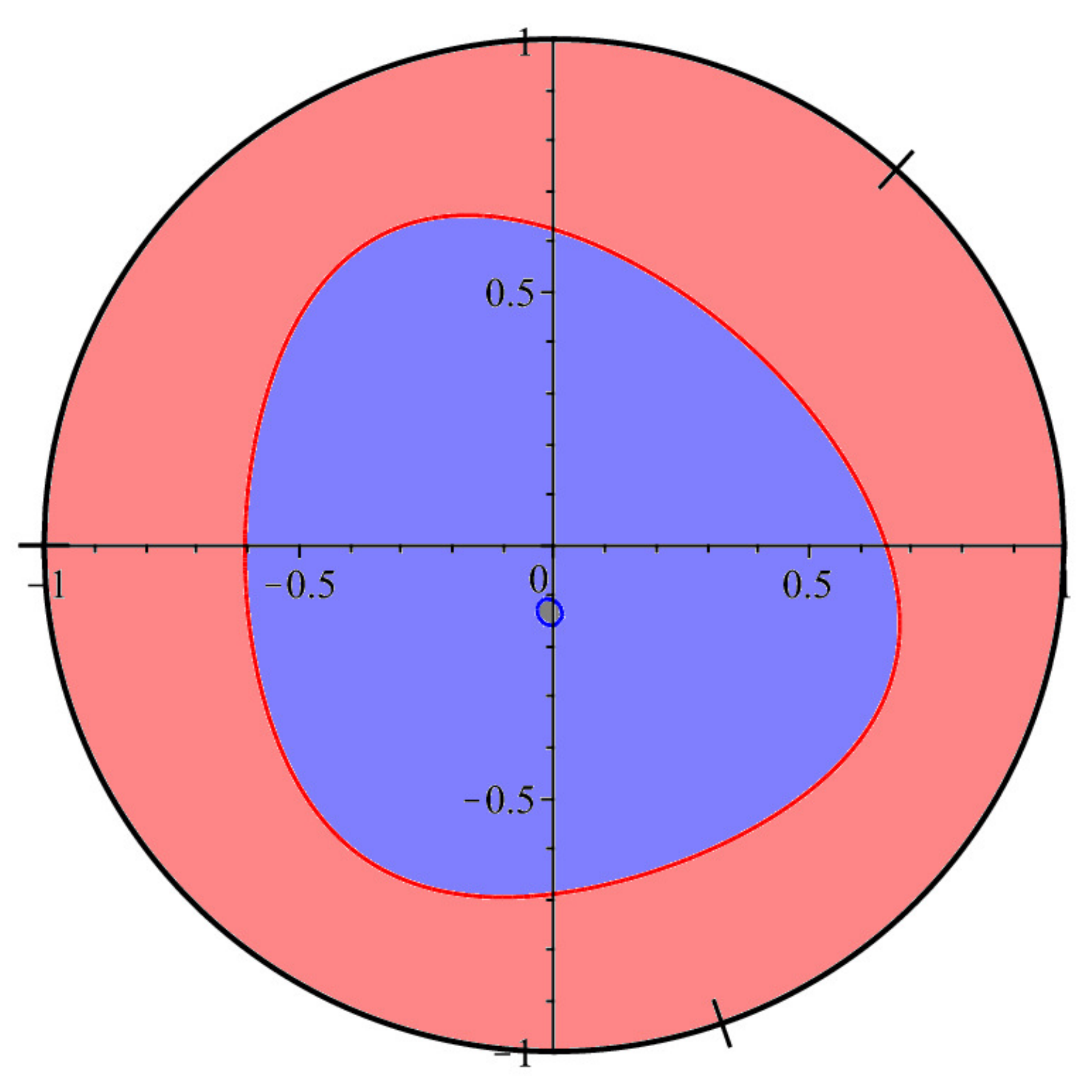}%
\raisebox{4cm}{i)}
\includegraphics[angle=0,width=4cm]{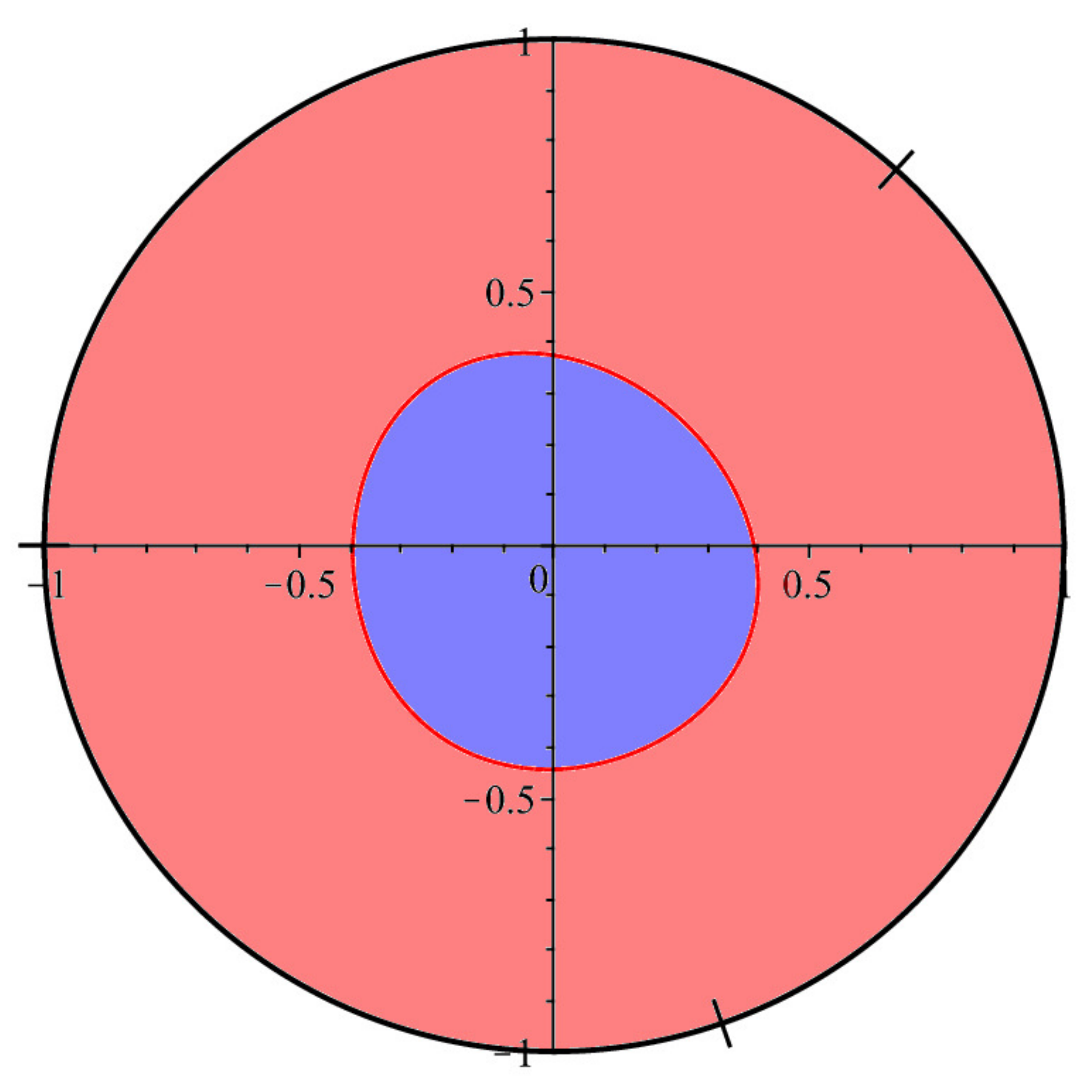}%
\end{center}
\caption{\label{charged_fig:Newton_levelsa}
Gravitational three-body problem:
projections of the Hill regions to the shape space for specific values of $\nu$: 
$  \nu > \nu_{\text{collinear}\,9}    $ (a), 
$ \nu = \nu_{\text{collinear}\,9}  $ (b), 
$ \nu_{\text{collinear}\,9}> \nu >  \nu_{\text{collinear}\,8} $ (c),  
$ \nu = \nu_{\text{collinear}\,8}  $ (d), 
$  \nu_{\text{collinear}\,8} > \nu >  \nu_{\text{collinear}\,7} $ (e),  
$ \nu = \nu_{\text{collinear}\,7}  $ (f), 
$ \nu_{\text{collinear}\,7}  > \nu >   \nu_{\text{Lagrange}\,6}  $ (g),  
$ \nu = \nu_{\text{Lagrange}\,6}   $ (h),  and
$ \nu_{\text{Lagrange}\,6}  >\nu > \nu_{\text{diabolic}\,5}  $ (i).  
The contours are shown on the shape space $\tilde{\Qtr}$ represented in the same way as in Fig.~\ref{charged_fig:Newton_levels}. 
}
\end{figure}

When $\nu$ decreases below the values $\nu_{\text{collinear}\,9} $,  $\nu_{\text{collinear}\,8} $ and $\nu_{\text{collinear}\,7} $ the dark grey shaded region successfully detaches from the boundary of the shape space such that for $\nu_{\text{Lagrange}\,6}  <  \nu <  \nu_{\text{collinear}\,7} $ the dark shaded region is completely contained in the interior of the shape space and the red region has successively become connected. At all instances a blue region is located between the dark grey region and the red region. Recall that on the boundary of the shape space which corresponds to collinear configurations
the principal moments of inertia $\tilde{M}_2$ and $\tilde{M}_3$ are equal (and  $\tilde{M}_1=0$). This is why the red and the blue contours detach from the boundary of the shape space simultaneously in Fig.~\ref{charged_fig:Newton_levels} and why the blue strips in Fig.~\ref{charged_fig:Newton_levelsa} join at $\nu_{\text{collinear}\,9} $,  $\nu_{\text{collinear}\,8} $ and $\nu_{\text{collinear}\,7} $ in the observed way.

 \begin{figure}
\begin{center}
\raisebox{4cm}{j)}
\includegraphics[angle=0,width=4cm]{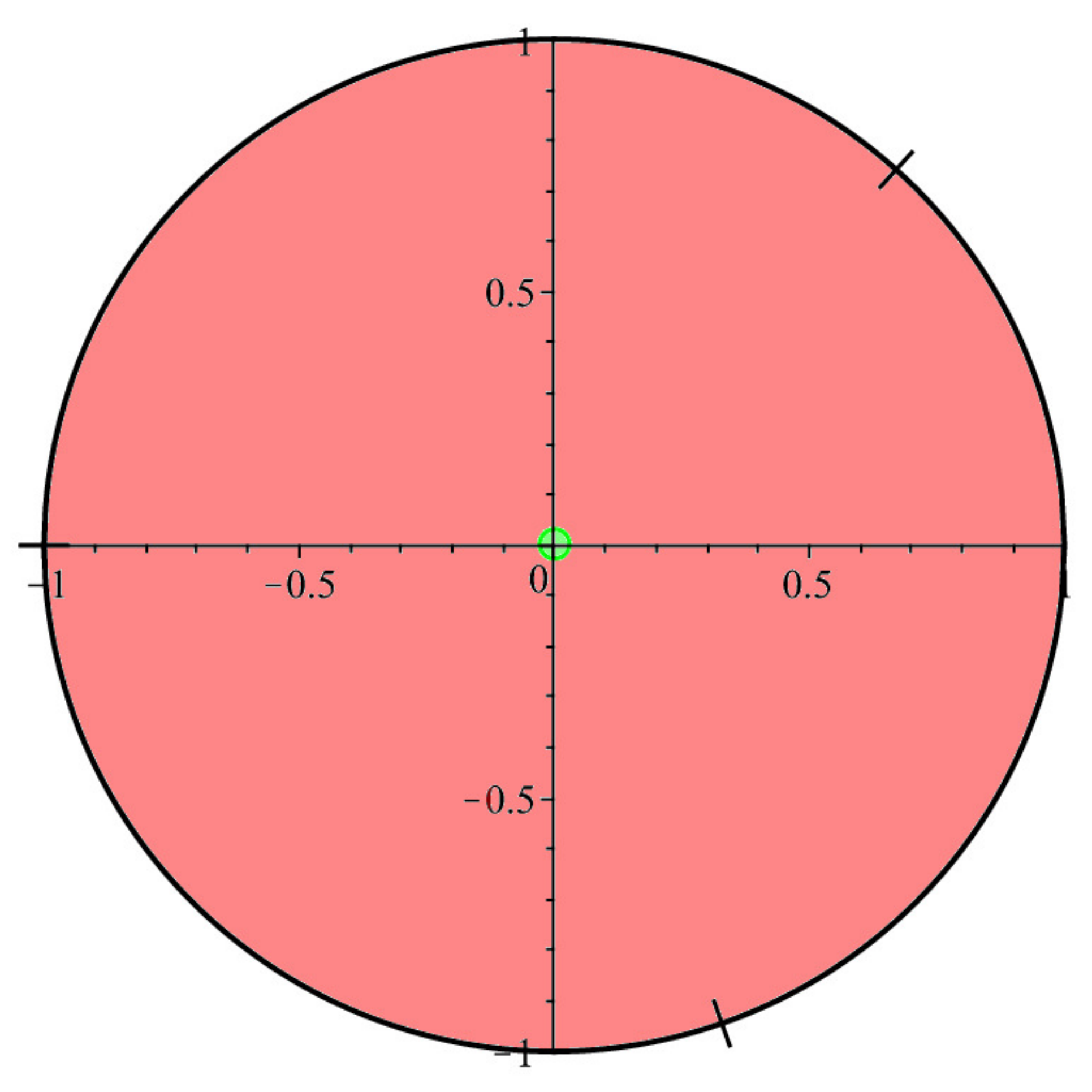}%
\raisebox{4cm}{k)}
\includegraphics[angle=0,width=4cm]{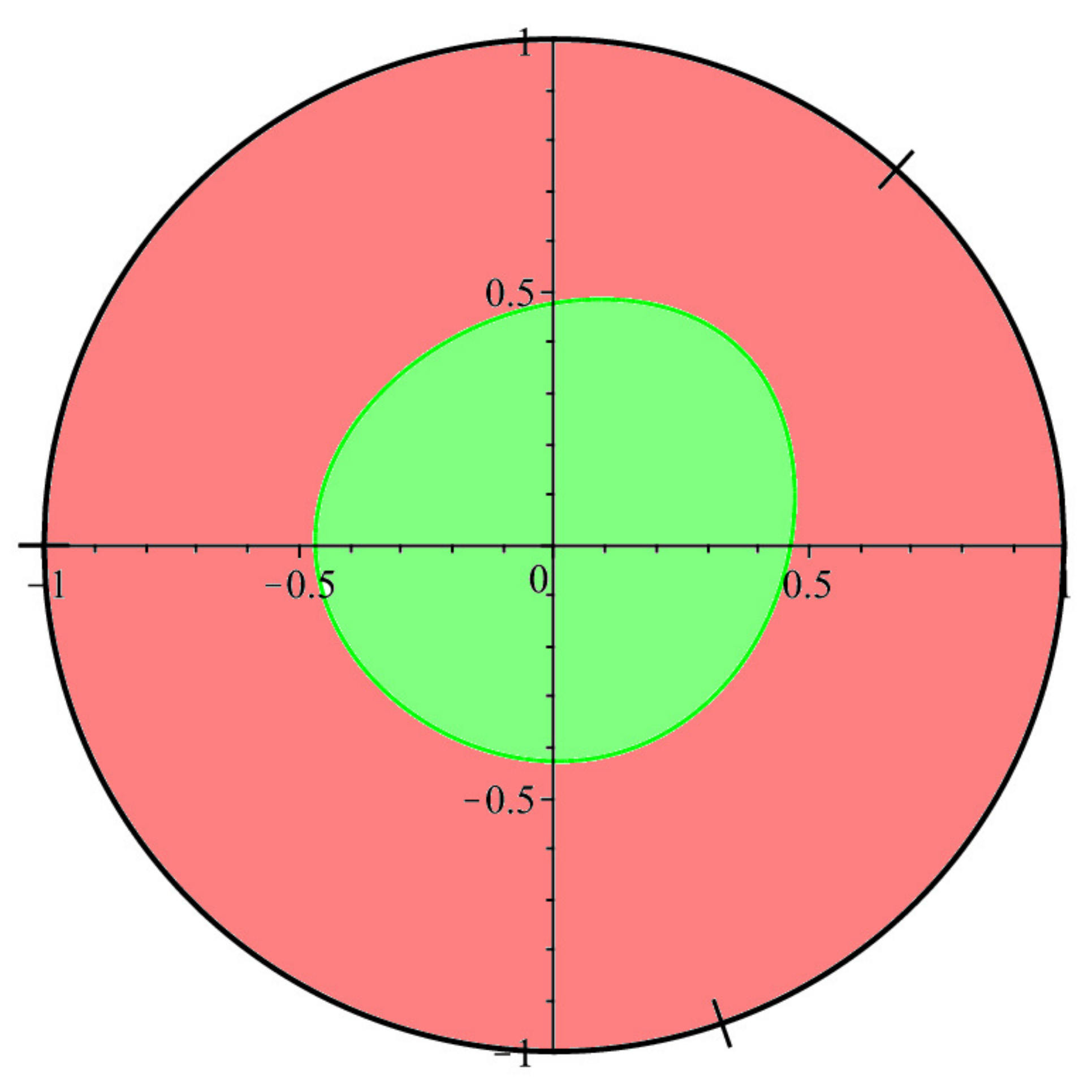}%
\raisebox{4cm}{l)}
\includegraphics[angle=0,width=4cm]{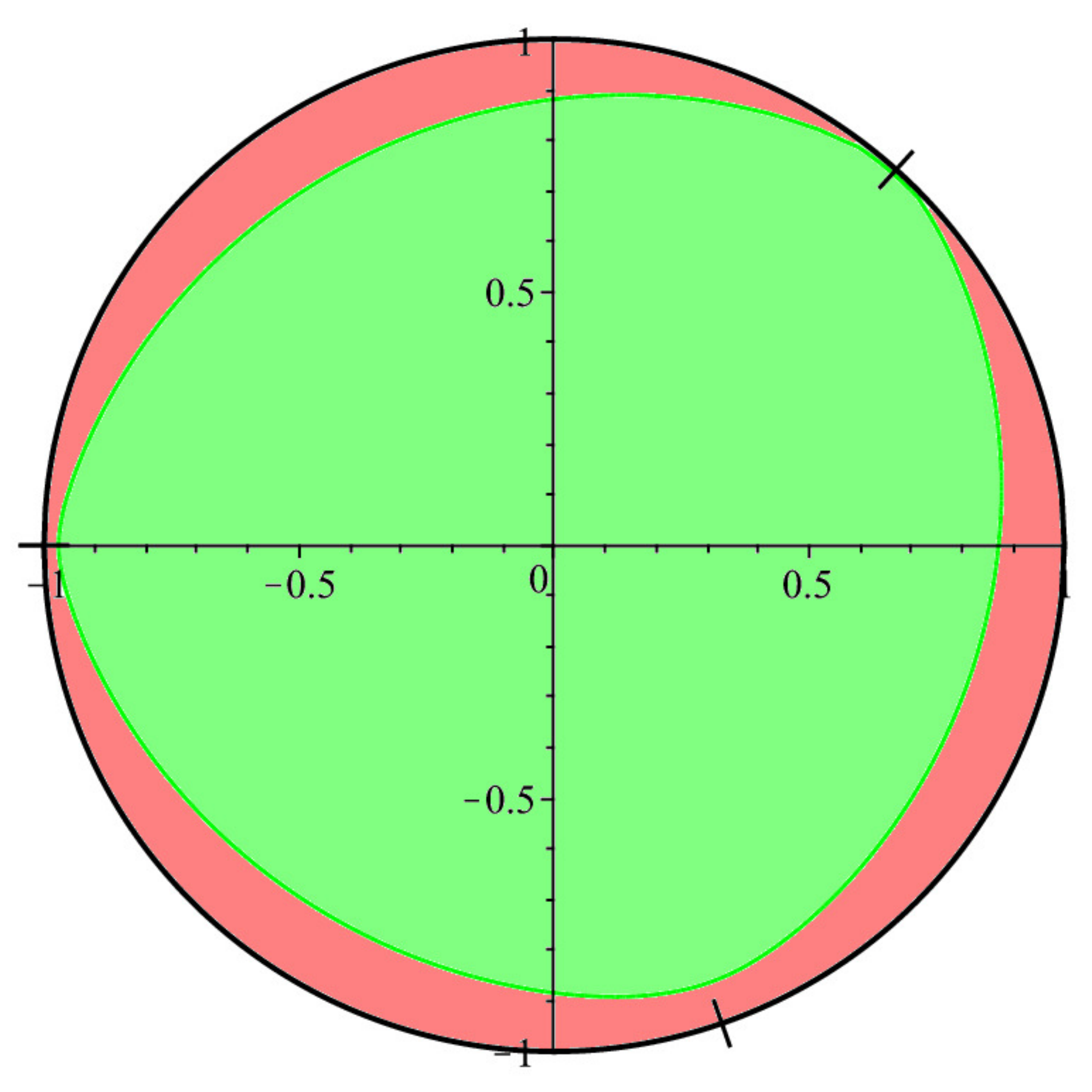}\\ %
\raisebox{4cm}{m)}
\includegraphics[angle=0,width=4cm]{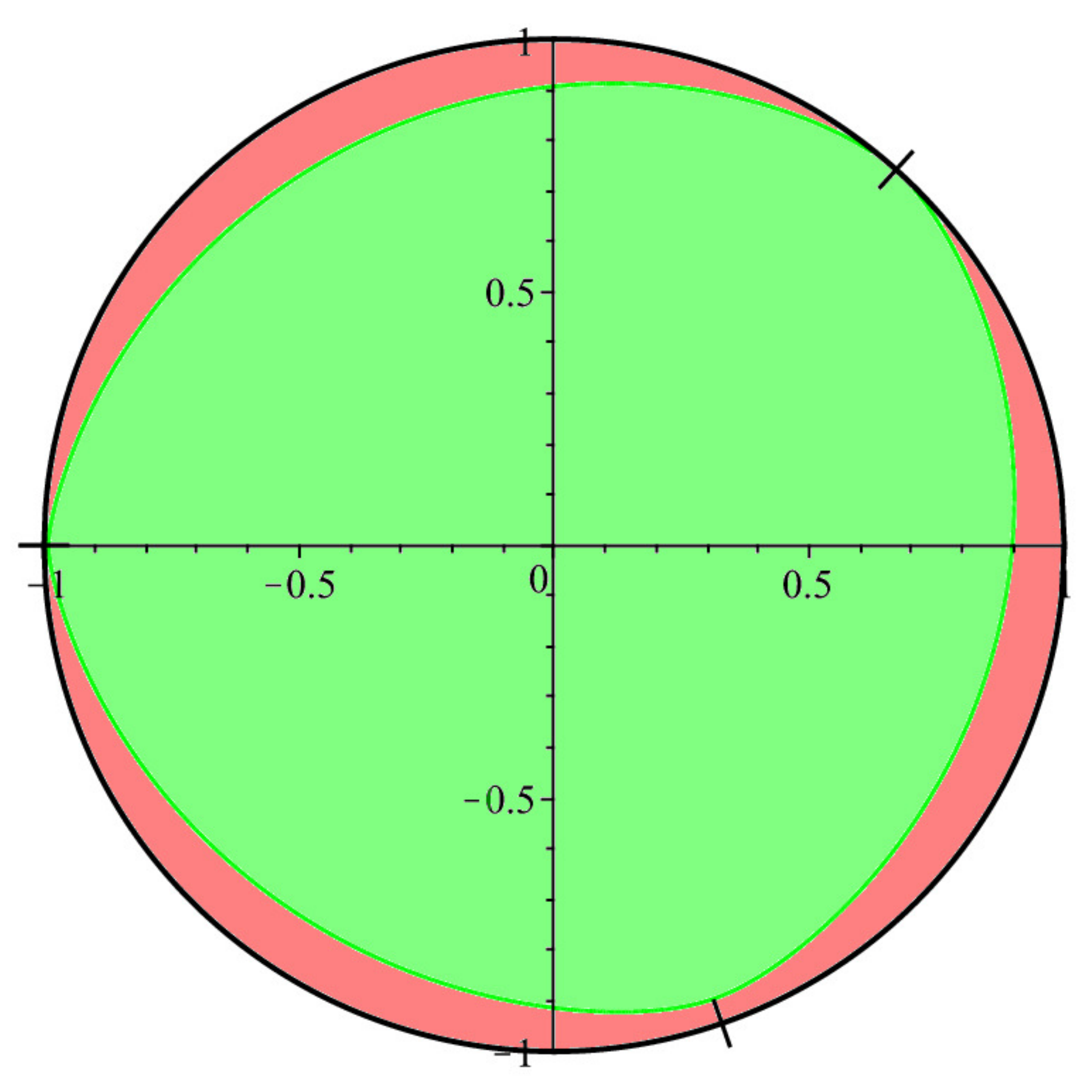} %
\raisebox{4cm}{n)}
\includegraphics[angle=0,width=4cm]{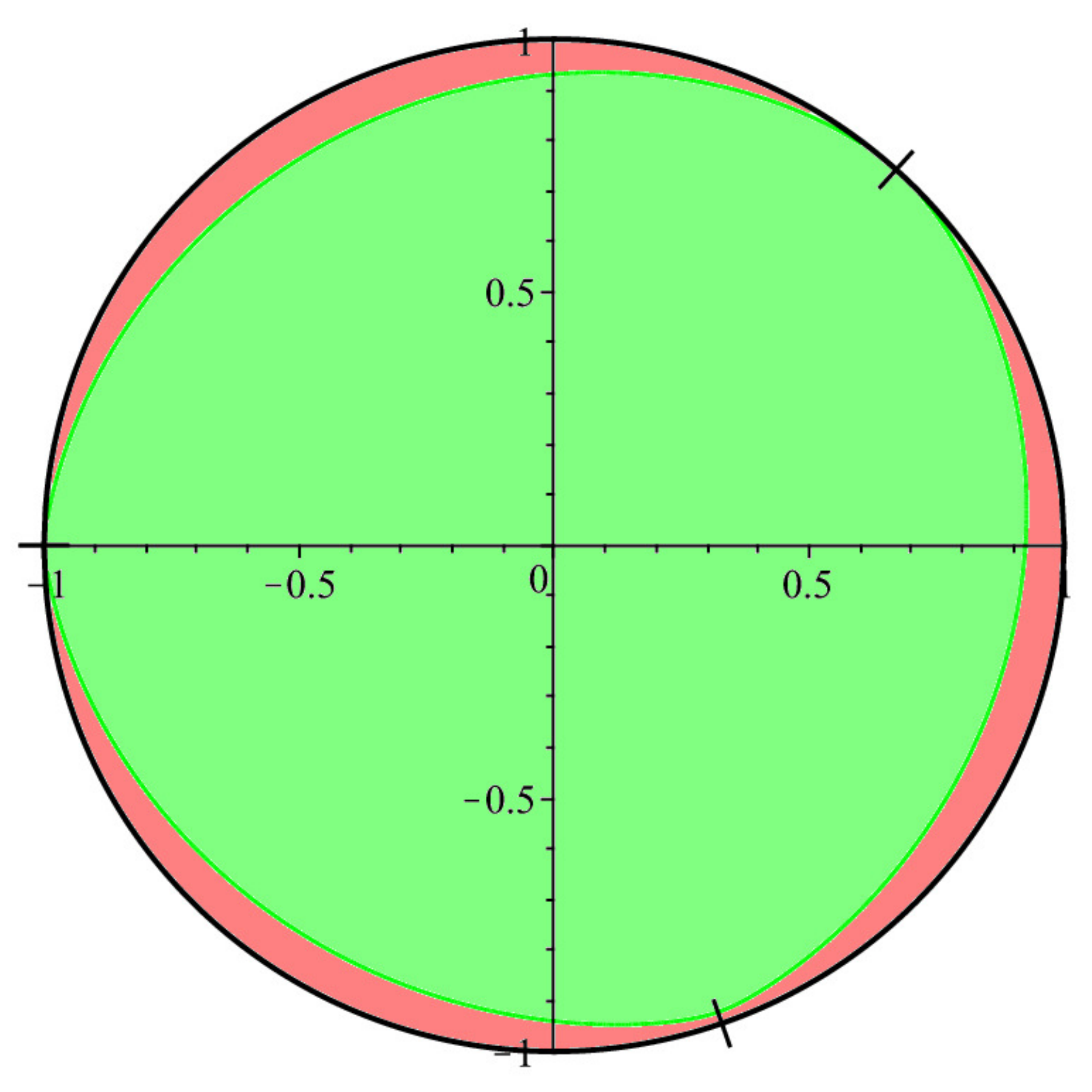}%
\raisebox{4cm}{o)}
\includegraphics[angle=0,width=4cm]{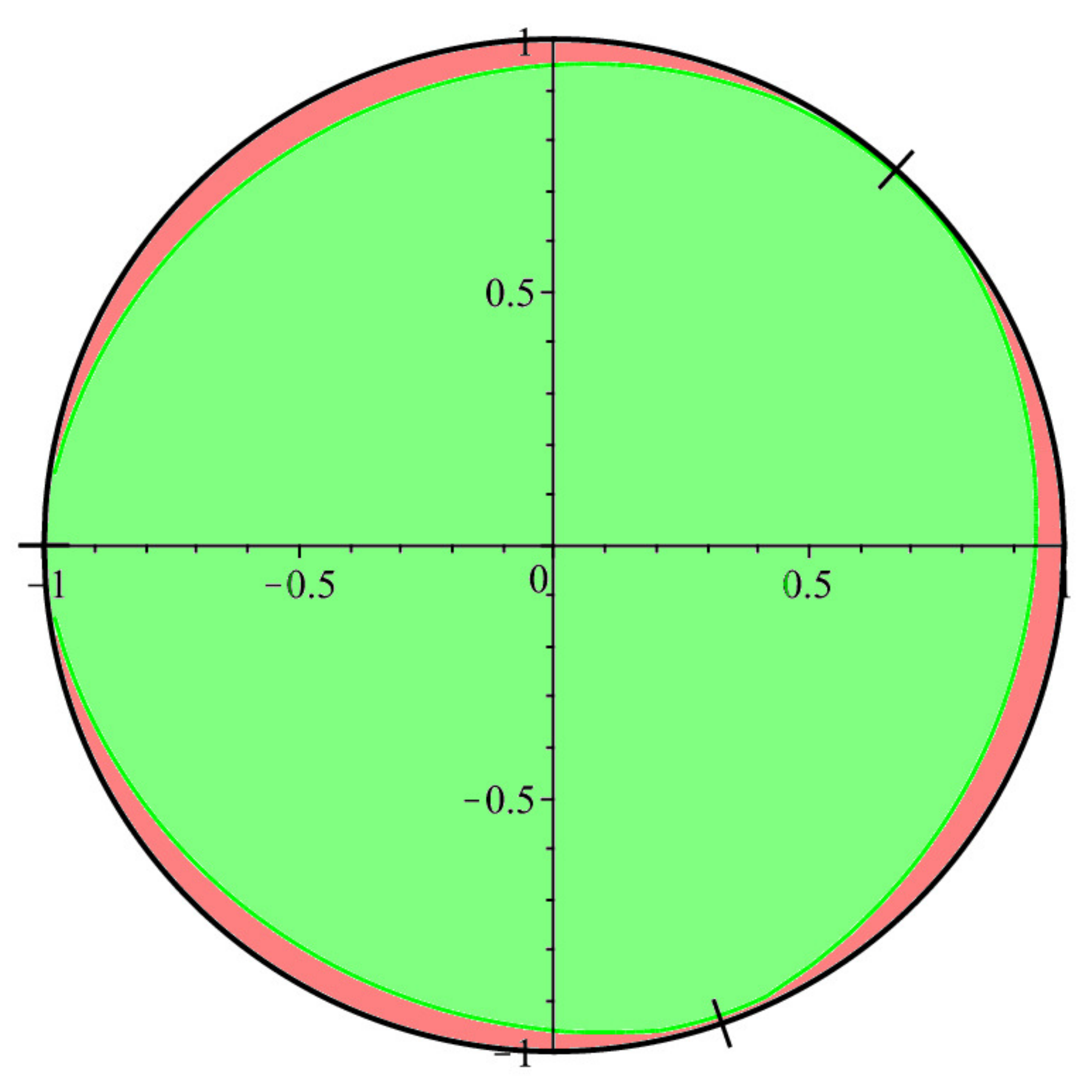}\\ %
\raisebox{4cm}{p)}
\includegraphics[angle=0,width=4cm]{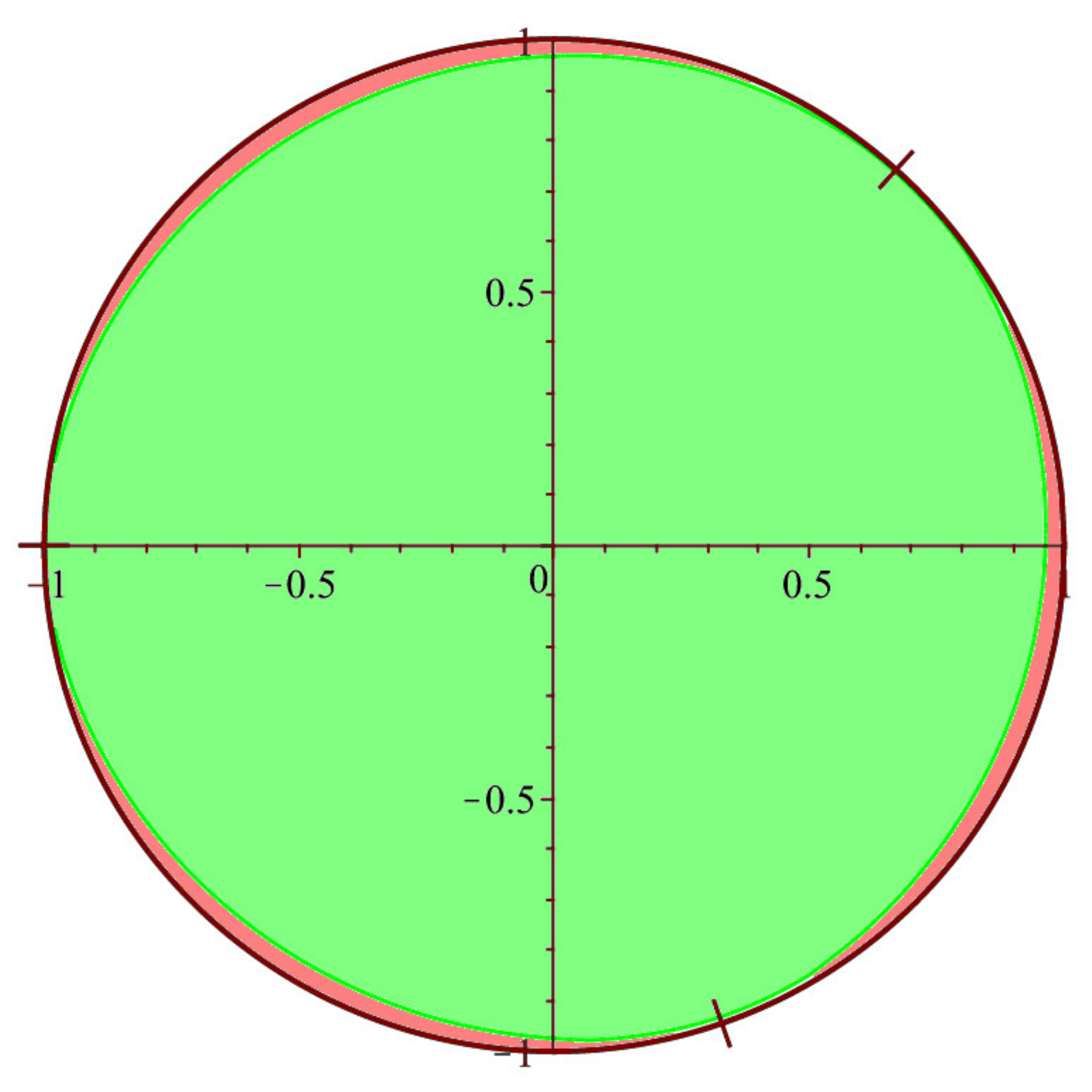}%
\raisebox{4cm}{q)}
\includegraphics[angle=0,width=4cm]{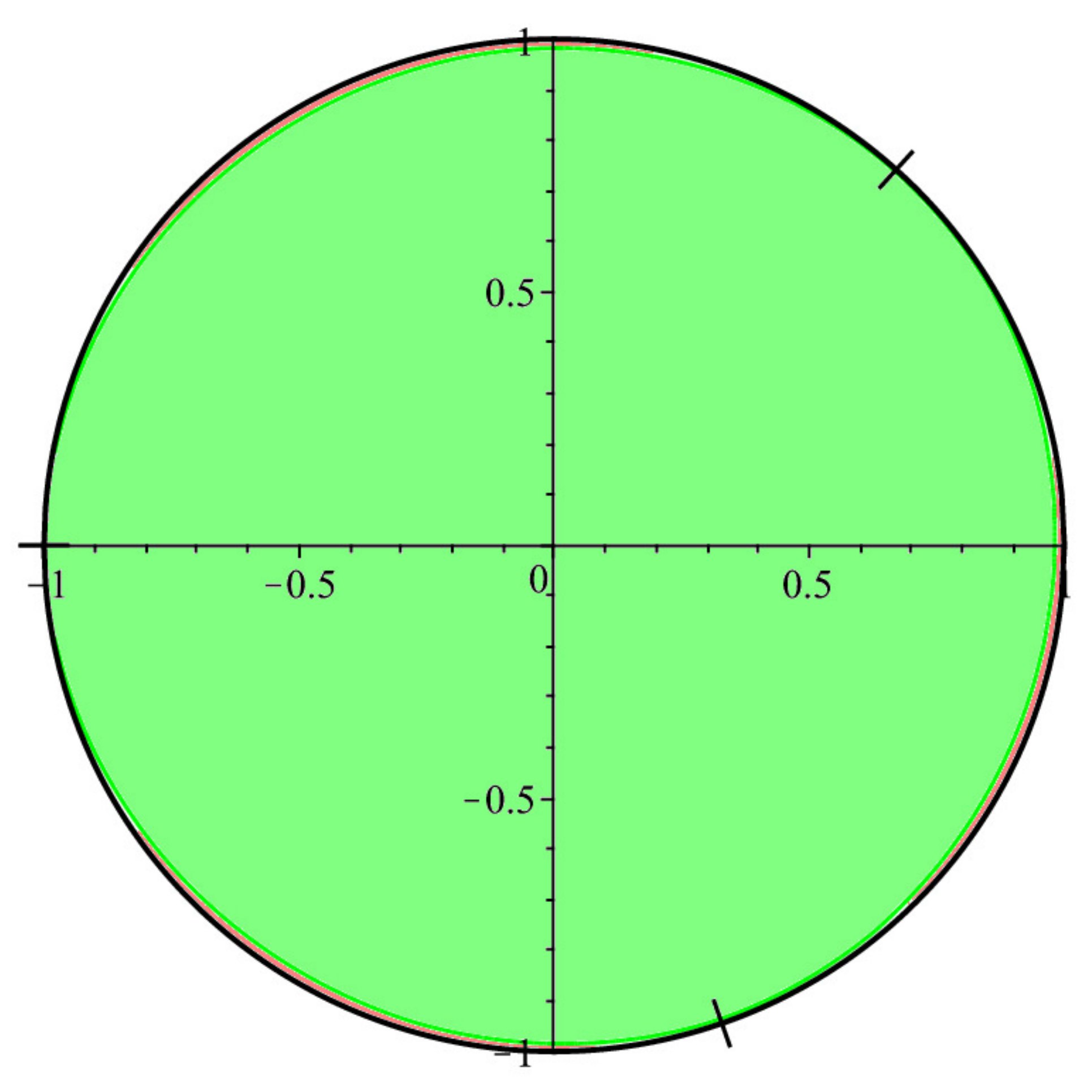}%
\raisebox{4cm}{r)}
\includegraphics[angle=0,width=4cm]{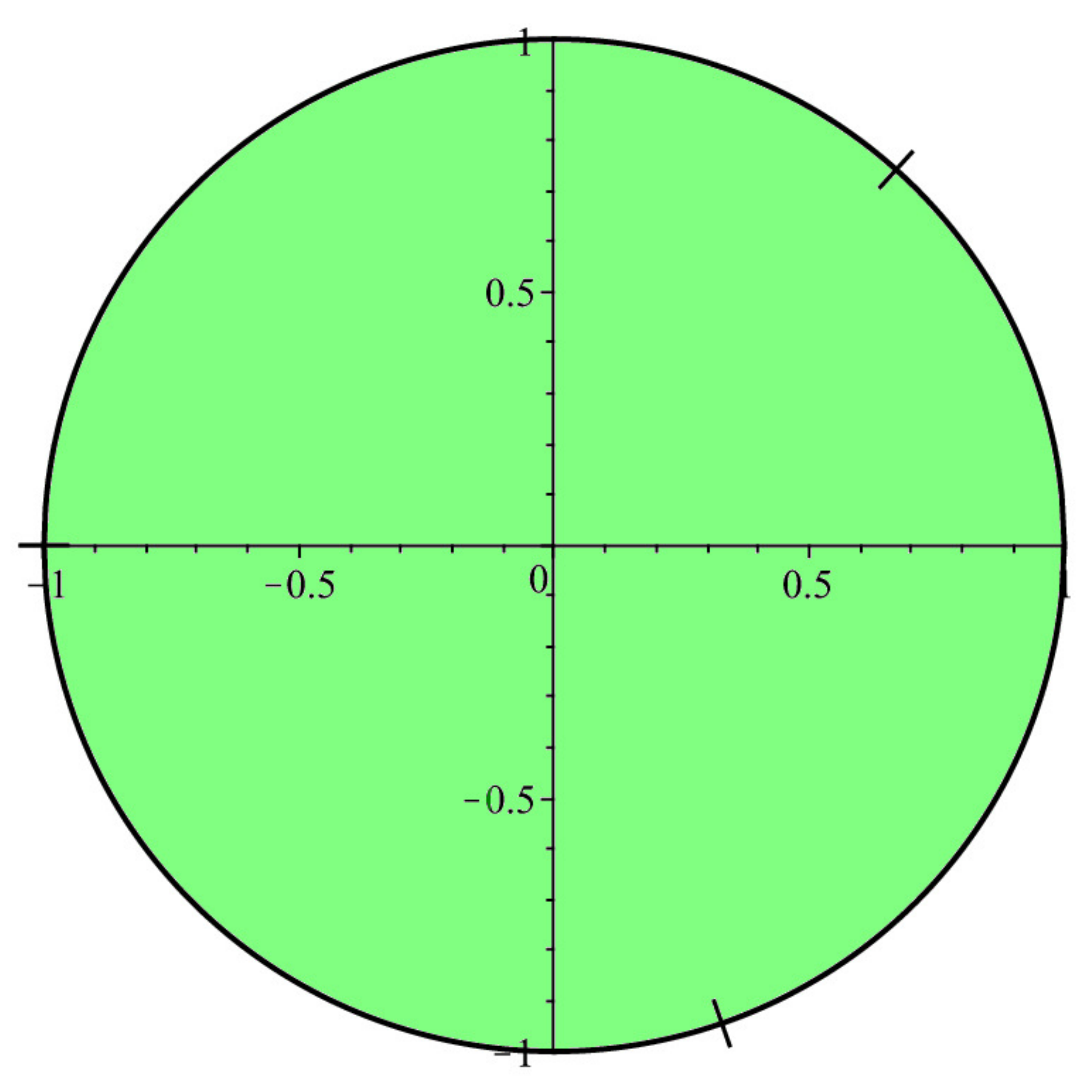} %
\end{center}
\caption{\label{charged_fig:Newton_levelsb}
Gravitational three-body problem:
continuation of Fig.~\ref{charged_fig:Newton_levelsa}  with 
 values of $\nu$ chosen according to 
$  \nu  = \nu_{\text{diabolic}\,5}   $ (j), 
$\nu_{\infty\,4}  >\nu >  \nu_{\text{diabolic}\,5}   $ (k),  
$ \nu = \nu_{\infty\,4}   $ (l), 
$ \nu_{\infty\,3}  >\nu >  \nu_{\infty\,4}   $ (m),  
$ \nu = \nu_{\infty\,3}   $ (n), 
$\nu_{\infty\,2}  >\nu >  \nu_{\infty\,3}   $ (o),  
$ \nu = \nu_{\infty\,2}   $ (p),  
$ \nu_{\infty\,2} > \nu >  \nu_1=0     $ (q), and 
$ \nu_1 = 0 \ge \nu   $ (r). 
}
\end{figure}

When $\nu$ decreases to $\nu_{\text{Lagrange}\,6}$ from above the dark grey region shrinks to a point which for the present choice of masses is close to but not at $ (w_1,w_2,w_3) = (0,0,1) $
and for $\nu<\nu_{\text{Lagrange}\,6}$ it has vanished. The Hill region (projected to the shape space) is hence simply connected for $\nu<\nu_{\text{Lagrange}\,6}$.

The blue region is simply connected for $\nu_{\text{diabolic}\,5}  < \nu < \nu_{\text{Lagrange}\,6}$. 
When $\nu$ decreases to $\nu_{\text{diabolic}\,5} $ the blue region shrinks to a point at $ (w_1,w_2,w_3) = (0,0,1) $, % more to say on this?
and when $\nu$ decreases below $\nu_{\text{diabolic}\,5} $ a simply connected green region grows out of this point. 
For every point in the green region any point on the orientation sphere is accessible. The unusual bifurcation at $\nu=\nu_{\text{diabolic}\,5}$ observed in the projection to the shape space can be understood from noticing that the graphs of the functions \eqref{charged_eq:def_VtildesqrtMk} over the shape space give for  $k=1$ and $k=2$ a continuously distorted version of the double cone in Fig.~\ref{charged_fig:principal_moments}. When the value of $\nu$ passes through $\nu_{\text{diabolic}\,5}$ the sublevel sets of the functions  \eqref{charged_eq:def_VtildesqrtMk} shown in Figs.~\ref{charged_fig:Newton_levelsa}i-\ref{charged_fig:Newton_levelsb}k are continous deformations of the sublevel sets of a height function in Fig.~\ref{charged_fig:principal_moments}.

\rem{
Recall that at $ (w_1,w_2,w_3) = (0,0,1) $, the moments of inertia $\tilde{M}_1$ and $\tilde{M}_2$ are equal (see Remark~\ref{charged_remark:degenerate_moments_inertia}). In this case  the full circle $\mathbf{\tilde{J}}_1^2+ \mathbf{\tilde{J}}_2^2=1$, $\mathbf{\tilde{J}}_3=0$ consists of
critical points of the function \eqref{charged_eq:defGEuler}. The critical value $\nu_{\text{diabolic}\,5} $ is degenerate in this sense.
} % end rem

For $\nu_{\infty\,4}  < \nu < \nu_{\text{diabolic}\,5} $, the green region is completely contained in the interior of the shape space. 
When $\nu$ takes the values $\nu_{\infty\,4} $, $\nu_{\infty\,3} $ and $\nu_{\infty\,2} $ the green region successively starts to touch the boundary of the shape space at the double collision points. The manifestation of the bifurcations due to critical points at infinity in terms of the sublevel sets touching the boundary at the double collision points can be understood from noticing that a critical point at infinity involves one co-rotating pair of two bodies with the third body at rest at infinity. After scaling  such a configuration by the dilation transformation  such that the polar moment of inertia $I$ becomes $1$ the co-rotating pair appears to collide.

At $\nu=\nu_1=0$ the green region fills the full shape space. This remains to be the case for any $\nu\le0$.  

As  the bifurcation scenario is partly difficult to see in Fig.~\ref{charged_fig:Newton_levelsb} we show in Fig.~\ref{charged_fig:Newton_levels_psi_chi} the corresponding contours of the function 
$  \sqrt{\tilde{M}_1  } \, \tilde{V} $ in the $( \chi , \psi )$-plane, $0\le \chi \le 2\pi$, $0\le\psi \le \pi / 2$.

  \begin{figure}
\begin{center}
\raisebox{4cm}{j)}
\includegraphics[angle=0,width=4cm]{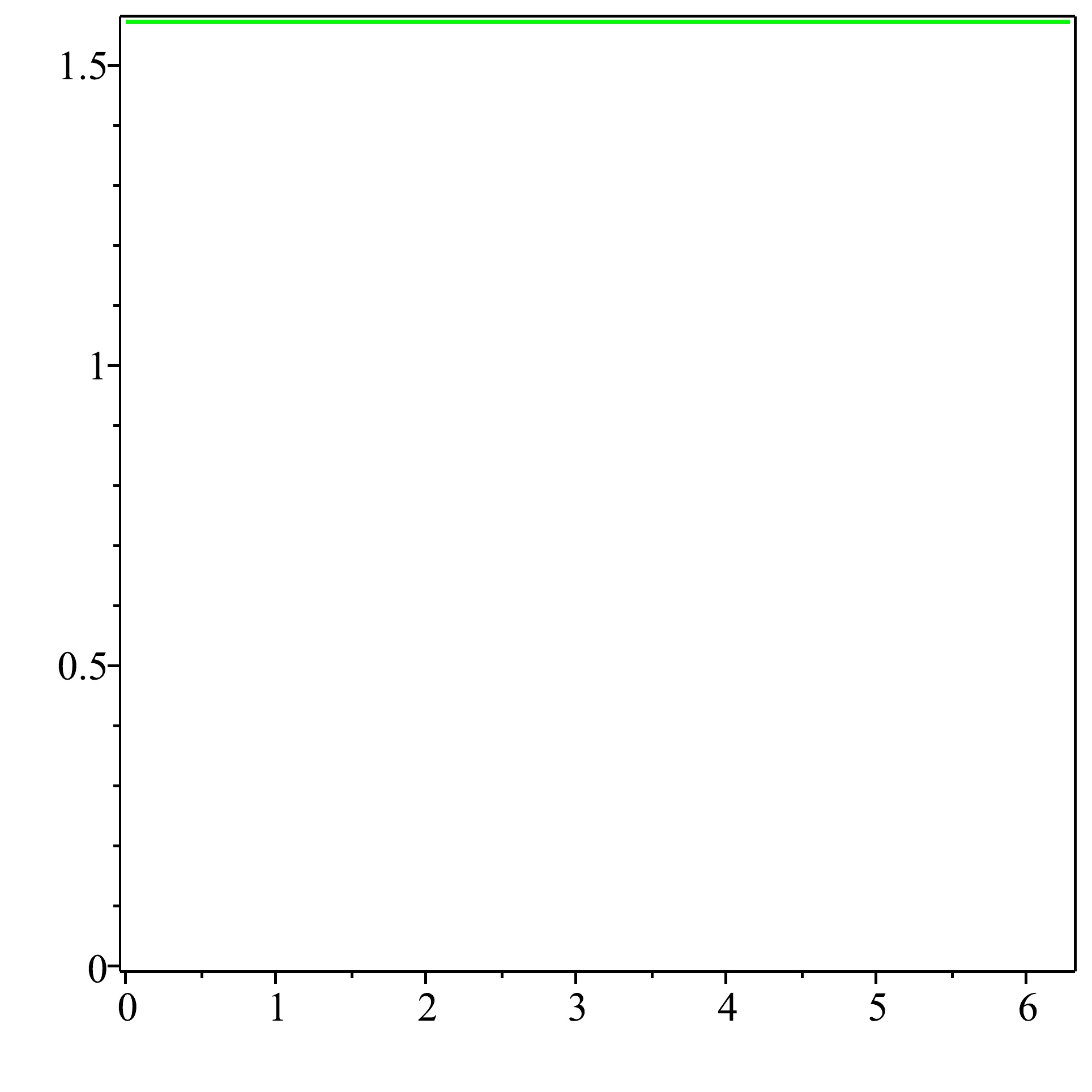} %
\raisebox{4cm}{k)}
\includegraphics[angle=0,width=4cm]{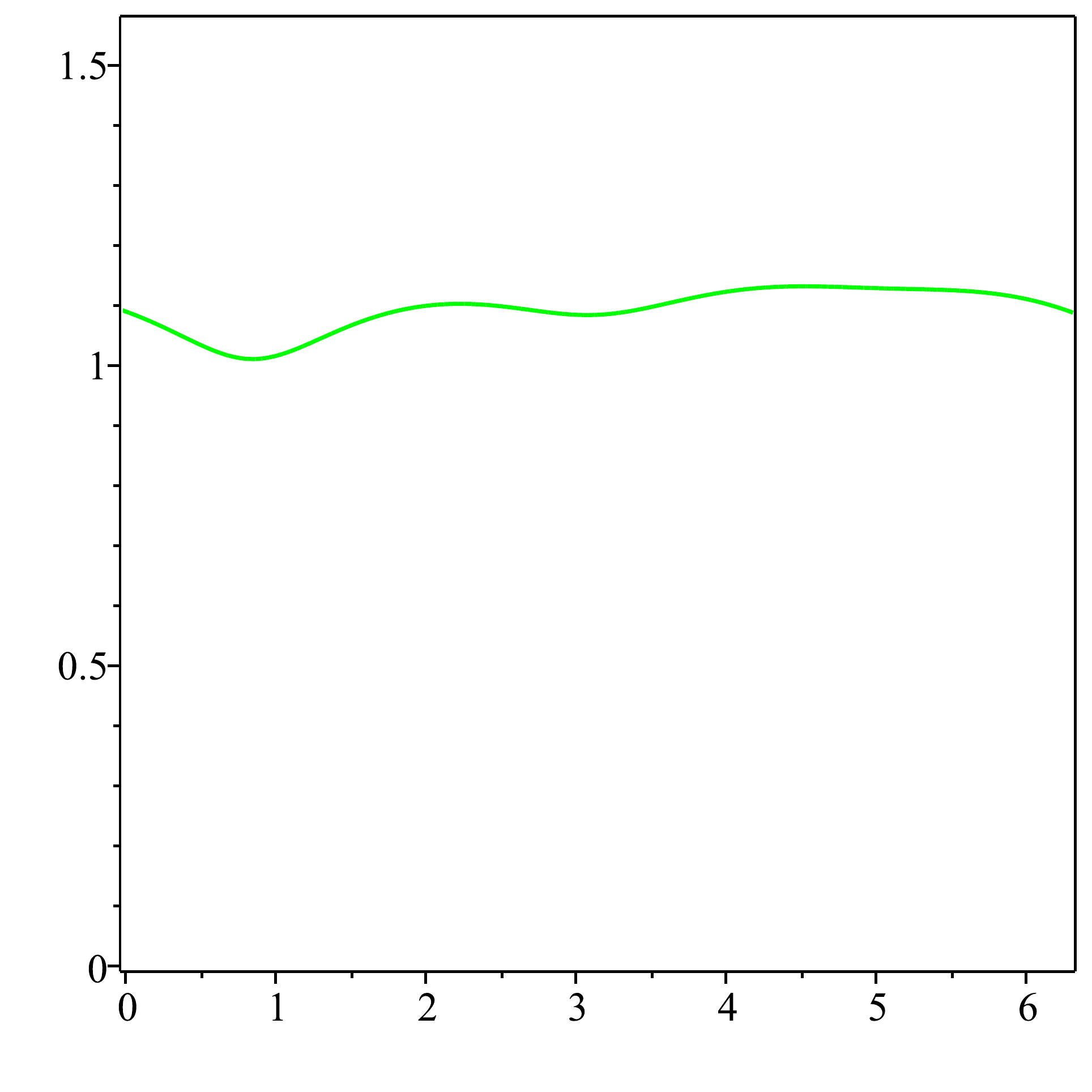} %
\raisebox{4cm}{l)}
\includegraphics[angle=0,width=4cm]{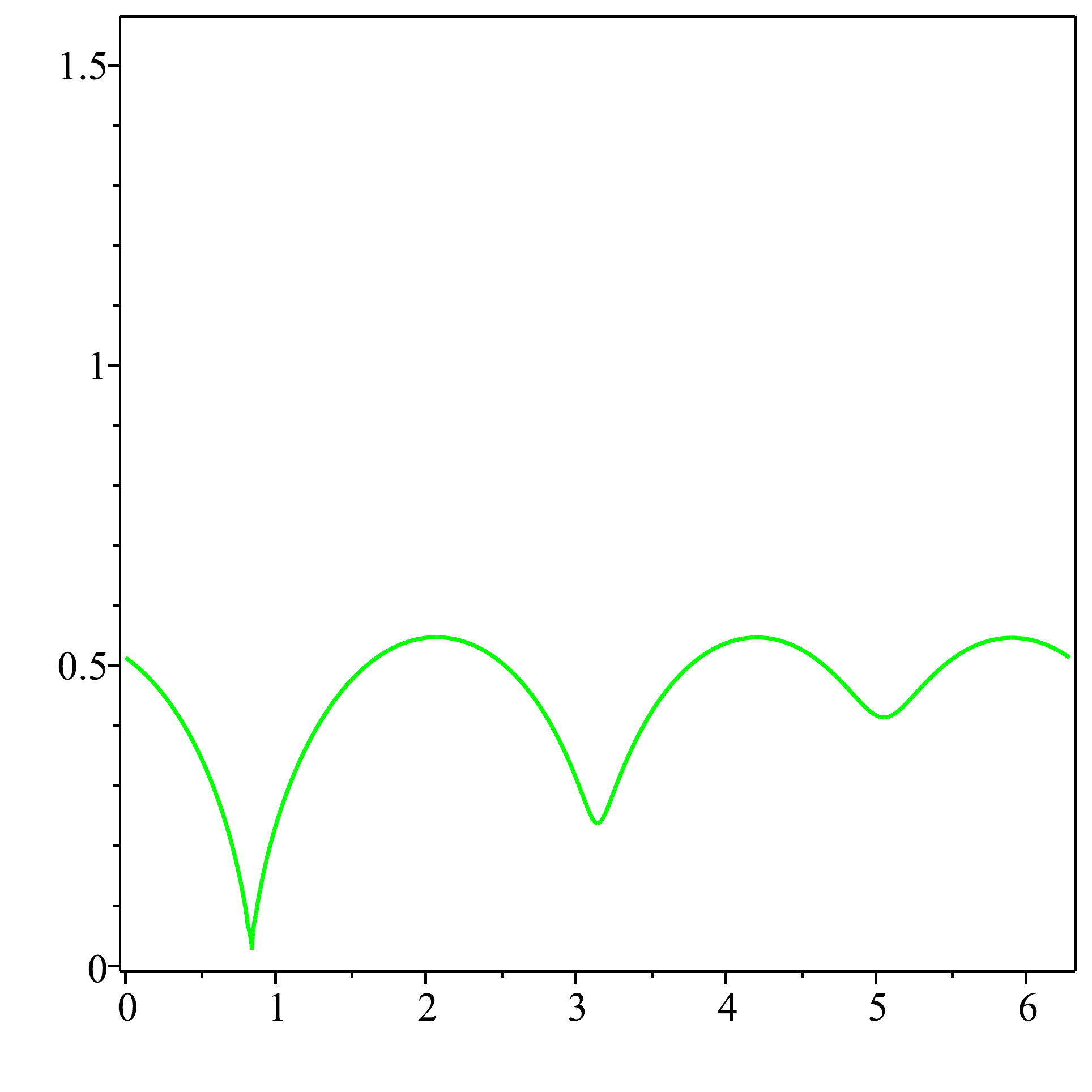}\\ %
\raisebox{4cm}{m)}
\includegraphics[angle=0,width=4cm]{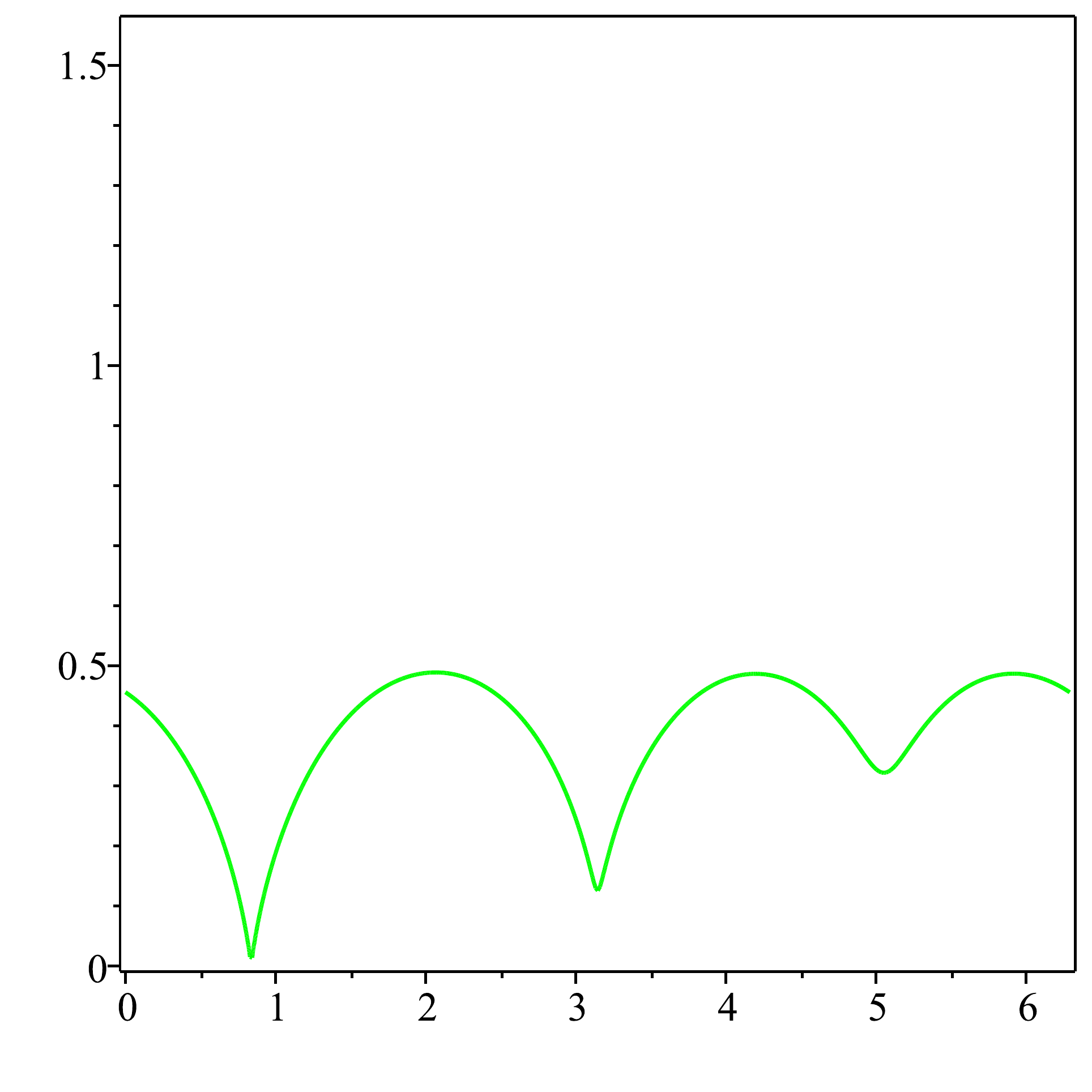}%
\raisebox{4cm}{n)}
\includegraphics[angle=0,width=4cm]{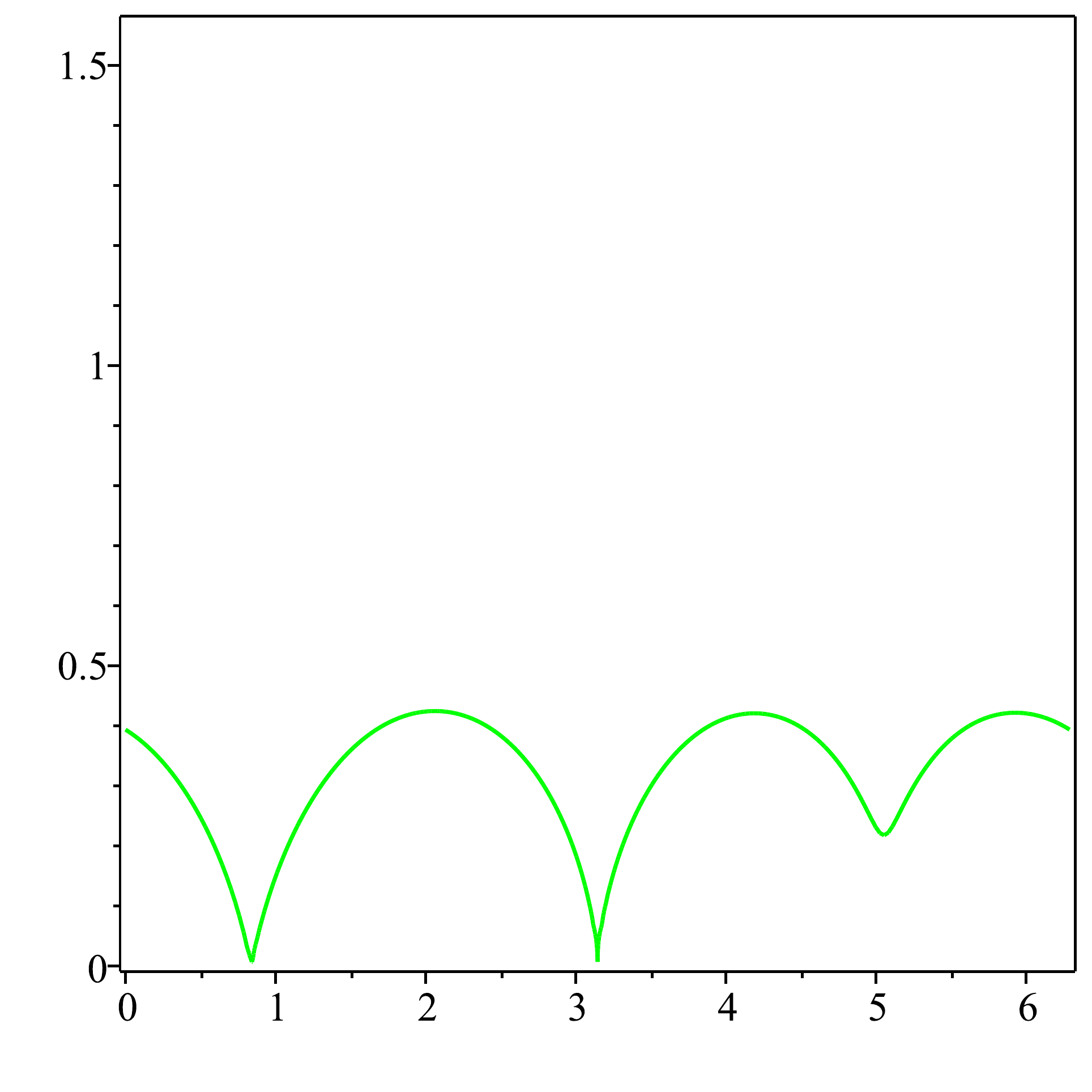} %
\raisebox{4cm}{o)}
\includegraphics[angle=0,width=4cm]{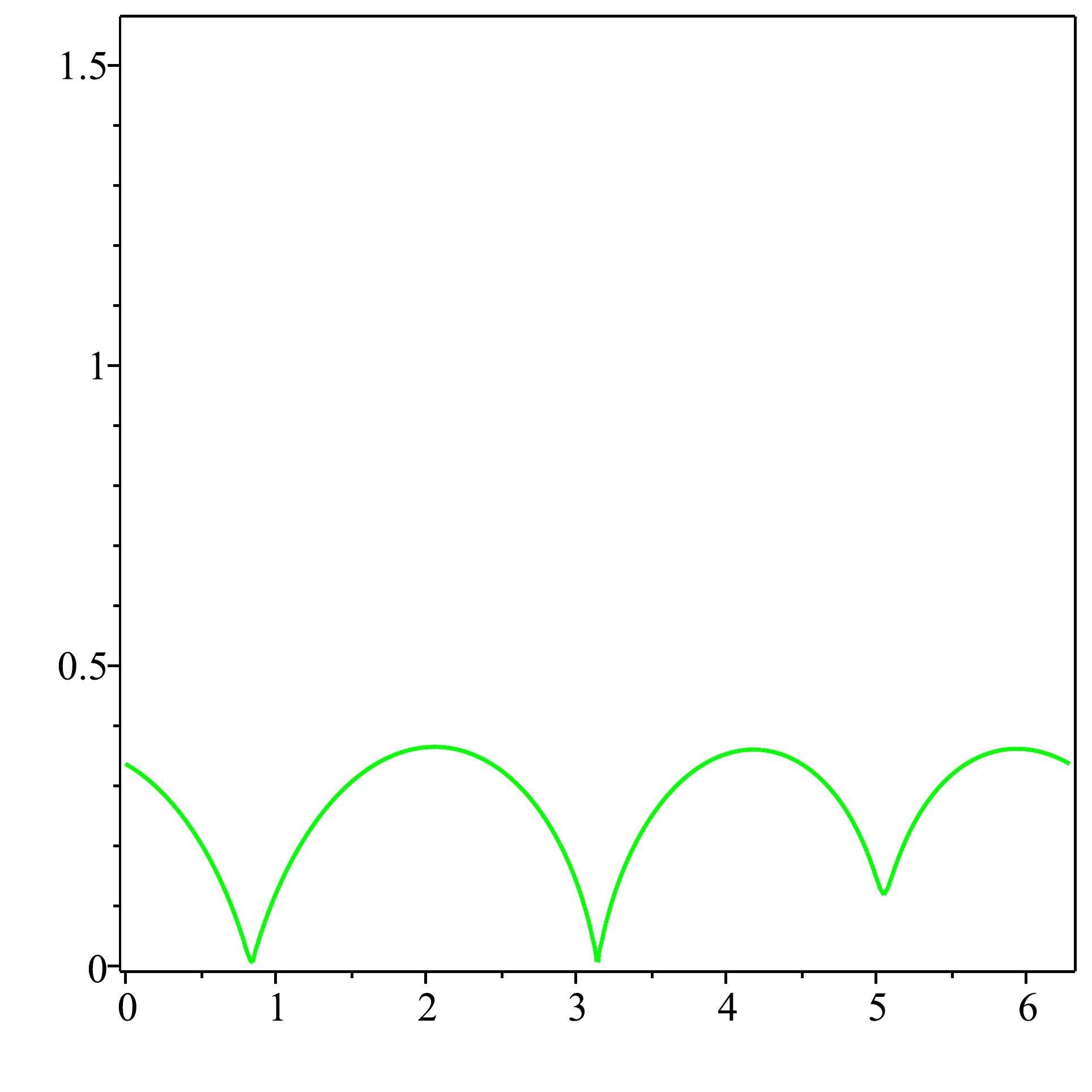}\\ %
\raisebox{4cm}{p)}
\includegraphics[angle=0,width=4cm]{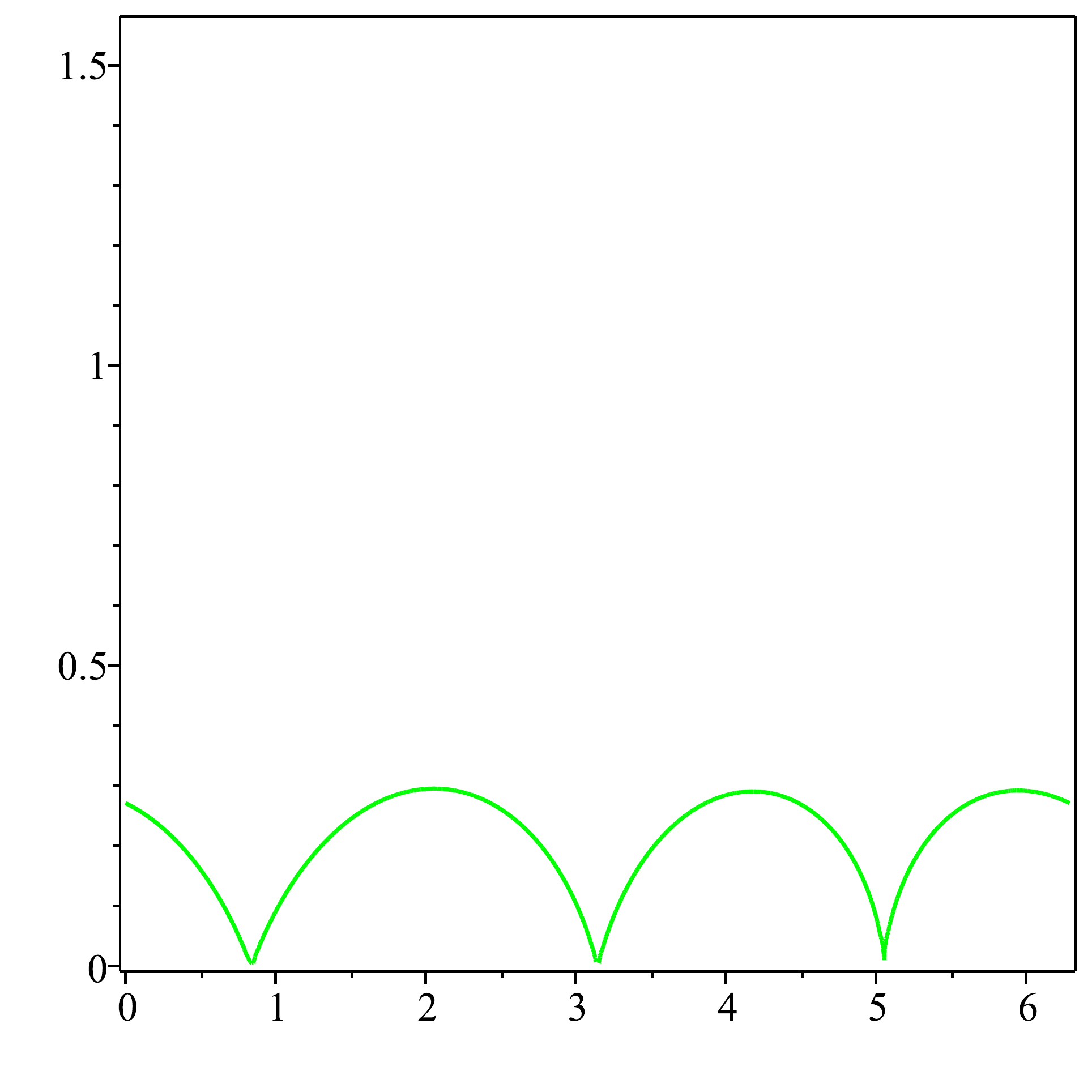}%
\raisebox{4cm}{q)}
\includegraphics[angle=0,width=4cm]{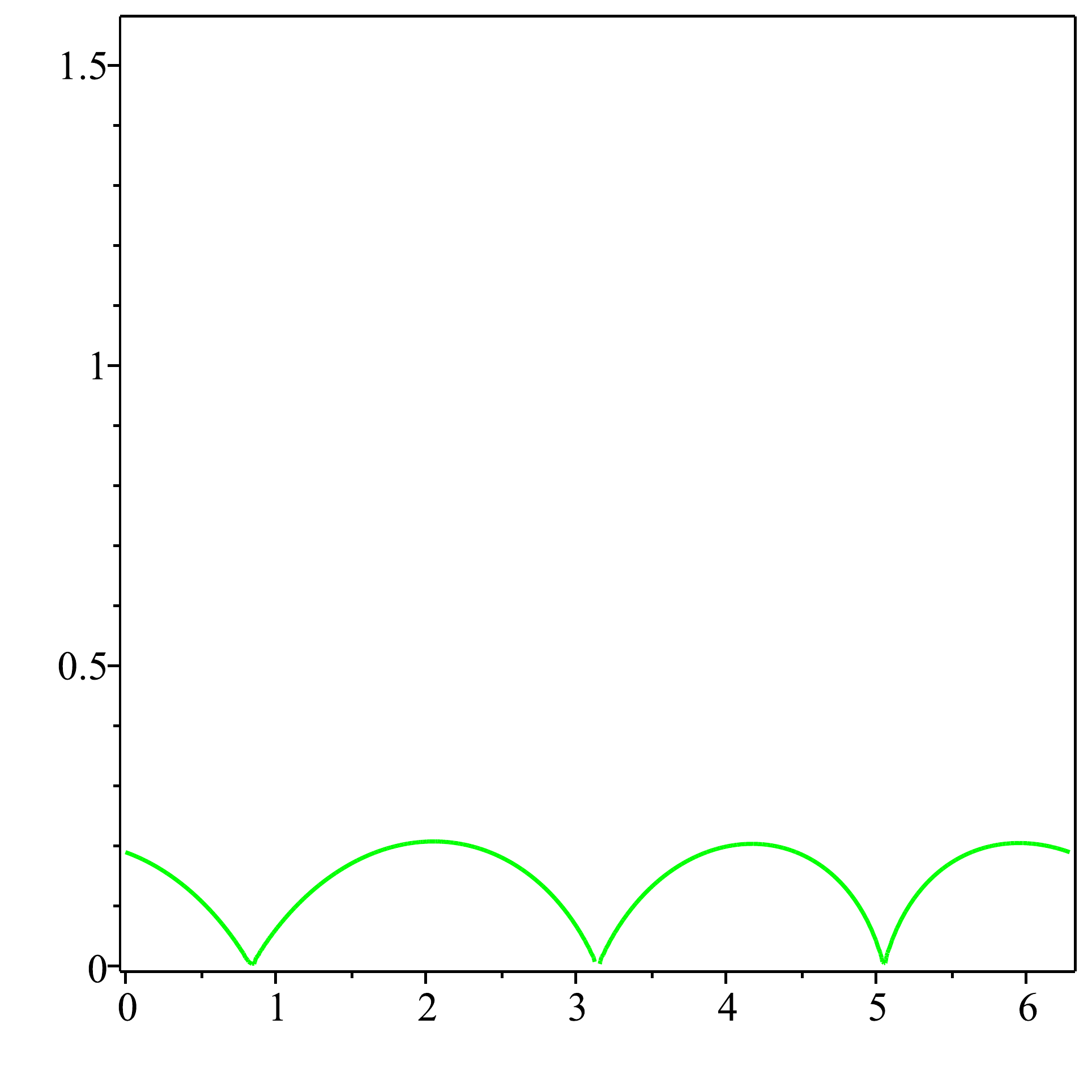} %
\raisebox{4cm}{r)}
\includegraphics[angle=0,width=4cm]{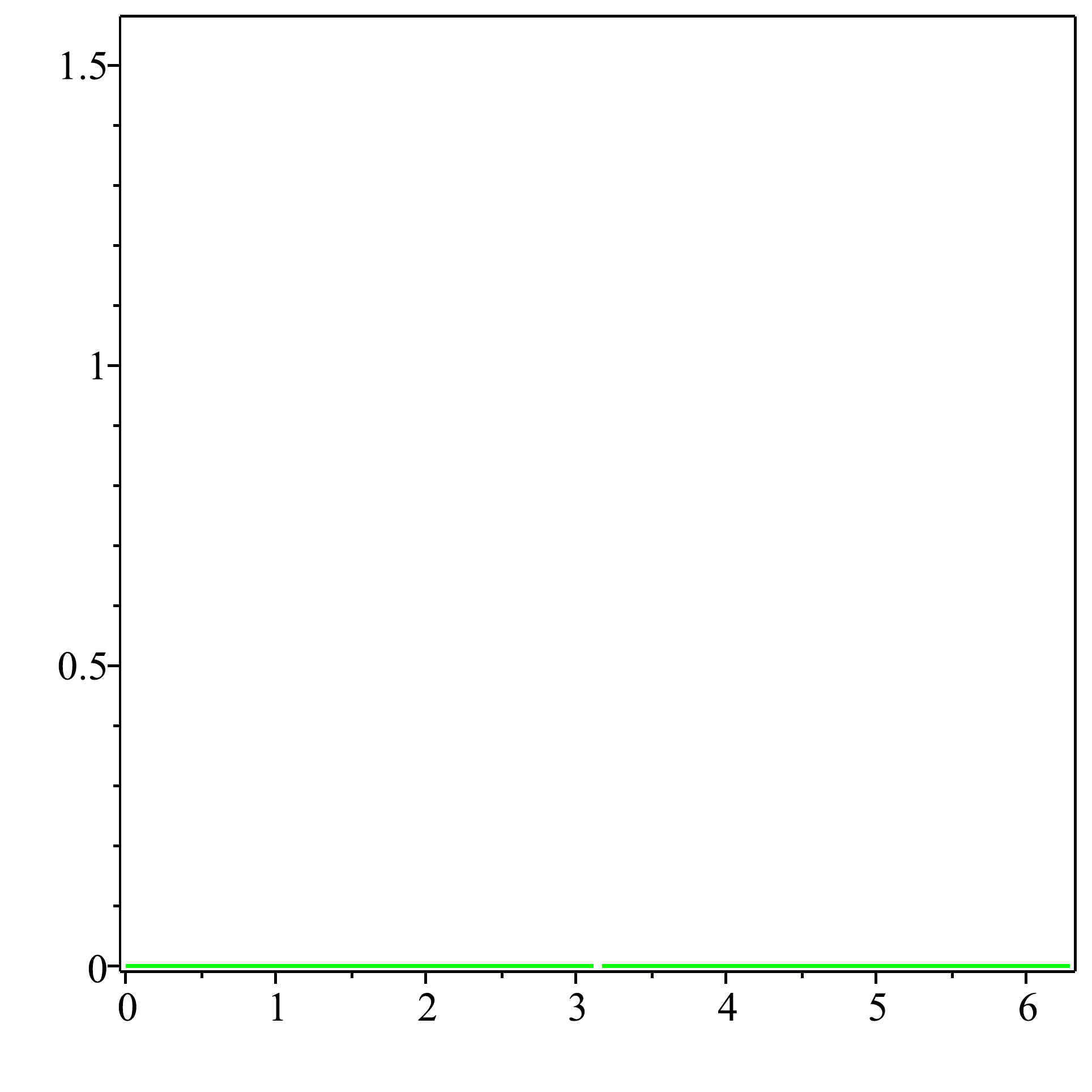}%
\end{center}
\caption{\label{charged_fig:Newton_levels_psi_chi}
Gravitational three-body problem:
contours of the function  $  \sqrt{\tilde{M}_1  } \, \tilde{V} $ in the $( \chi , \psi )$-plane, $0\le \chi \le 2\pi$, $0\le\psi \le \pi / 2$,
 for the same values of $\nu$ as in the corresponding panels in Fig.~\ref{charged_fig:Newton_levelsb}.
}
\end{figure}

We note that bifurcations observed in  Figs.~\ref{charged_fig:Newton_levels} and \ref{charged_fig:Newton_levelsb} agree with
similar figures shown in \cite{mccord1998integral}.

%%%%%%%%%%%%%%%%%%%%%%

\subsection{Helium}
\label{charged_sec:Helium}

We here consider the helium atom as a charged three-body system consisting of two electrons and a nucleus that interact via Coulomb forces. Using atomic units the electrons have masses $m_1=m_2=1$, the nucleus has mass 
$m_3 = 7\,289.56$. The electron both have charge -1. The nucleus has charge $+2$. The coefficients in the potential \eqref{eq:potential3b} are then  $\alpha_1=\alpha_2=2$ and
$\alpha_3=-1$.

In Fig.~\ref{charged_fig:Helium_levels_general} we show again the contours of the functions $ \sqrt{\tilde{M}_k  } \, \tilde{V} $, $k=1,2,3$, defined in \eqref{charged_eq:def_VtildesqrtMk}  analogously to Fig.~\ref{charged_fig:Newton_levels} for the gravitational case.
The collision of particles 1 and 2 (two electrons) occurs at a polar angle of about $89.99214109^\circ$ 
and of particles 2 and 3 (electron and nucleus) at a polar angle of about $-0.01571780034^\circ$ (see \eqref{charged_eq:angle_collision_12} and \eqref{charged_eq:angle_collision_23}). The collision of particles 1 and 3 (again electron and nucleus) is located  at $180^\circ$. As the double collision are close to the coordinate axes in Fig.~\ref{charged_fig:Helium_levels_general} we do not mark them by special ticks.
The potential $\tilde{V}$ is $-\infty$ at the collisions of either of  the  electrons with the nucleus (polar angles $-0.01571780034^\circ$ and $180^\circ$) and  
$+\infty$ at the collisions of the two electrons  (polar angle $89.99214109$).

\begin{figure}
\begin{center}
\raisebox{4cm}{a)}
\includegraphics[angle=0,width=4cm]{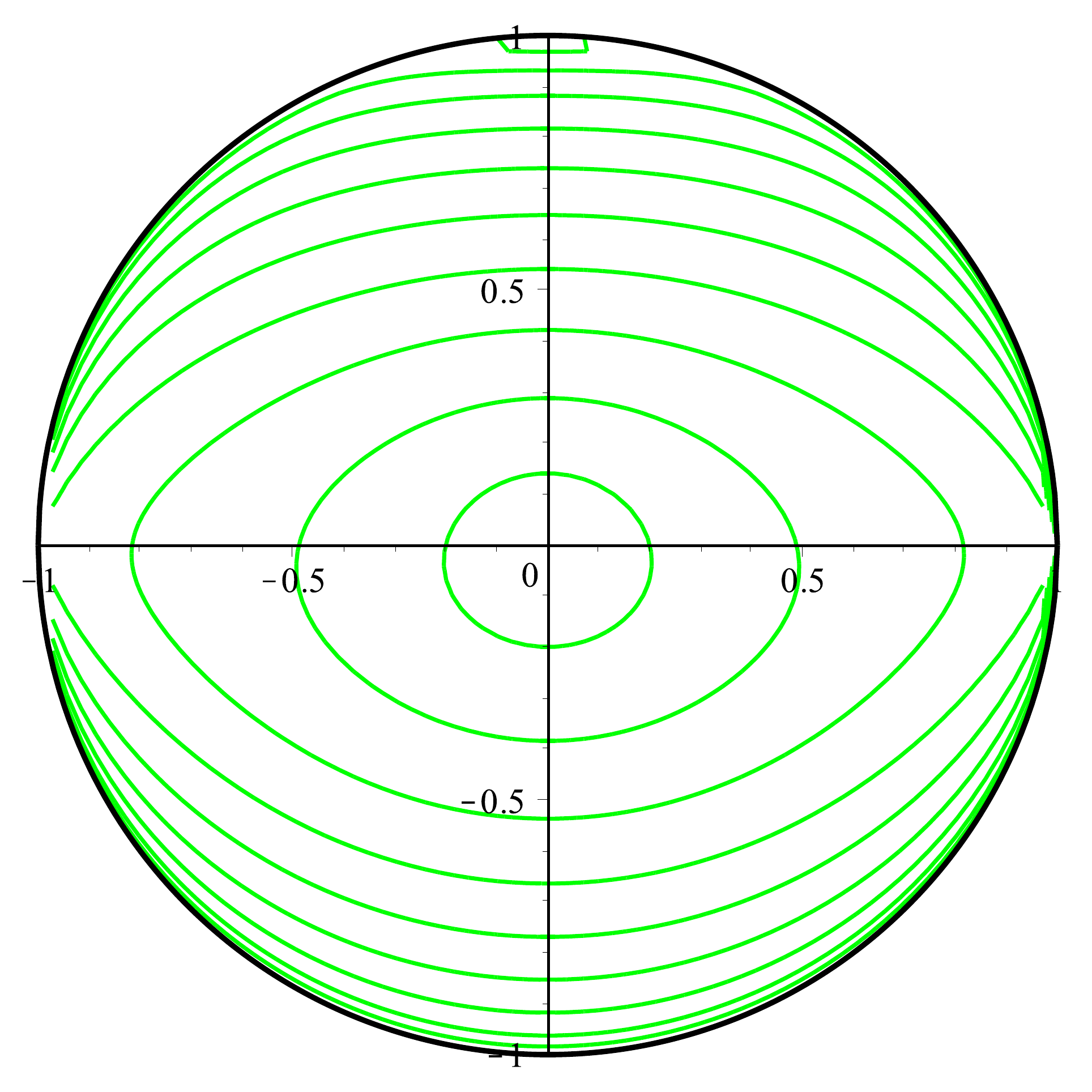}
\raisebox{4cm}{b)}
\includegraphics[angle=0,width=4cm]{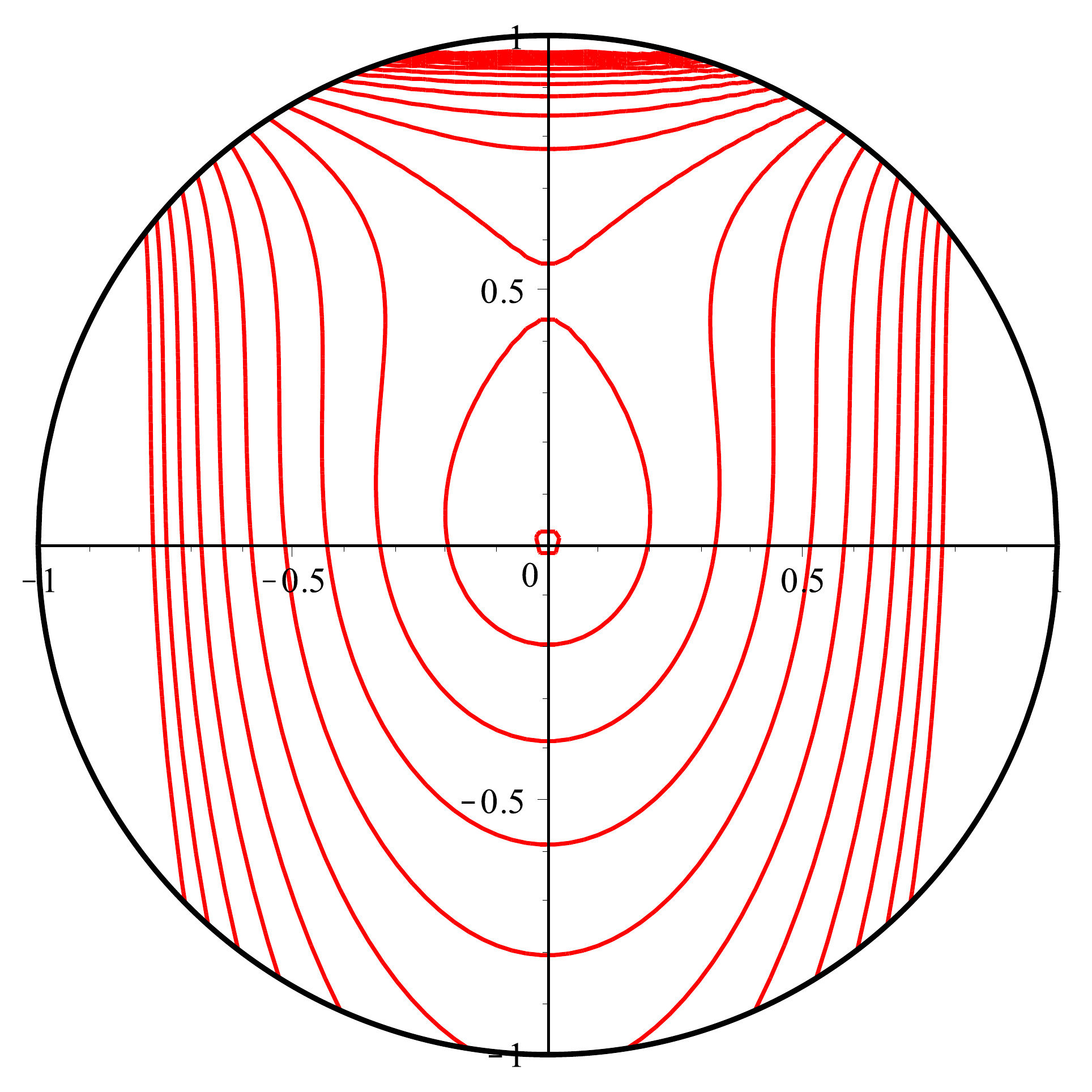}
\raisebox{4cm}{c)}
\includegraphics[angle=0,width=4cm]{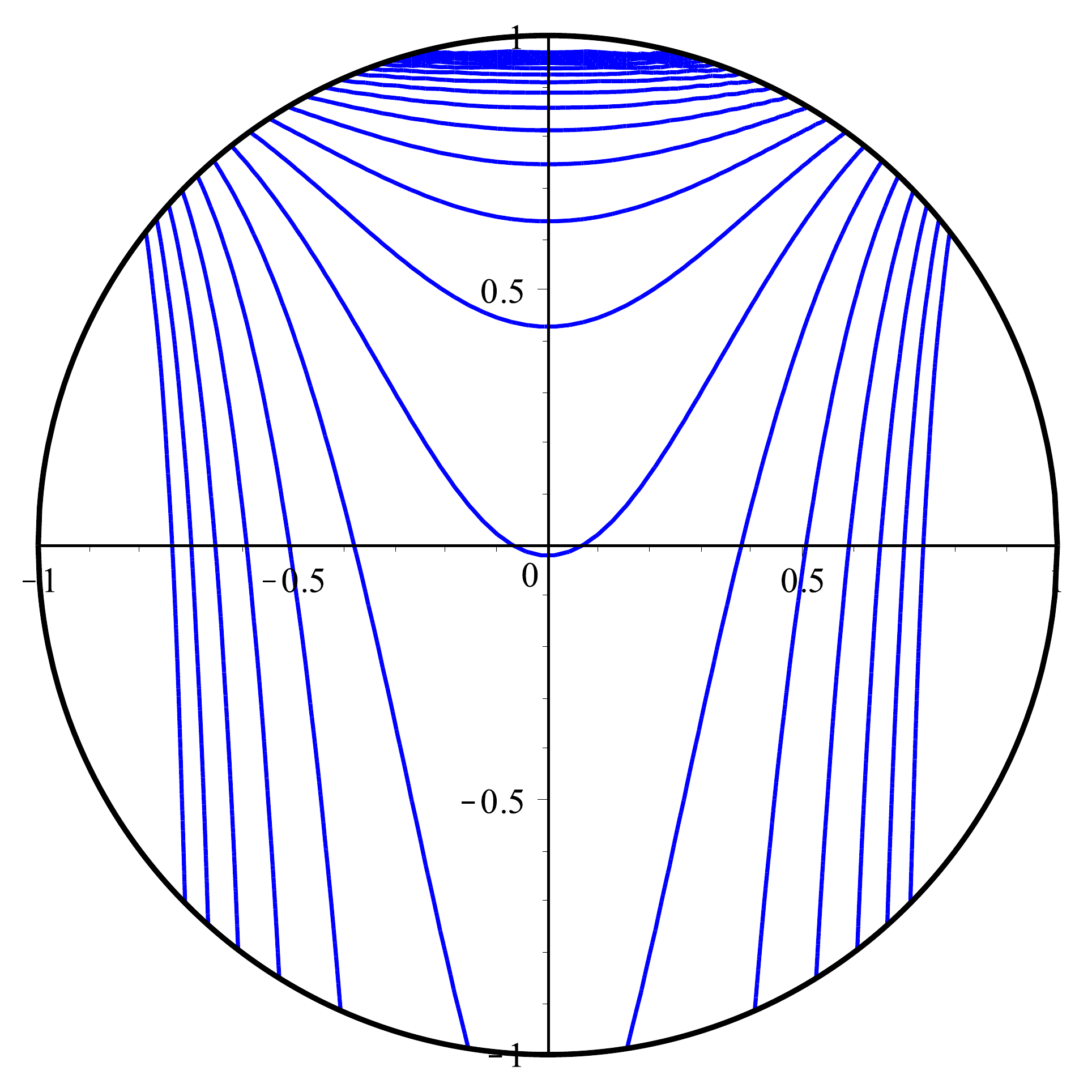}
\end{center}
\caption{\label{charged_fig:Helium_levels_general}
Helium:  
Contours  $-n/3$ for $n=1,\ldots,20$, of the functions
$  \sqrt{\tilde{M}_k  } \, \tilde{V} $ 
defined in \eqref{charged_eq:def_VtildesqrtMk} with $k=1$ (a),  $k=2$ (b) and $k=3$ (c). 
The presentation is analogously to Fig.~\ref{charged_fig:Newton_levels}.  
}
\end{figure}

From the contours in Fig.~\ref{charged_fig:Helium_levels_general} we conclude that there are together with the critical value $\nu_1=0$  (where $\nu_1=0$ is further commented on in the discussion of Fig.~\ref{charged_fig:Helium_contours} below) in total 5 critical values of $\nu$ resulting from the following events.
\begin{itemize}
\item[(i)]  The contours of $\sqrt{\tilde{M}_1}\tilde{V}$ touch/detach from the boundary of the shape space (see Fig.~\ref{charged_fig:Helium_levels_general}(a)). This happens simultaneously at both of the double collision points of the electrons with the nucleus.

\item[(ii)] $\sqrt{\tilde{M}_1}\tilde{V}$ and $\sqrt{\tilde{M}_1}\tilde{V}$ simultaneously have a critical point at the center of the unit disk in the $(w_1,w_2)$-plane (see Fig.~\ref{charged_fig:Helium_levels_general}(a) and (b)).

\item[(iii)]  $\sqrt{\tilde{M}_1}\tilde{V}$ has a critical point close to the positive $w_2$-axis away from the center of the unit disk in the $(w_1,w_2)$-plane (see Fig.~\ref{charged_fig:Helium_levels_general}(b)).

\item[(iv)] The contours of $\sqrt{\tilde{M}_2}\tilde{V}$ and $\sqrt{\tilde{M}_3}\tilde{V}$ touch/detach from the boundary of the shape space  at polar angle $-120^\circ$ (see Figs.~\ref{charged_fig:ElectronsPositron_levels_general}(b) and (c)). This again happens for the same value of $\nu$ as $\sqrt{\tilde{M}_2}\tilde{V}$ and $\sqrt{\tilde{M}_3}\tilde{V}$ agree on the boundary of the shape space (i.e. for collinear configurations).

\end{itemize}

These events can be related to the following   nontrivial critical values of $\nu$. 
Event (i) is a result of the two symmetry related critical points at infinity where either of the electrons is co-rotating with the nucleus and the other electron is at rest with an infinite distance in between. The corresponding value of $\nu$ can be computed from  \eqref{charged_eq:nu_for_corotating_charges}. For the helium nucleus and one electron, this gives

\begin{equation}
\nu_{\infty\,2} \approx 1.999725672\,.
\end{equation}
Event (ii) results from the diabolic point of the moment of inertia tensor.
 For the helium system, the corresponding value of $\nu$ can be computed from \eqref{eq:nu_diabolic} to be

\begin{equation}
\nu_{\text{diabolic}\,3} \approx 5.420669550\,.
\end{equation}

Event (iii) is caused by a non-collinear relative equilibrium involving rotation about the middle principal axis. This is the Langmuir orbit discussed in Sec.~\ref{sec:Langmuir}. 
For systems like the helium atom where  $\alpha_2 / \alpha_3 =-2$, \eqref{eq:app_nu_general} reduces to 
\begin{equation}\label{eq:app_nu_special}
\nu_{\text{Langmuir}} = \left(\frac{3}{2} \right)^{3} \mu \alpha_3^2 %Q_1^4.
\end{equation}
giving
for helium
\begin{equation}
\nu_{\text{Langmuir}\,4}\approx 6.748148600\,.
\end{equation}
Event (iv) results from the relative equilibrium associated with the collinear central configuration with the nucleus located in the middle  between the two electrons. 
For the symmetric case with particle  $m_3$ in the middle, $m_1=m_2$ and $\alpha_1=\alpha_2$, the value of $\nu$ in \eqref{eq:nu_collinear} reduces to

\begin{equation} \label{eq:nu_collinear_symmetric}
\nu_{\text{collinear}}=\frac{1}{4}m_1 (4\alpha_1+\alpha_3)^2\,.
\end{equation}
Filling in the values of helium gives
\begin{equation}
\nu_{\text{collinear}\,5}=\frac{49}{4}=12.25\,.
\end{equation}

\begin{figure}
\centerline{
\raisebox{4cm}{a)}\includegraphics[angle=0,width=4cm]{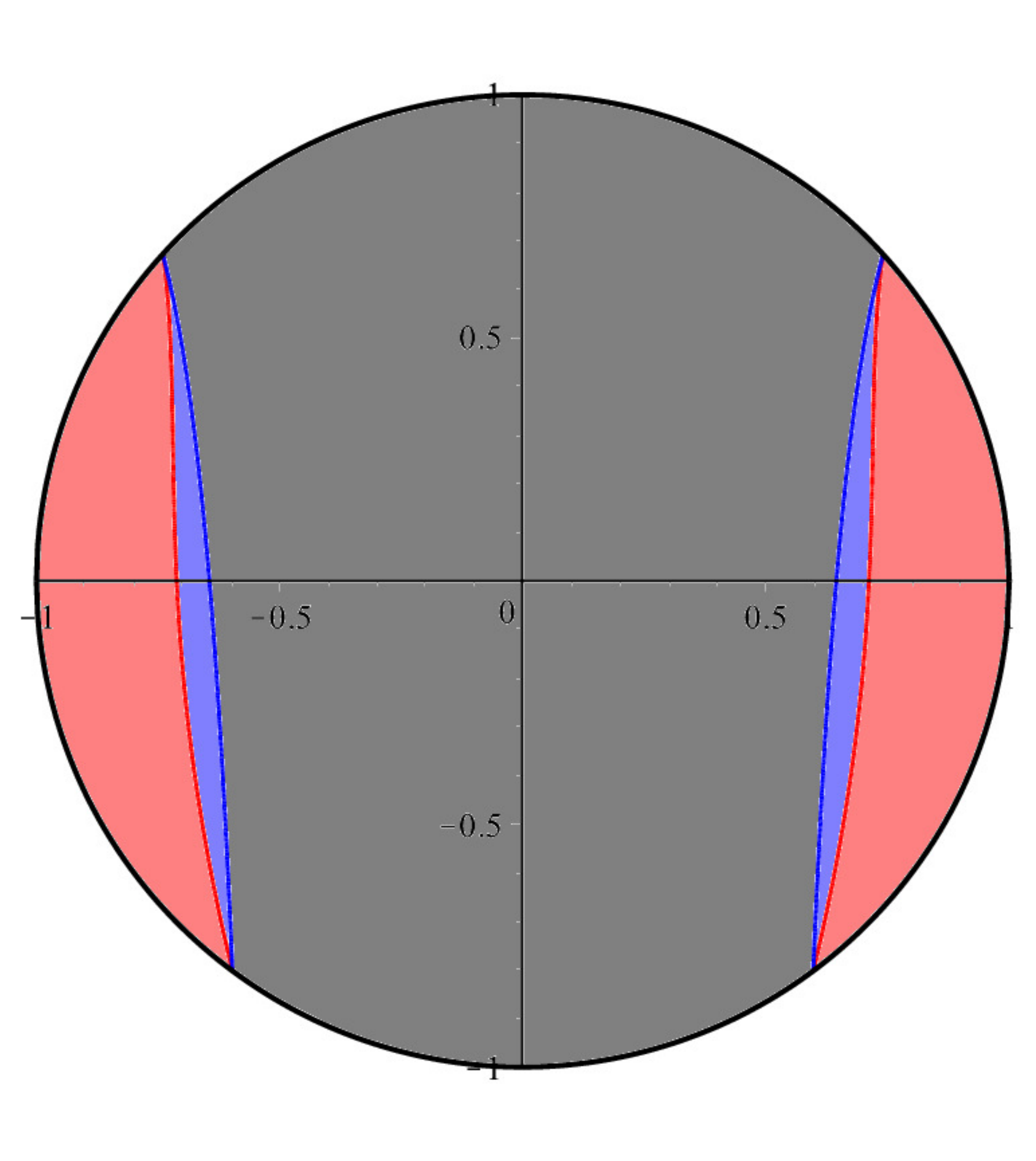} %
\raisebox{4cm}{b)}\includegraphics[angle=0,width=4cm]{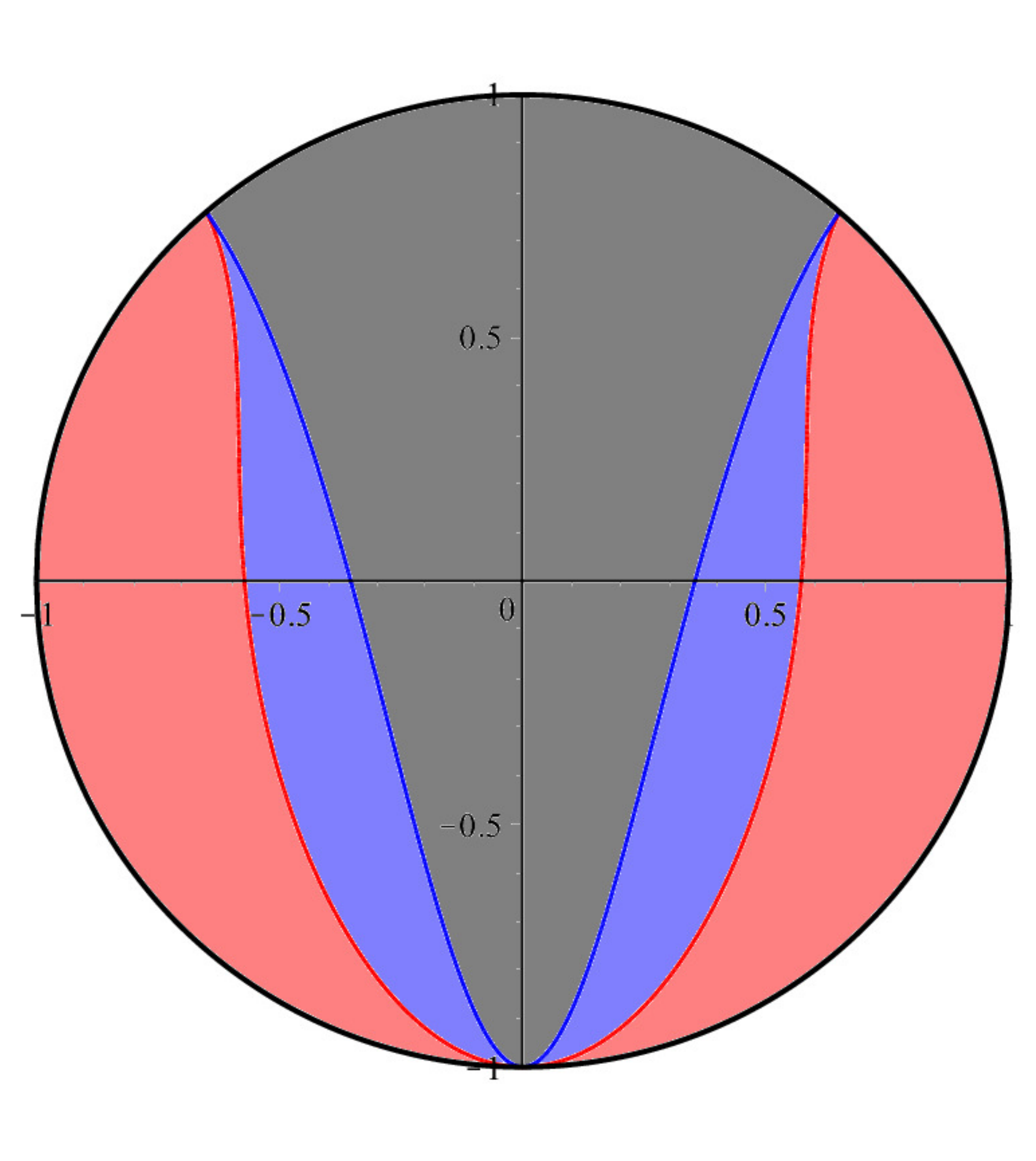} %
\raisebox{4cm}{c)}\includegraphics[angle=0,width=4cm]{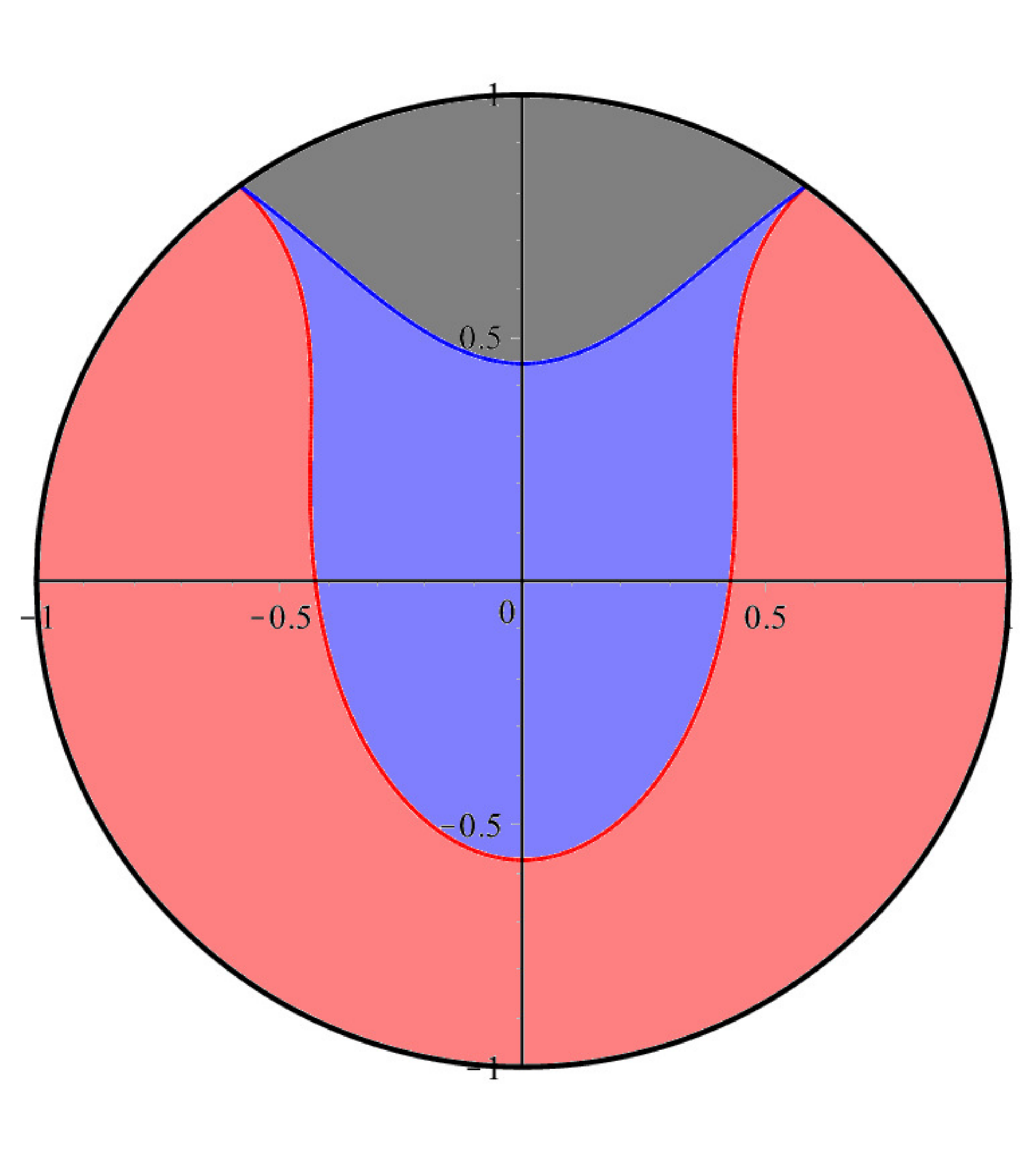} %
}
\centerline{
\raisebox{4cm}{d)}\includegraphics[angle=0,width=4cm]{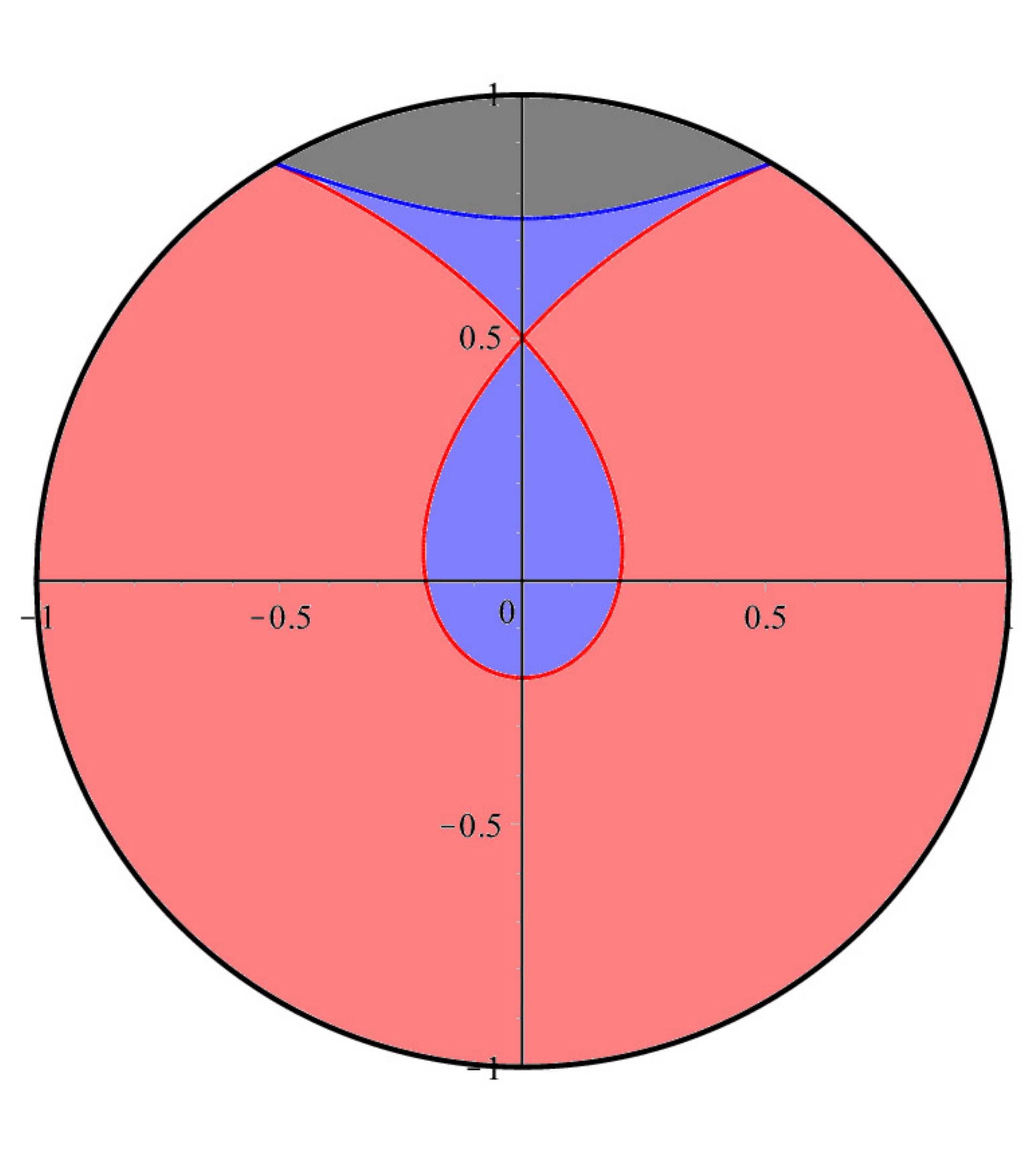} % 
\raisebox{4cm}{e)}\includegraphics[angle=0,width=4cm]{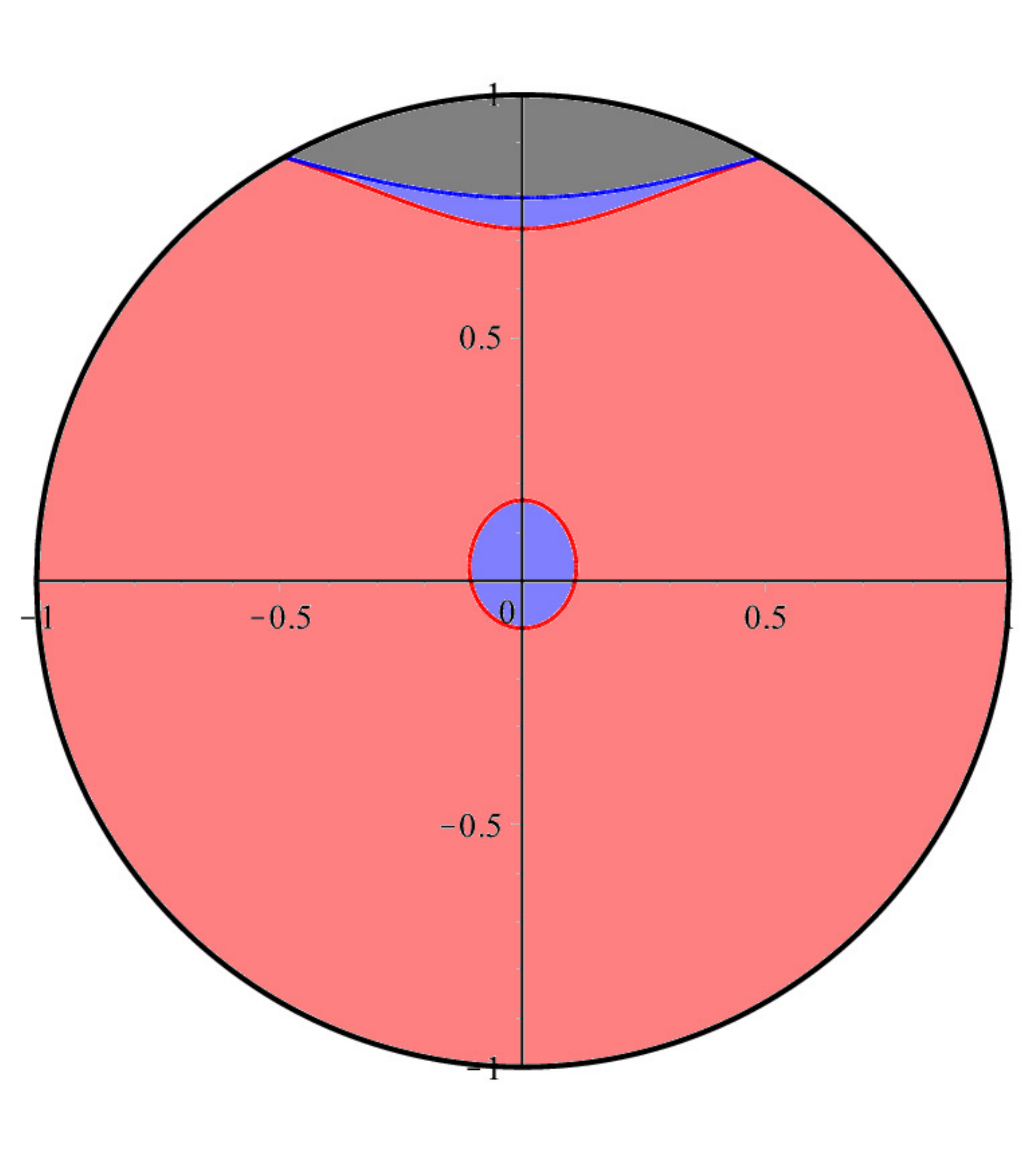} %
\raisebox{4cm}{f)}\includegraphics[angle=0,width=4cm]{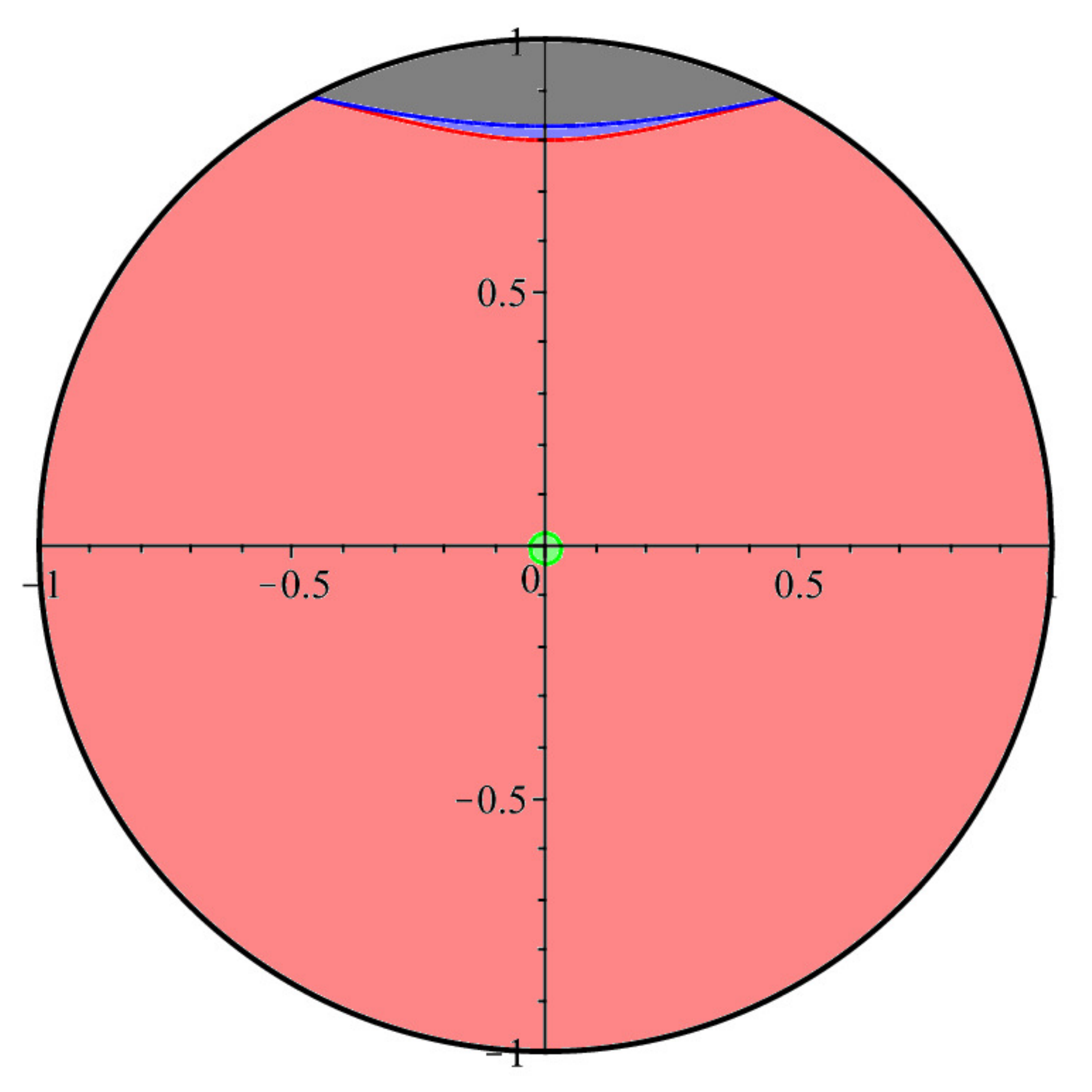}
}
\centerline{
\raisebox{4cm}{g)}\includegraphics[angle=0,width=4cm]{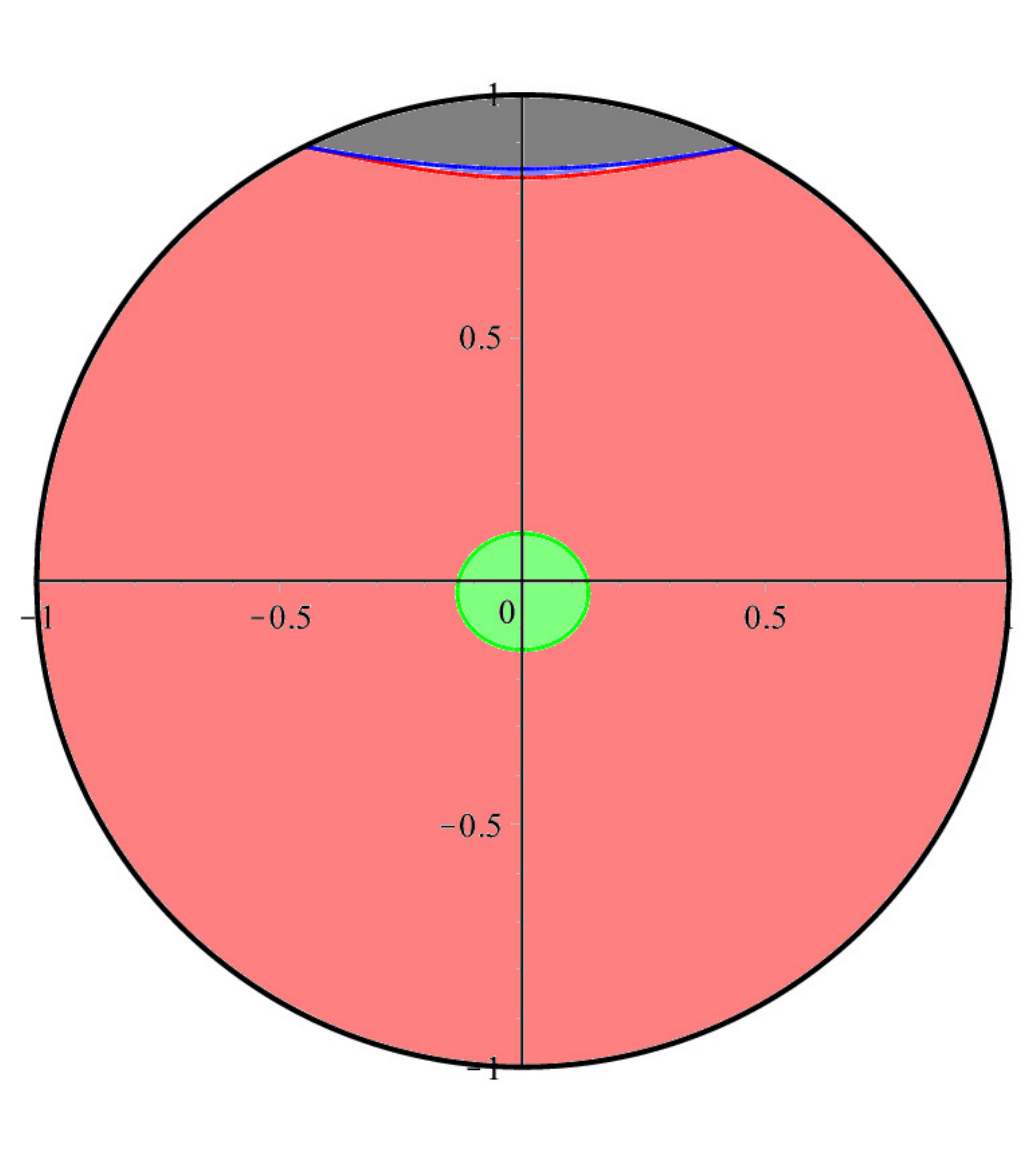}
\raisebox{4cm}{h)}\includegraphics[angle=0,width=4cm]{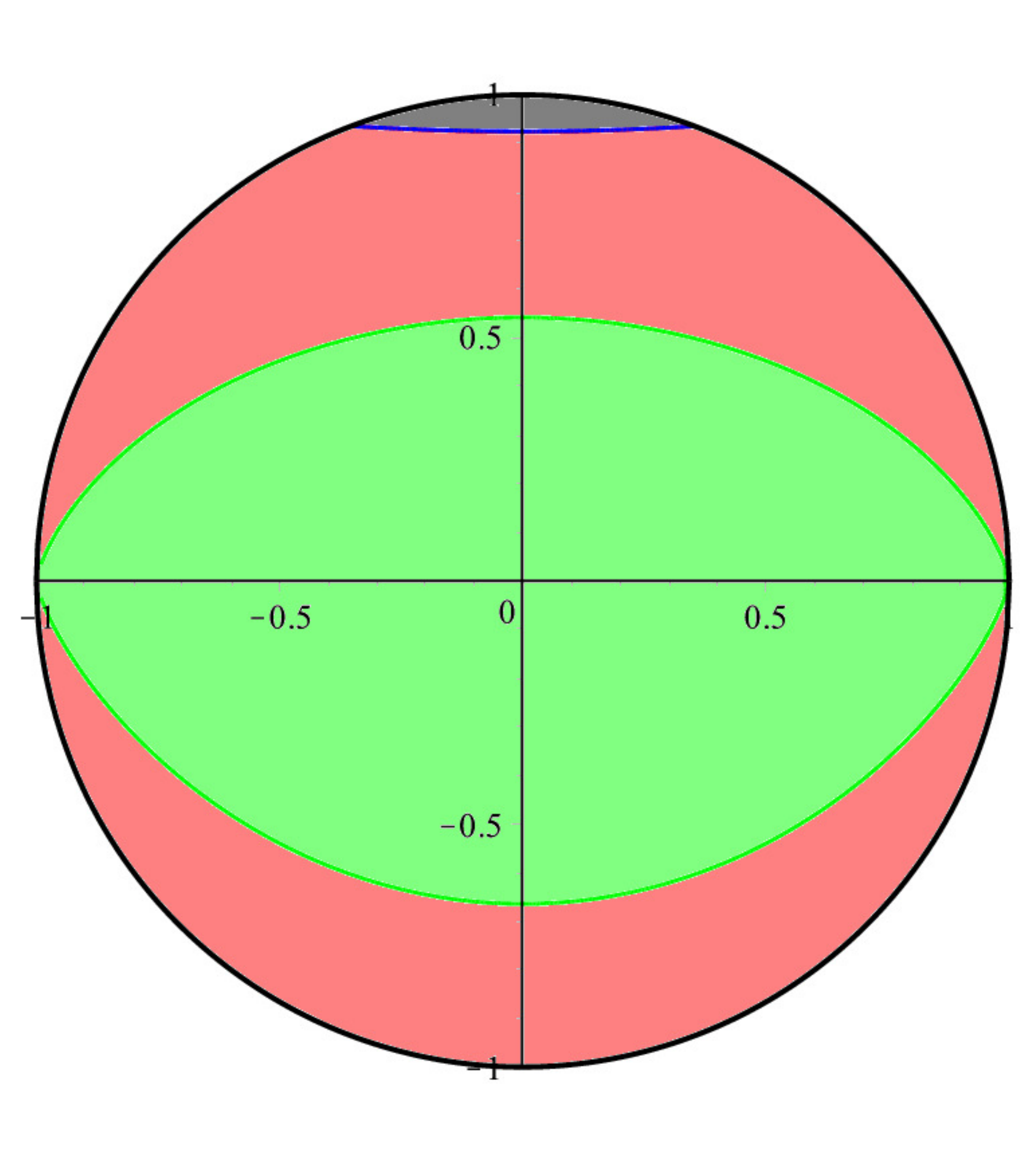} %
\raisebox{4cm}{i)}\includegraphics[angle=0,width=4cm]{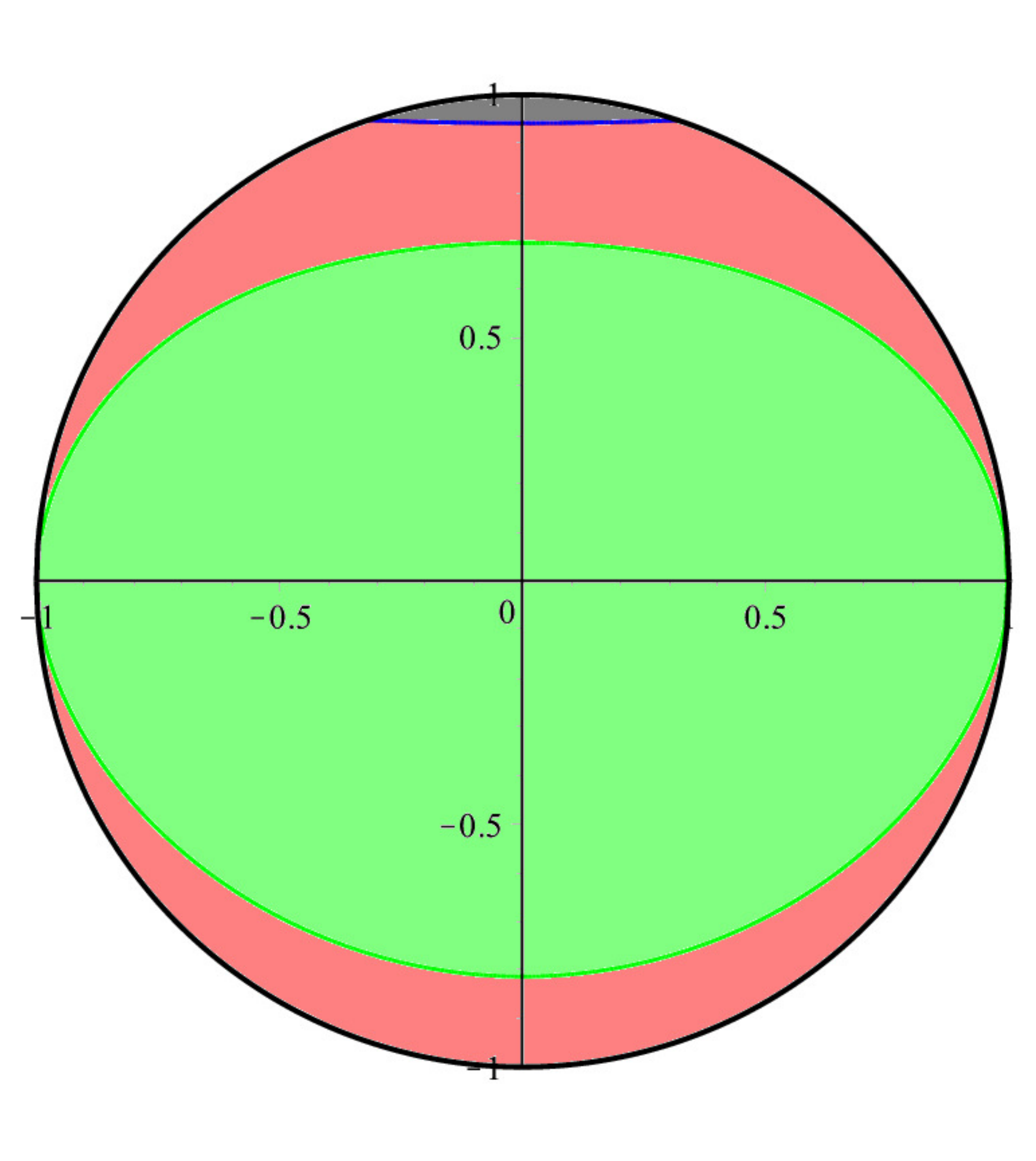} %
}
\centerline{
\raisebox{4cm}{j)}\includegraphics[angle=0,width=4cm]{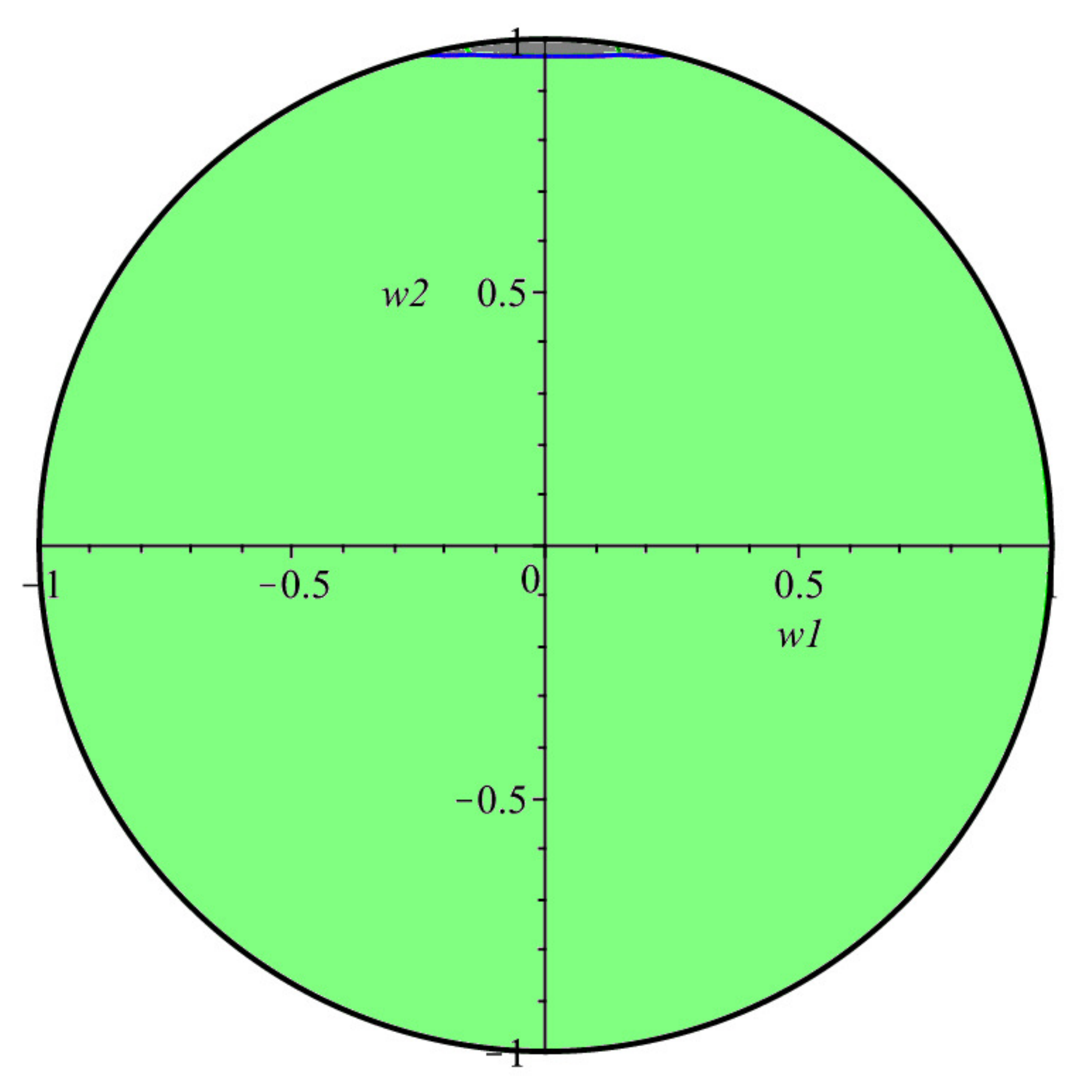} %
\raisebox{4cm}{k)}\includegraphics[angle=0,width=4cm]{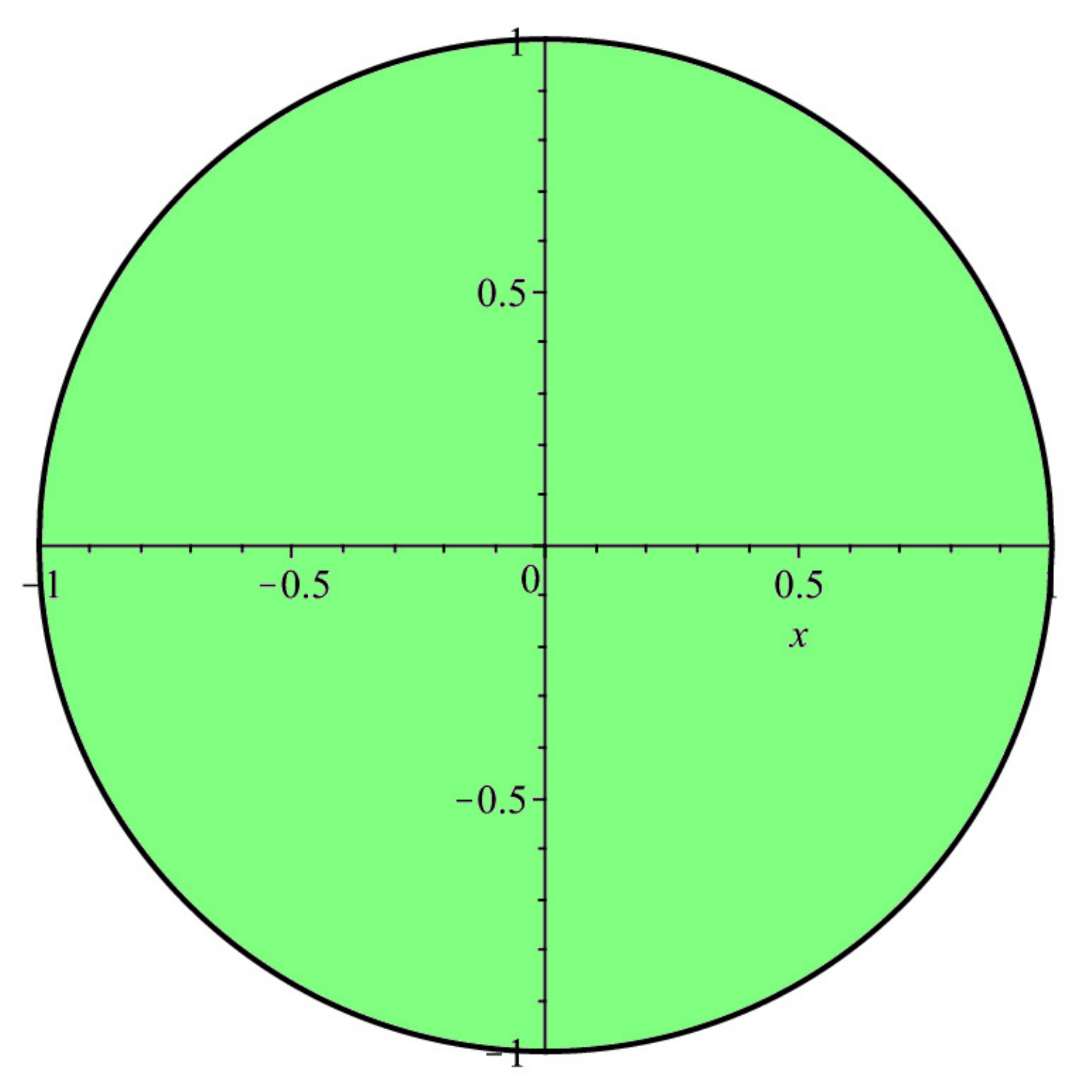}%
}

\caption{\label{charged_fig:Helium_contours}
Helium: projections of the Hill regions to the shape space
 for specific values of $\nu$: 
$\nu >  \nu_{\text{collinear}\,5}   $ (a), 
$ \nu = \nu_{\text{collinear}\,5}  $ (b), 
$\nu_{\text{collinear}\,5} >\nu> \nu_{\text{Langmuir}\,4} $ (c),  
$ \nu = \nu_{\text{Langmuir}\,4}  $ (d), 
$\nu_{\text{Langmuir}\,4}  >\nu>  \nu_{\text{diabolic}\,3} $ (e),  
$ \nu = \nu_{\text{diabolic}\,3}  $ (f), 
$\nu_{\text{diabolic}\,3}   >\nu > \nu_{\infty\,2}   < $ (g),  
$ \nu = \nu_{\infty\,2}  $ (h), 
$ \nu_{\infty\,2} > \nu >  \nu_1 =0          $ (i), 
$ \nu  = \nu_1 = 0 $ (j), and 
$\nu_1 = 0\ge  \nu   $.
The contours are shown on the shape space $\tilde{\Qtr}$ represented in the same way as in Fig.~\ref{charged_fig:ElectronsPositron_levels_general}. 
}
\end{figure}

In the different panels in Fig.~\ref{charged_fig:Helium_contours} we again show  %superpositions of the contours  of the functions $\sqrt{\tilde{M} _k} \tilde{V}$, $k=1,2,3$,  for a representative value of $\nu$ fixed in each panel. 
the projection of the Hill regions to the shape space.
The presentation and coulor code is the same as in Figs.~\ref{charged_fig:Newton_levelsa} %, \ref{charged_fig:Newton_levels_psi_chi} 
and \ref{charged_fig:Newton_levelsb} for the gravitational case. 
We again start with large values of $\nu$. 
When $\nu$ crosses the value $\nu_{\text{Langmuir}\,4}$ of the Langmuir relative equilibrium a blob detaches from the blue region.
At $\nu_{\text{diabolic}\,3}$ the blue blob shrinks to a point and similar to the analogous scenario in the gravitational case a green region grows out of this region when $\nu$ decreases below $\nu_{\text{diabolic}\,3}$. The green region then grows and  reaches the boundary of the shape space for the first time at the two double collision points of the nucleus with either of the electrons when $\nu=\nu_{\infty\,2} $.
When $\nu$ decreases further to $\nu_1=0$ the green region grows until the red and blue regions have vanished at $\nu_1=0$. 
A grey forbidden region near the double collision of the electrons  remains at $\nu_1=0$. This is due to the potential energy going to plus infinity at the double collision of the two electrons. This is different from the gravitational case where the forces between the bodies is always attractive. 
However, like in any charged three-body system for every shape, any orientation is possible for $\nu<\nu_1=0$ (i.e. for positive total energies $E$), see Theorem~\ref{charged_thm:Hillregionfirst}.

%%%%%%%%%%%%%%%%%%%%%%

%%%%%%%%%%%%%%%%%%%%%%

\subsection{Compound of two electrons and one positron}
\label{charged_sec:electronpositronelectron}

In the following we consider a system of two electrons e$^-$  and one positron e$^+$ interacting via Coulomb forces. In atomic units, 
 e$^-$ has charge -1 and mass 1. The positron e$^+$  has the same mass but opposite charge. The two electrons are labeled 1 and 2. The positron is assigned the label 3. The coefficients in the potential \eqref{eq:potential3b} are  then $\alpha_1=\alpha_2=1$ and $\alpha_3=-1$.
 
 \begin{figure}
\begin{center}
\raisebox{4cm}{a)}
\includegraphics[angle=0,width=4cm]{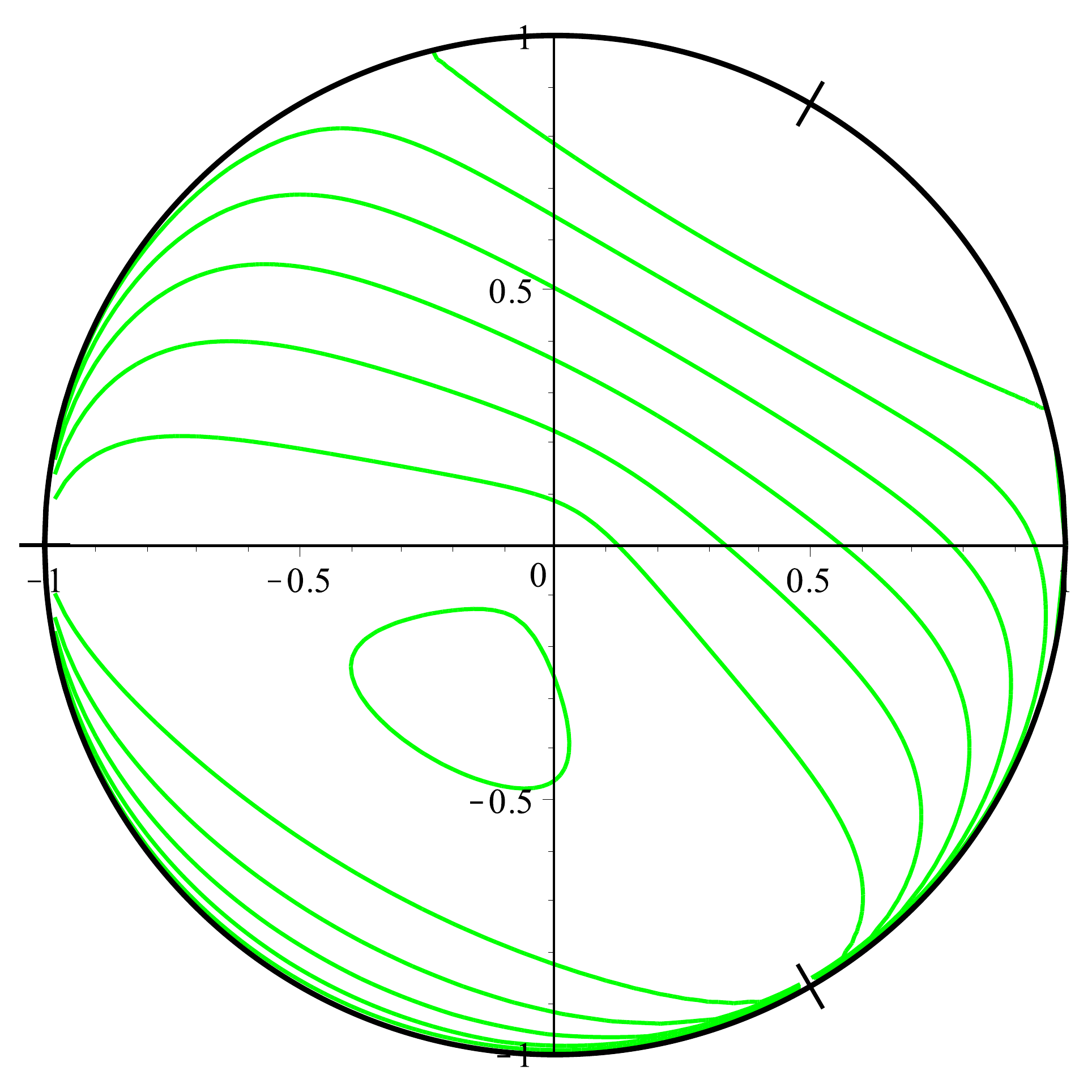}
\raisebox{4cm}{b)}
\includegraphics[angle=0,width=4cm]{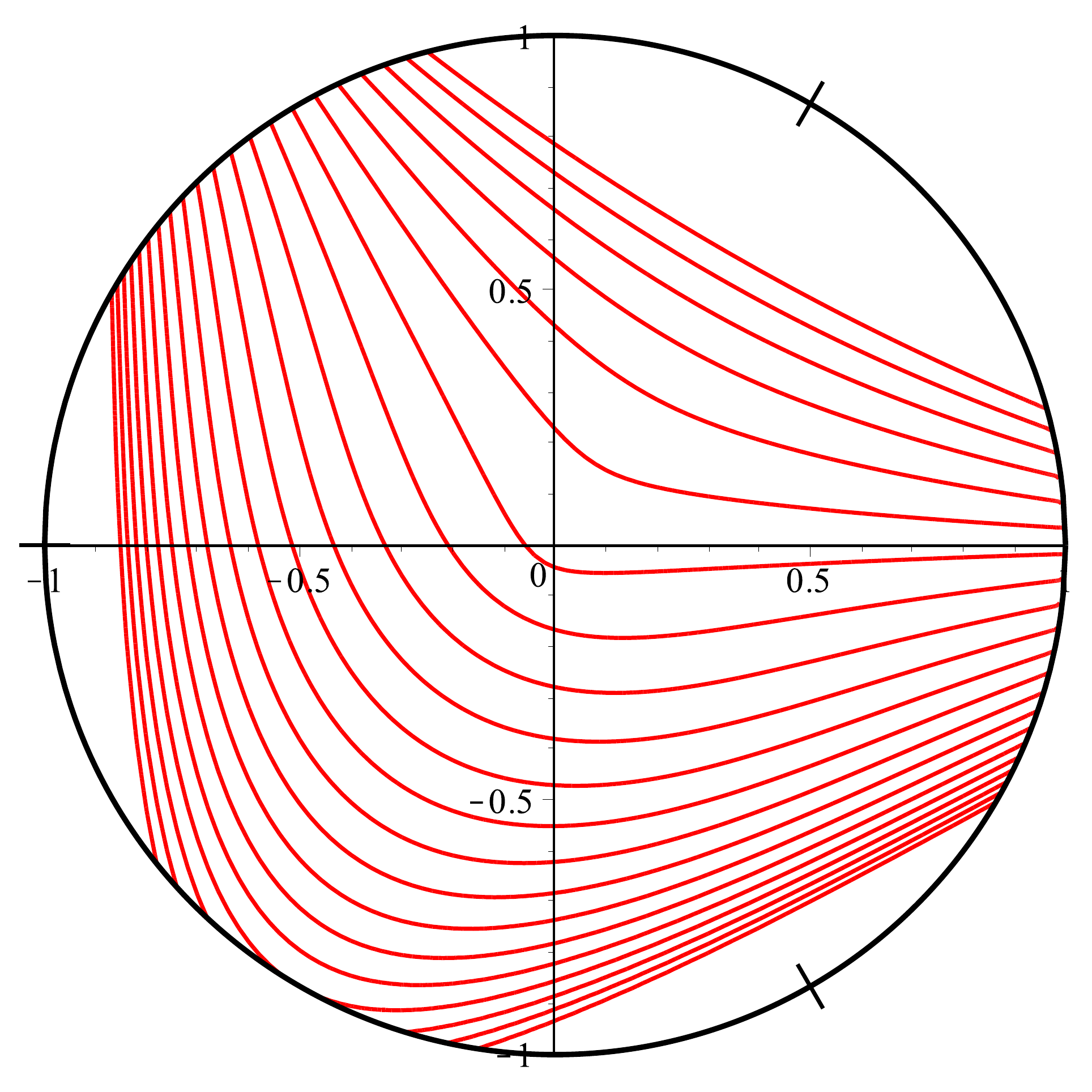}
\raisebox{4cm}{c)}
\includegraphics[angle=0,width=4cm]{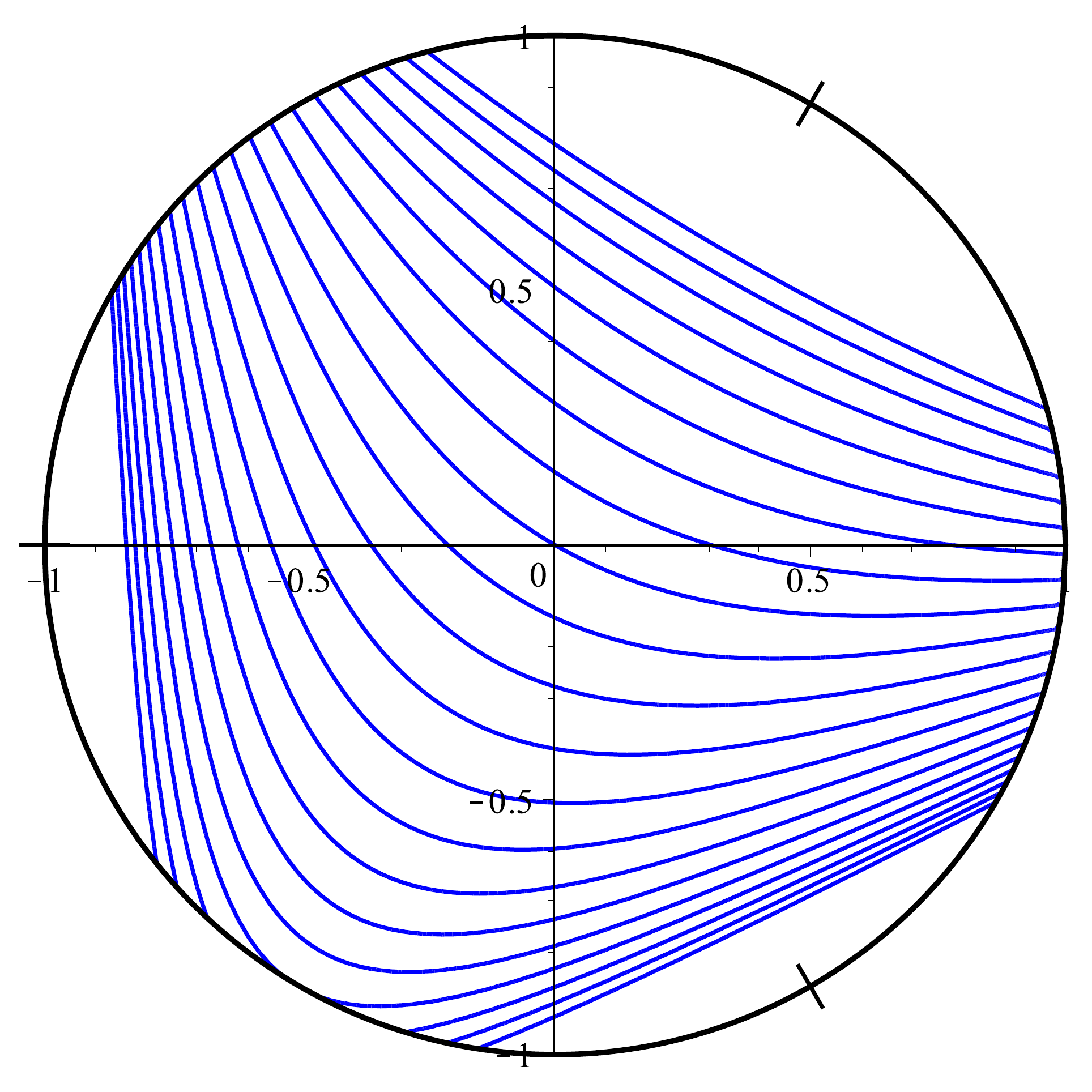}
\end{center}
\caption{\label{charged_fig:ElectronsPositron_levels_general}
Compound of two electrons and one positron: contours  $-n/8$,  $n=1,\ldots,20$, of the functions 
$  \sqrt{\tilde{M}_k  } \, \tilde{V} $ 
defined in \eqref{charged_eq:def_VtildesqrtMk} with $k=1$ (a),  $k=2$ (b) and $k=3$ (c). The presentation is analogously to Fig.~\ref{charged_fig:Newton_levels}.
}
\end{figure}

In Fig.~\ref{charged_fig:ElectronsPositron_levels_general} we show the contours of the functions  $ \sqrt{\tilde{M}_k  } \, \tilde{V} $, $k=1,2,3$, defined in \eqref{charged_eq:def_VtildesqrtMk}  analogously to Fig.~\ref{charged_fig:Newton_levels} for the gravitational case.
For equal masses, the double collision of particles 1 and 2 (two electrons) occurs at a polar angle of $60^\circ$ and of particles 2 and 3 (electron and positron) at angle $-60^\circ$ (see \eqref{charged_eq:angle_collision_12} and \eqref{charged_eq:angle_collision_23}). The collision of particles 1 and 3 (again electron and positron) is located  at $180^\circ$.
The potential $\tilde{V}$ is $-\infty$ at the collisions of either of  the  electrons with the positron (polar angles $-60^\circ$ and $180^\circ$) and  
$+\infty$ at the collisions of the two electrons  (polar angle $60^\circ$).  

From the contours in Fig.~\ref{charged_fig:ElectronsPositron_levels_general}  we conclude that together with the critical value $\nu_1=0$
there are four critical values for $\nu$ caused by the following events.

\begin{itemize}
\item[(i)]  The contours of $\sqrt{\tilde{M}_1}\tilde{V}$ touch/detach from the boundary of the shape space (see Fig.~\ref{charged_fig:ElectronsPositron_levels_general}(a)). This happens simultaneously at both of the symmetric double collisions points of the electrons with the positron.

\item[(ii)] $\sqrt{\tilde{M}_1}\tilde{V}$ has a critical point in the interior of the the shape space (see Fig.~\ref{charged_fig:ElectronsPositron_levels_general}(a)). 

\item[(iii)] The contours of $\sqrt{\tilde{M}_2}\tilde{V}$ and $\sqrt{\tilde{M}_3}\tilde{V}$ touch/detach from the boundary of the shape space  at polar angle $-120^\circ$ (see Figs.~\ref{charged_fig:ElectronsPositron_levels_general}(b) and (c)). This happens for the same value of $\nu$ as $\sqrt{\tilde{M}_2}\tilde{V}$ and $\sqrt{\tilde{M}_3}\tilde{V}$ agree on the boundary of the shape space (i.e. for collinear configurations).

\end{itemize}

As discussed for the previous two examples % the gravitational case we expect 
event (i) can be related to critical points at infinity. Like in the case of helium there are two symmetry related critical points at infinity corresponding to a co-rotating electron positron pair and a resting second electron with an infinite distance in between. The corresponding value of the bifurcation parameter is

\rem{
 The critical points at infinity result from two charges  $Q_1$ and  $Q_2$ of opposite sign and masses $m_1$ and $m_2$ co-rotating infinitely far apart from the third particle. The resulting value for $\nu$ is easily calculated to give

\begin{equation} \label{charged_eq:nu_for_corotating_charges}
\nu=\frac12 \mu Q_1^2 Q_2^2, 
\end{equation}
where $\mu=m_1m_2/(m_1+m_2)$ is the reduced mass. 
Filling in $Q_1=1$ and $Q_2=+1$ and $m_1 = m_2 =1$ gives
} % end rem

%{\bf critical point at infinity and diabolic point}
\begin{equation}
\nu_{\infty \, 2} =  \frac{1}{4} =0.25\,.
\end{equation}

As a result of the equal masses and absolute values of the charges in  the electron-electron-positron compound, it happens that the value of the bifurcation parameter corresponding to the diabolic point of the moment of inertia tensor gives the same value, i.e. 
\begin{equation}
 \nu_{\text{diabolic}\, 2} = \nu_{\infty \, 2} .
\end{equation}

The event (ii) results from a non-collinear relative equilibrium involving rotation about the first principal axis. This relative equilibrium is again a Langmuir orbit and the  corresponding value of $\nu$ can be calculated from \eqref{eq:app_nu_general} to be

\begin{equation}
\nu_{\text{Langmuir}\,3} \approx 0.2925594730\,.
\end{equation}

Similarly to the helium case, event (iii)  can be related to the collinear central configuration where the positron is located right between the two electrons on a line. 
The corresponding value of $\nu$ can be calculated from the symmetric version of \eqref{eq:nu_collinear} in \eqref{eq:nu_collinear_symmetric} to give

\begin{equation}
\nu_{\text{collinear}\,4}=\frac{9}{4}=2.25.
\end{equation}

 \begin{figure}
\begin{center}
\raisebox{4cm}{a)}
\includegraphics[angle=0,width=4cm]{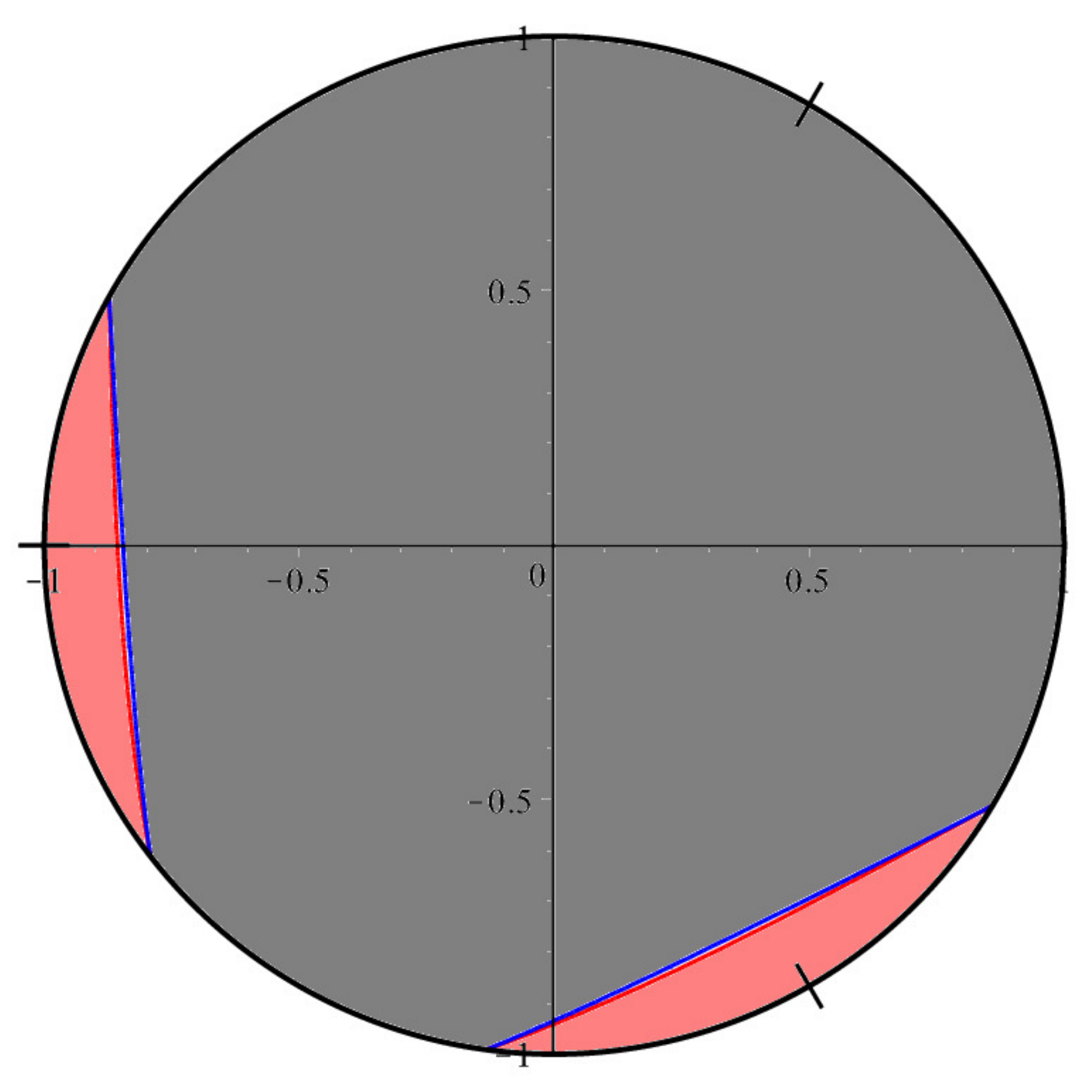}
\raisebox{4cm}{b)}
\includegraphics[angle=0,width=4cm]{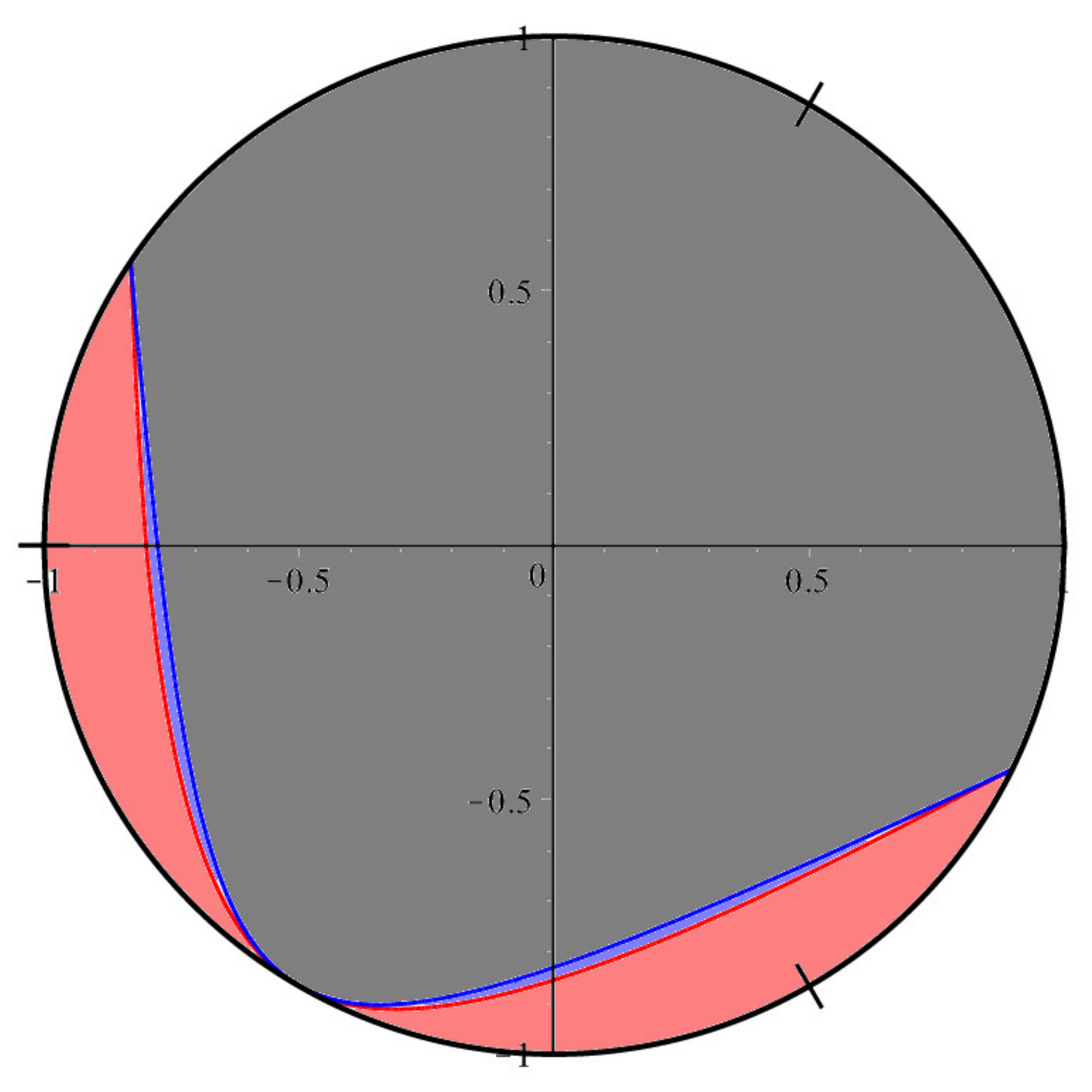} 
\raisebox{4cm}{c)}
\includegraphics[angle=0,width=4cm]{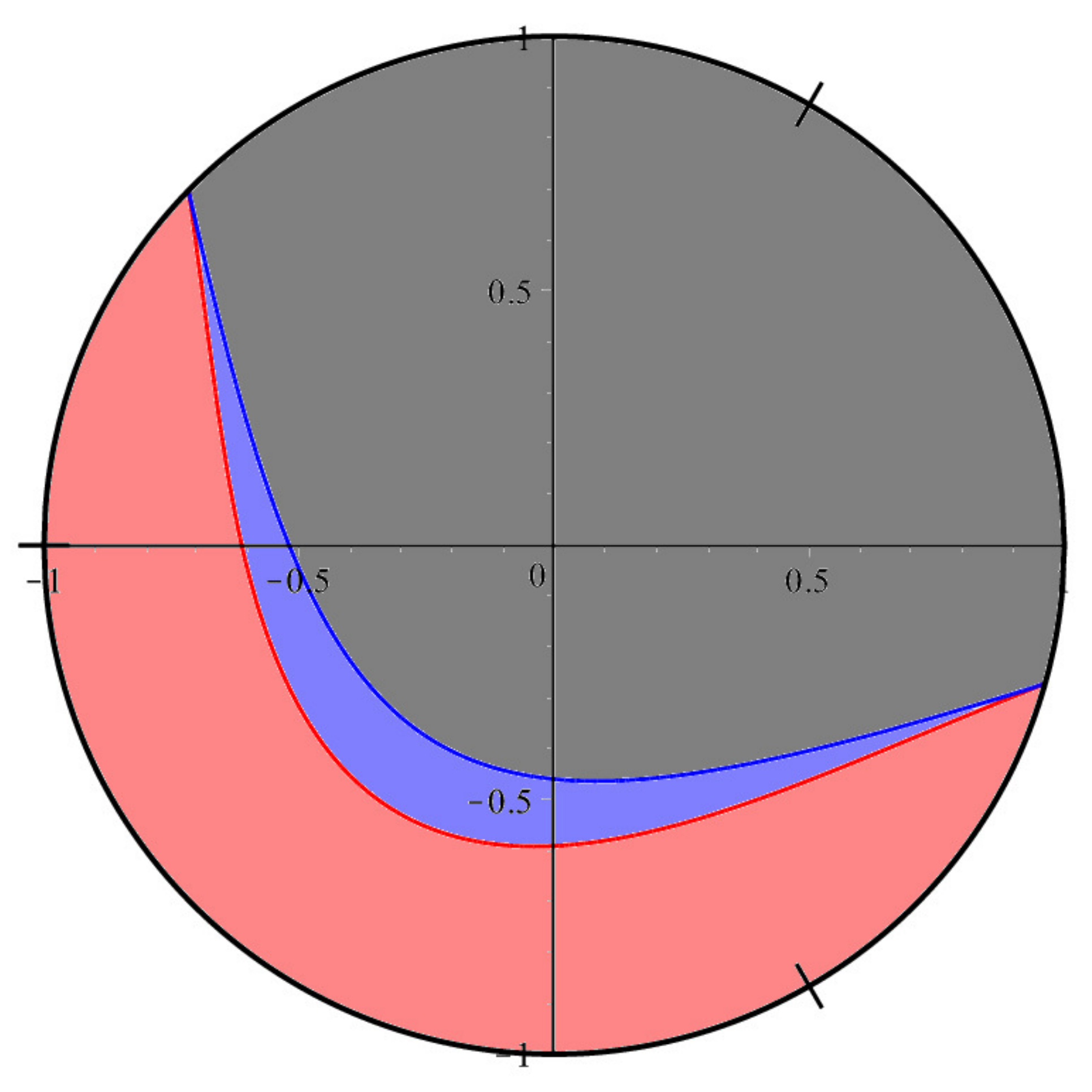}\\
\raisebox{4cm}{d)}
\includegraphics[angle=0,width=4cm]{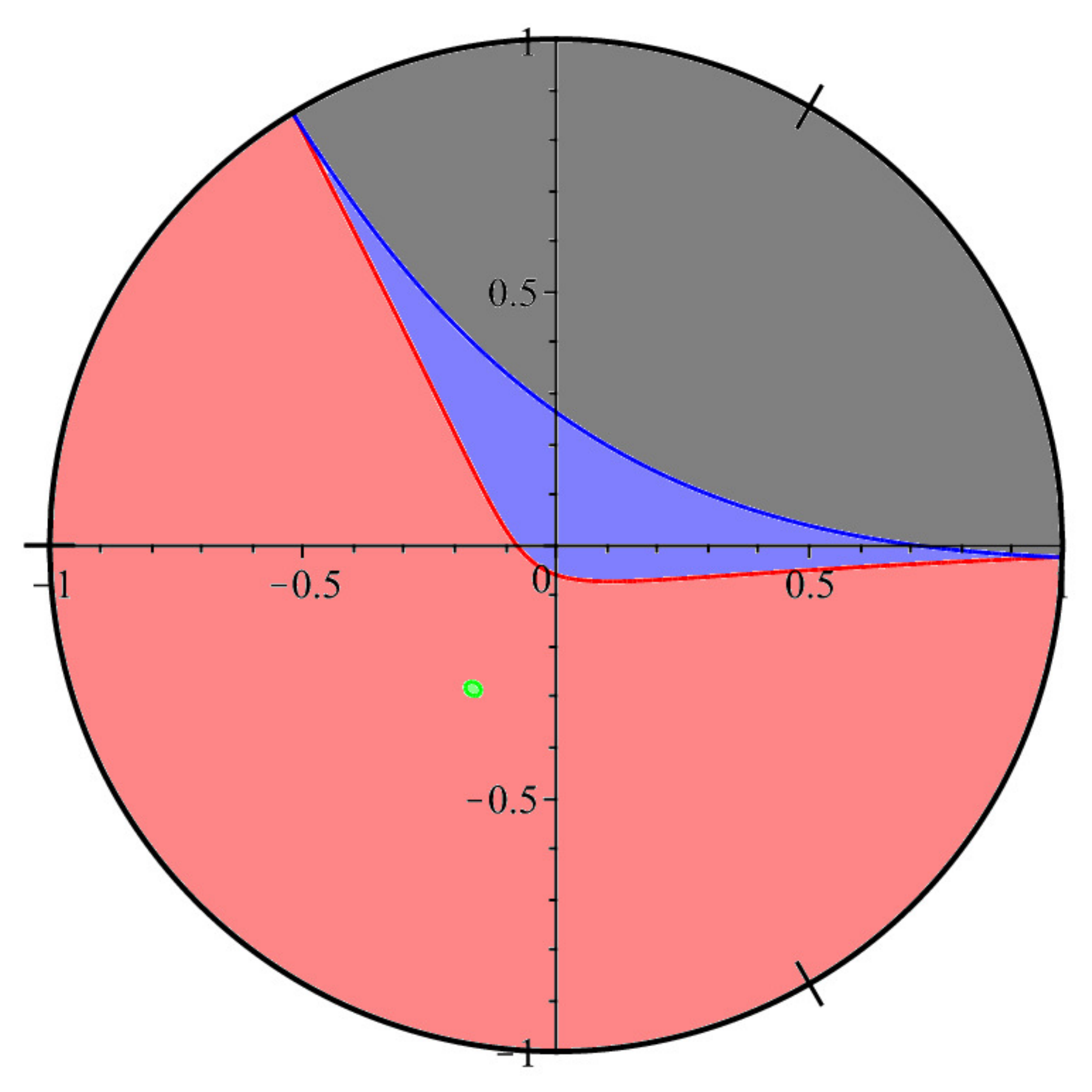}
\raisebox{4cm}{e)}
\includegraphics[angle=0,width=4cm]{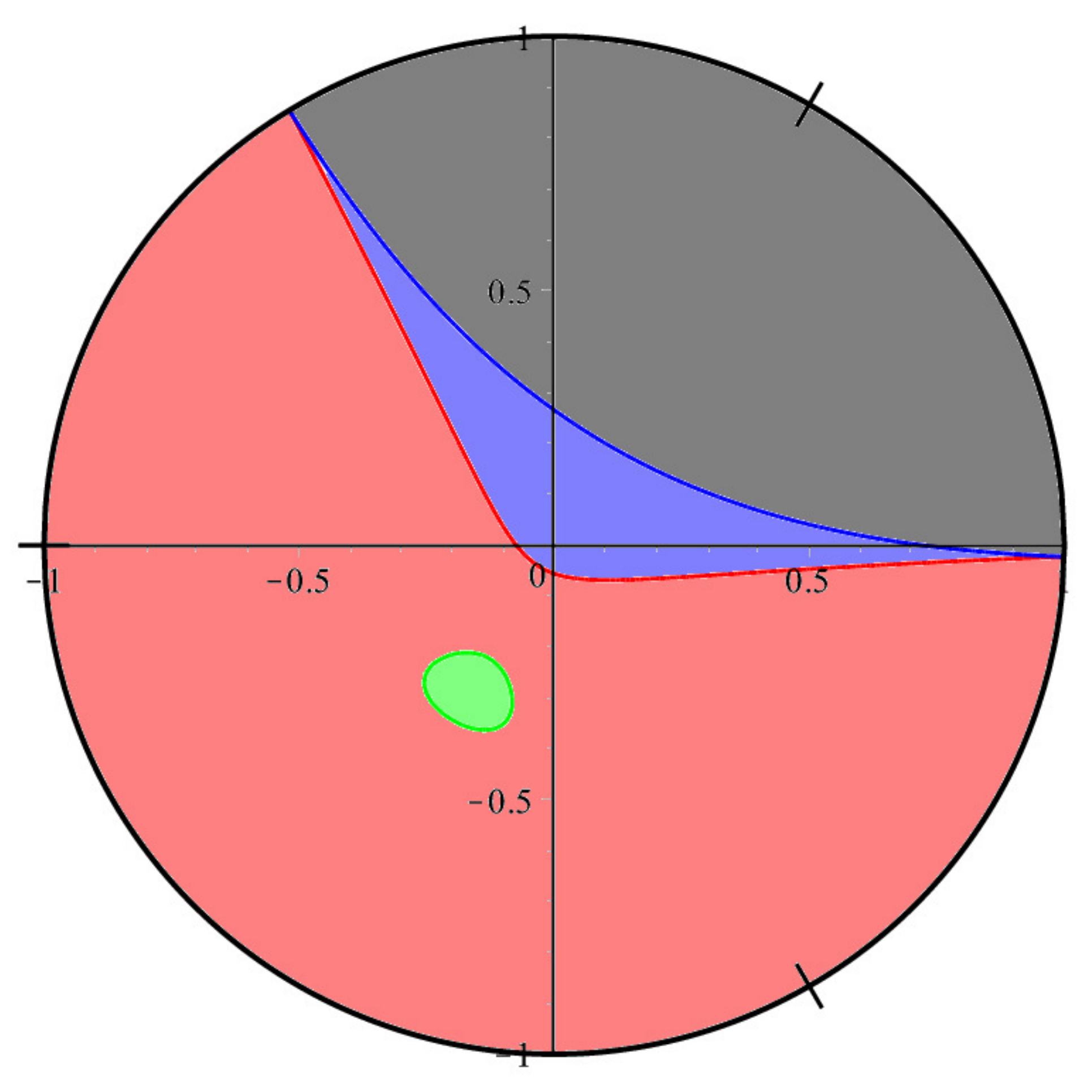}
\raisebox{4cm}{f)}
\includegraphics[angle=0,width=4cm]{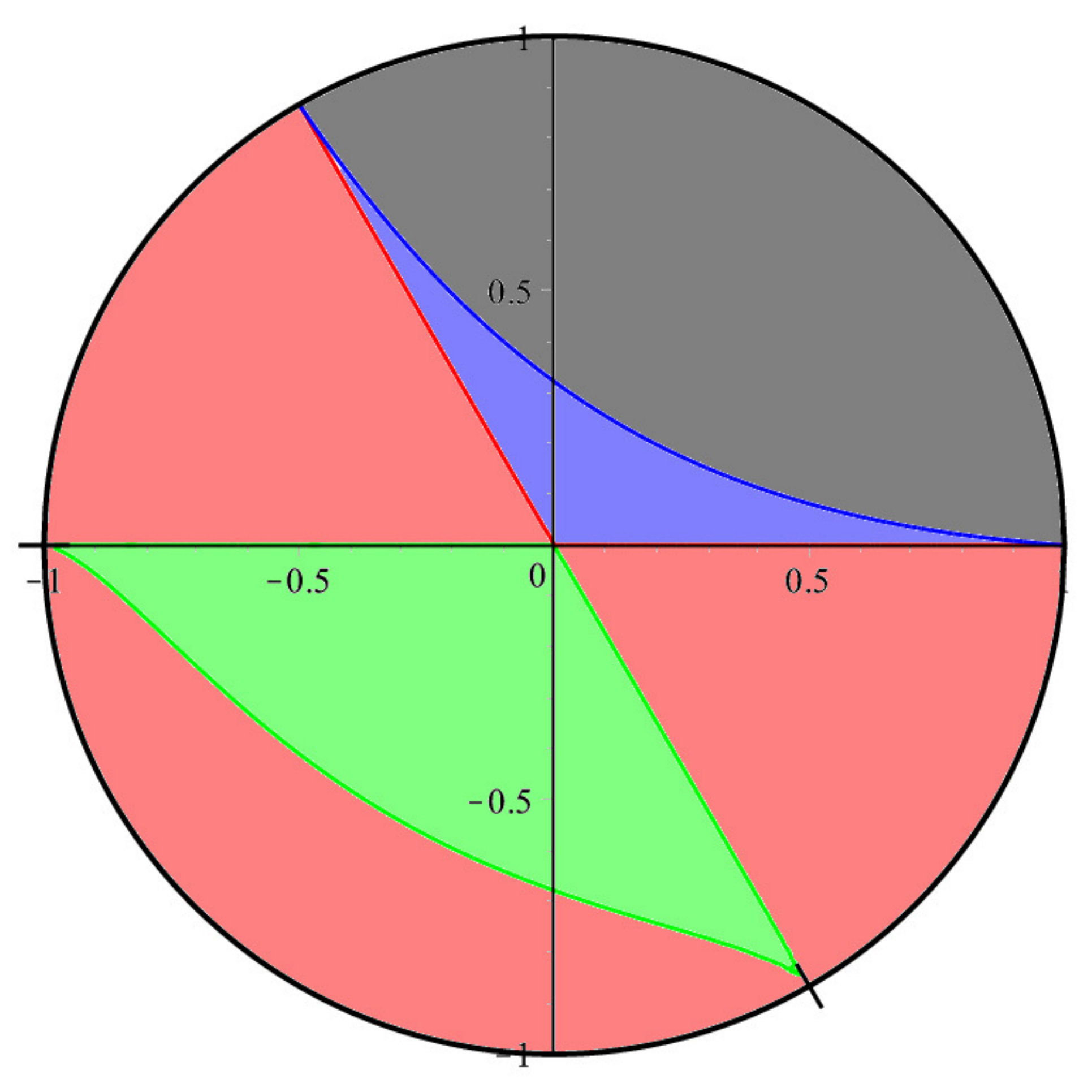}\\
\raisebox{4cm}{g)}
\includegraphics[angle=0,width=4cm]{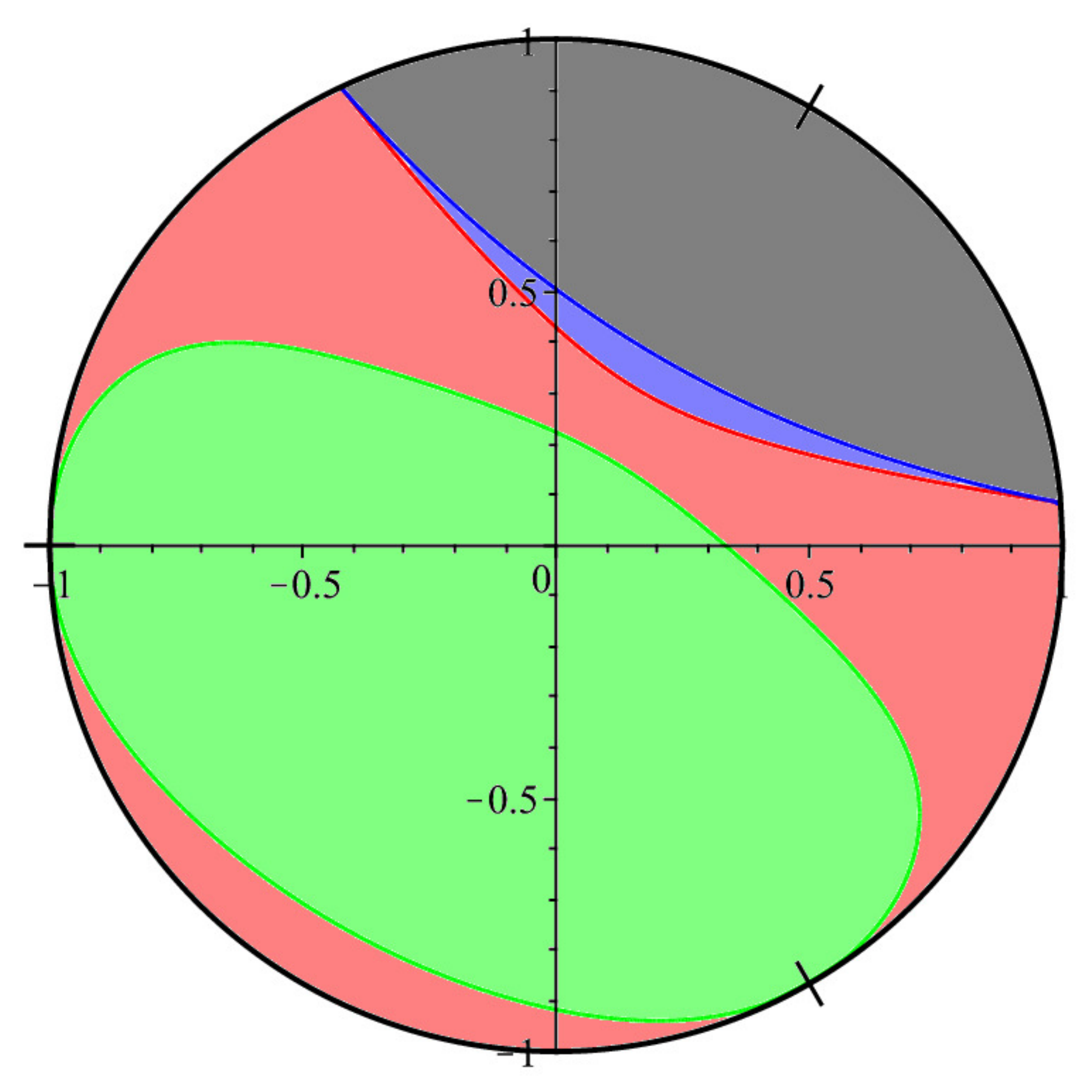}
\raisebox{4cm}{h)}
\includegraphics[angle=0,width=4cm]{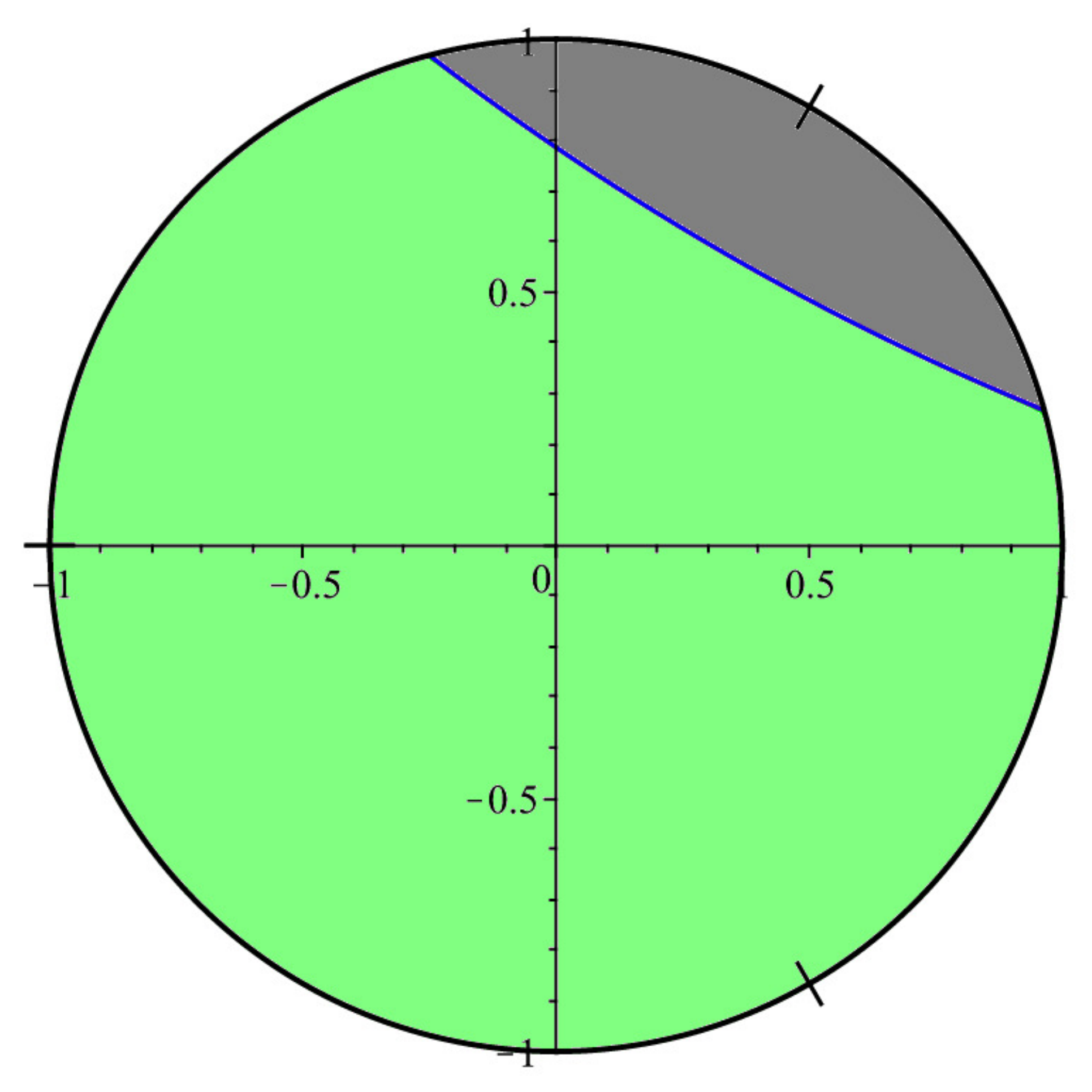} 
\raisebox{4cm}{i)}
\includegraphics[angle=0,width=4cm]{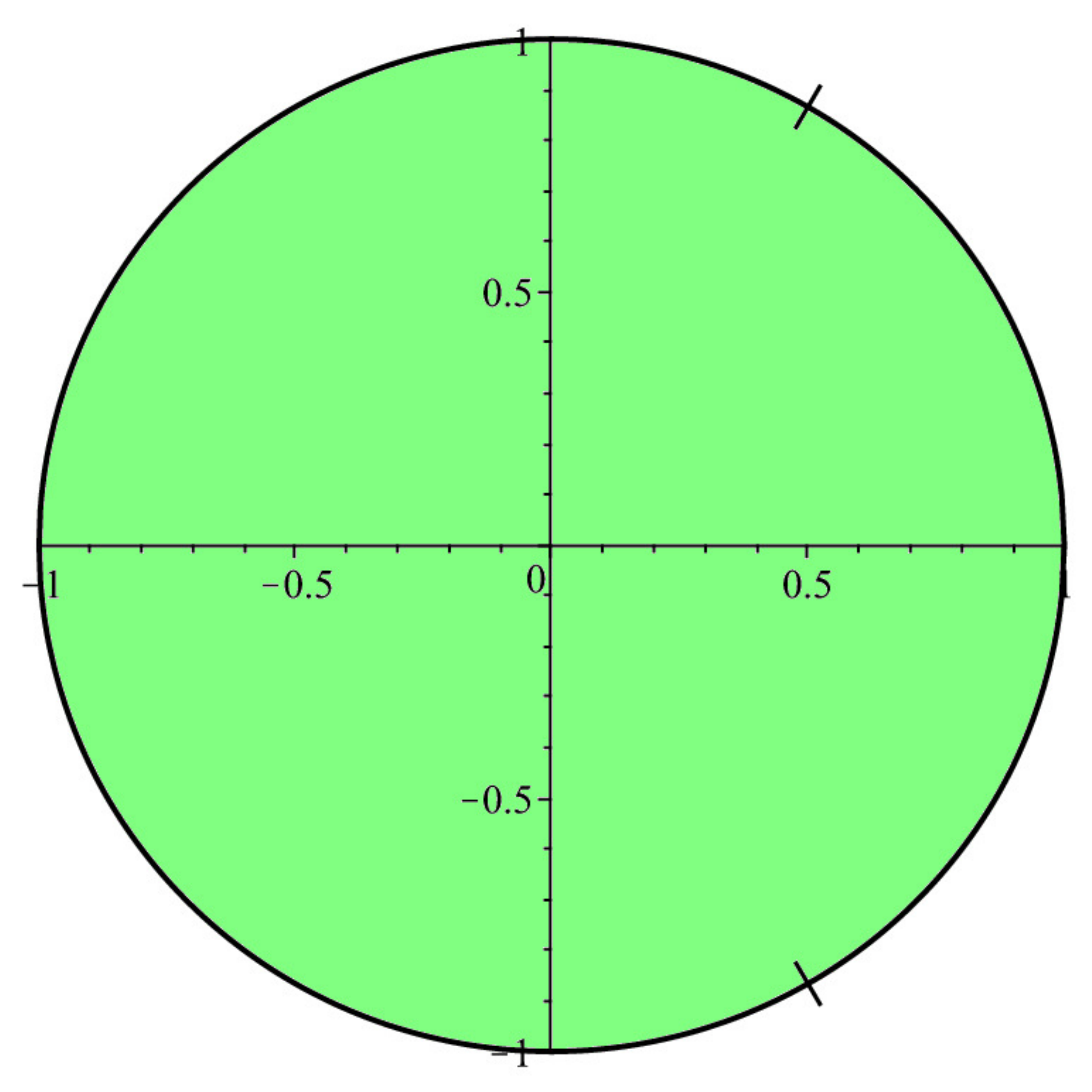}
\end{center}
\caption{\label{charged_fig:ElectronsPositron_levels}
Compound of two electrons and one positron: projections of the Hill regions to the shape space for specific values of $\nu$: 
$ \nu  >  \nu_{\text{collinear}\,4}   $ (a), 
$ \nu = \nu_{\text{collinear}\,4}  $ (b), 
$ \nu_{\text{collinear}\,4} > \nu >  \nu_{\text{Langmuir}\,3}   $ (c),  
$ \nu = \nu_{\text{Langmuir}\,3}  $ (d), 
$ \nu_{\text{Langmuir}\,3}  > \nu >  \nu_{\infty \, 2} = \nu_{\text{diabolic}\, 2}   $ (e),  
$ \nu = \nu_{\infty \, 2} = \nu_{\text{diabolic}\, 2}   $ (f), 
$ \nu_{\infty \, 2} =\nu_{\text{diabolic}\, 2}  > \nu >  \nu_1 =0$ (g),  
$ \nu = \nu_1  $ (h),  and
$ \nu_1 = 0>\nu $ (i).  
The contours are shown on the shape space $\tilde{\Qtr}$ represented in the same way as in Fig.~\ref{charged_fig:ElectronsPositron_levels_general}. 
}
\end{figure}

In the different panels in Fig.~\ref{charged_fig:ElectronsPositron_levels} we again show the projections of the Hill regions to the shape space
 for representative values of $\nu$.
% The regions enclosed 
%are coloured according to what the accessible region on the orientation sphere is.  
The presentation and colour code is again the same as in Figs.~\ref{charged_fig:Newton_levelsa} %, \ref{charged_fig:Newton_levels_psi_chi} 
and \ref{charged_fig:Newton_levelsb} for the gravitational case. 
% h
Starting with large values of $\nu$ there is for $\nu> \nu_{\text{collinear}\,4}$  a grey region not belonging to the Hill region that separates the accessible part of the shape space into two disconnected components. For points in the two red components,  the accessible region on the orientation sphere is a ring. These two components contain on their boundary the double collisions where one of the electrons collides with the positron. The forbidden region has the double collision of the electrons as a boundary point. Between the grey region and the red regions there are blue strips where the accessible region on the orientation sphere consists of two caps. At $\nu= \nu_{\text{collinear}\,4}$ the boundary of the grey regions and the blue region are tangental to the boundary of the shape space at the point corresponding to the collinear central configuration which has polar angle $-120^\circ$ in the $(w_1,w_2)$-plane. For $\nu<\nu_{\text{collinear}\,4}$ the Hill region is simply connected. For $\nu_{\text{Langmuir}\,3}<\nu<\nu_{\text{collinear}\,4}$, the blue and the red regions are both simply connected. At $\nu=\nu_3$ a green dot emerges in the red region because of which the red region is no longer simply connected for $\nu<\nu_{\text{Langmuir}\,3}$. This green dot results from the 
non-collinear relative equilibrium given by the Langmuir orbit which corresponds to rotation about the first principal axis. When $\nu$ decreases from $\nu_{\text{Langmuir}\,3}$ to $  \nu_{\text{diabolic}\, 2} =\nu_{\infty \, 2}$ the green region grows until it touches the boundary of the shape space for the first time at the two electron-electron double collision points
for $\nu= \nu_{\text{diabolic}\, 2} = \nu_{\infty \, 2} $. Whereas this is due to the two symmetry related critical points at infinity, simultaneously for the same value of $\nu$, the green and the blue region touch at the center of the unit disk in the $(w_1,w_2)$-plane where the smallest and the middle principal moments of inertia are equal. At the point where they touch the boundaries of the green and blue regions are not smooth.
When $\nu$ decreases below $ \nu_{\text{diabolic}\, 2} = \nu_{\infty \, 2} $ the green and the blue region detach again and the corners on their boundaries vanish again.
This scenario is different from the scenarios resulting from the diabolic point in the previous two examples. In the previous two examples we explained the scenarios from noticing that the graphs of the functions \eqref{charged_eq:def_VtildesqrtMk} over the shape space give for  $k=1$ and $k=2$ a distorted version of the double cone in Fig.~\ref{charged_fig:principal_moments}.  
In the present case the double cone is distorted in such a way that the boundaries of the sublevel sets have the topology of the  intersections of the double cone with vertical planes in Fig.~\ref{charged_fig:principal_moments} (as opposed to horizontal planes like in the gravitational case and in the helium atom). 
When the value of $\nu$ passes through $\nu_{\text{diabolic}\,5}$ the sublevel sets of the functions  \eqref{charged_eq:def_VtildesqrtMk} then lead to the bifurcation shown in Figs.~\ref{charged_fig:ElectronsPositron_levels}e-\ref{charged_fig:ElectronsPositron_levels}g.
Upon $\nu$ approaching $\nu_1=0$ from above
the red and the blue regions are shrinking due to  the growing green region until at $\nu=\nu_1=0$ both the red region and the blue regions have vanished. Similar to the case of the helium atom and unlike the gravitational case there remains a forbidden region near the point of the electron-electron collision at polar angle $60^\circ$. For $\nu<\nu_1=0$, the full shape space belongs to the Hill region and, as in every charged three-body system, for every shape, every orientation is possible (see Theorem~\ref{charged_thm:Hillregionfirst}).

%%%%%%%%%%%%%%%%%%%%%

\section{Discussion and outlook}
\label{sec:conclusions}

The study of the Hill region of the gravitational three-body problem has a long history. It has been, e.g.,  the subject of the famous and  by now disproved  Birkhoff conjecture stating that all bifurcations of Hill regions are related to central configurations (see the discussion in \cite{mccord1998integral}). In the present paper we studied the more general case of charged three-body systems which include the gravitational system  as a special case.
As shown in \cite{HoveijnWaalkensZaman2019}, charged three-body systems can have relative equilibria which do not project to central configurations. It is therefore even more elusive to expect the Birkhoff conjecture  to hold in the case of charged three-body problems. In this paper we provided two examples of systems which have relative equilibria that do not project to central configurations. In both examples the relative equilibrium is given by the Langmuir orbit. In fact we conjecture that for charged three-body systems the only possible non-collinear relative equilibria are the Lagrange type relative equilibria familiar from the gravitational case which, as stated in this paper, are expected to exist more generally  for purely attractive inter-body forces and the Langmuir type relative equilibria for systems with coexisting attractive and repulsive inter-body forces. In a charged three-body system the Lagrange and the Langmuir type relative equilibria cannot coexist. It is also possible that a charged-three body problem has no non-collinear relative equilibria at all (e.g. if all inter-body forces are repulsive).

In \cite{mccord1998integral} the Hill regions are also used to study the topology of the integral manifolds. It would be interesting to generalise such results to the charged case. The present paper provides first steps in this direction. We  however point out again that 
the symplectic reduction used in this paper to define the Hill regions is singular for collinear configurations. A study of the topology of the integral manifolds for the charged case in a similar fashion as in \cite{mccord1998integral}  would hence require a more careful analysis of the boundary of the shape space to which the collinear configurations project. In the present paper we have given an explanation of the pseudo critical point which 
leads to the seeming bifurcation of the Hill regions observed when projecting these  to the shape space by relating this seeming bifurcation to the diabolic point of the moment of inertia tensor.

Our future studies motivated by the present paper and the preceding two papers \cite{HoveijnWaalkensZaman2019,CPI2021} concern
the stability of relative equilibria and the existence of the Langmuir type relative equilibria also for less symmetric systems than the helium atom and the electron-electron-positron compound studied in the present paper. Both of these items are of importance for the final goal  to determine the
topology of the integral manifolds for charged three-body systems in full generality.

%%%%%%%%%%%%%%%%%%%%%
%%%%%%%%%%%%%%%%%%%%%
%%%%%%%%%%%%%%%%%%%%%

%%%%%%%%%%%%%%%%%%%%%%

%\bibliographystyle{unsrt}
%\bibliography{redbib}

\end{document}